\documentclass[a4paper]{article}

\usepackage[english]{babel}
\usepackage[utf8x]{inputenc}
\usepackage[T1]{fontenc}
\usepackage[colorinlistoftodos]{todonotes}
\usepackage{tikz}
\usepackage{caption}
\usepackage{enumerate}
\usepackage[a4paper,top=3cm,bottom=2cm,left=3cm,right=3cm,marginparwidth=1.75cm]{geometry}
\usepackage{mathrsfs, amsthm}
\newtheorem{theorem}{Theorem}[section]

\newtheorem{lemma}[theorem]{Lemma}

\newtheorem{assumption}{Assumption}[section]

\newtheorem{remarkex}{Remark}[section]
\newtheorem{experiment}{Experiment}[section]
\newenvironment{remark}
  {\pushQED{\qed}\remarkex}
  {\popQED\endremarkex}
\renewenvironment{abstract}
 {\small
  \begin{center}
  \bfseries \abstractname\vspace{-0.5em}\vspace{0pt}
  \end{center}
  \list{}{%
    \setlength{\leftmargin}{7mm}
    \setlength{\rightmargin}{\leftmargin}%
  }%
  \item\relax}
 {\endlist}

\usepackage{mathtools}

\usepackage{algorithm2e}
\RestyleAlgo{ruled} 
\SetKwInput{KwInput}{Input}                
\SetKwInput{KwOutput}{Output}
\SetKwInput{KwInitialize}{Initialize} 
\usepackage{comment}
\usepackage{amsmath,amssymb}
\usepackage{graphicx}
\usepackage[colorinlistoftodos]{todonotes}
\usepackage[colorlinks=true, allcolors=blue]{hyperref}
\usepackage[english]{babel}
\usepackage{bbm}
\numberwithin{equation}{section}

\DeclareMathOperator{\diag}{Diag}

\DeclareMathOperator{\var}{Var}
\DeclareMathOperator{\M}{M}
\usepackage{newpxtext,newpxmath}

\newcommand{\Zstroke}{%
  \text{\ooalign{\hidewidth\raisebox{0.2ex}{--}\hidewidth\cr$Z$\cr}}%
}
\newcommand{\zstroke}{%
  \text{\ooalign{\hidewidth -\kern-.3em-\hidewidth\cr$z$\cr}}%
}

\begin{document}

\title{Consensus based optimization via jump-diffusion stochastic differential equations}

\author{D. Kalise\thanks{%
Department of Mathematics, Imperial College London, South Kensington Campus, SW7 2AZ London, UK; dkaliseb@imperial.ac.uk} \and A. Sharma\thanks{School of Mathematical Sciences, University of Nottingham,  UK; Akash.Sharma1@nottingham.ac.uk} \and M.V. Tretyakov\thanks{%
School of Mathematical Sciences, University of Nottingham,
 UK; Michael.Tretyakov@nottingham.ac.uk}}

\date{}
\maketitle
\begin{abstract}
We introduce a new consensus based optimization (CBO) method where interacting particle system  is driven by  jump-diffusion stochastic differential equations.  We study well-posedness of the particle system as well as of its mean-field limit.  The major contributions of this paper are proofs of convergence of the interacting particle system towards  the mean-field limit and convergence of a discretized particle system towards the continuous-time dynamics in the mean-square sense. We also prove convergence of the mean-field jump-diffusion SDEs towards global minimizer for a large class of objective functions. We demonstrate improved performance of the proposed CBO method over earlier CBO methods in numerical simulations on benchmark objective functions.   
\end{abstract}

\section{Introduction}
Large-scale individual-based models have become a well-established modelling tool in modern science and engineering, with applications including pedestrian motion, collective animal behaviour, swarm robotics and molecular dynamics, among many others. Through the iteration of basic interactions forces such as attraction, repulsion, and alignment,  these complex systems of  exhibit a rich self-organization behaviour (see e.g. \cite{cbo23,cbo20,cbo21,cbos19,cbo22,cbo39}).

Over the last decades, individual-based models have also entered the field of global optimization and its many applications in operations research, control, engineering, economics, finance, and machine learning. In many  applied problems arising in the aforementioned fields, the objective function to be optimized can be non-convex and/or non-smooth, disabling the use of traditional continuous/convex optimization technique. In such scenarios, individual-based metaheuristic models have been proven surprisingly effective. Examples include genetic algorithms, ant colony optimization, particle swarm optimization, simulated annealing, etc. (see \cite{cbo26,cbo24,cbo25} and references therein). These methods are probabilistic in nature which set them apart from other derivative-free algorithms \cite{cbo30}. Unlike many convex optimization methods, metaheuristic algorithms,  are relatively simple to implement and easily parallelizable. This combination of simplicity and effectiveness has fuelled the application of metaheuristic in complex engineering problems such as shape optimization, scheduling problems, and hyper-parameter tuning in machine learning models. However, it is often the case that metaheuristics lack rigorous convergence results, a question which has become an active area of research \cite{cbo50,cbo41}.

In \cite{cbo1}, the authors introduced a optimization algorithm which employs an individual-based model to frame a global minimization
\begin{equation*}
    \min\limits_{x \in \mathbb{R}^{d}} f(x),
\end{equation*}
where $f(x)$ is a positive function from $\mathbb{R}^{d}$ to $\mathbb{R}$, as a consensus problem . In this model, each individual particle explores the energy landscape given by $f(x)$, broadcasting its current value to the rest of the ensemble through a weighted average. This iterated interaction generates trajectories which flock towards a consensus point which correspond to a global minimizer of $f(x)$, hence the name \textit{Consensus Based Optimization} (CBO). We refer to \cite{cbo40,cbo41} for two recent surveys on the topic. The dynamics of existing CBO models are governed by stochastic differential equations with Wiener noise \cite{cbo1,cbo2,cbo3}. Hence, we can resort to a toolbox from stochastic calculus and stochastic numerics  to perform analysis of these models.  This amenability of CBO models to theoretical as well as numerical  analysis differentiates them from other agent based optimization algorithms.  

In this paper, we propose a  new CBO model which is governed by jump-diffusion stochastic differential equations. This means randomness in the dynamics of the proposed CBO models comes from Wiener noise  as well as compound Poisson process. The following are the contributions of this paper:
\begin{itemize}
    \item[(i)] We prove the well-posedness of the interacting-particle system and of its mean-field limit driven by jump-diffusion SDEs and convergence of the mean-field SDEs to the global minimum. The approach to study well-posedness and convergence to the global minimum is similar to \cite{cbo2} but adapted to the jump-diffusion case with time-dependent coefficients.
    
    \item[(ii)] The major contribution of the paper is that we prove mean-square convergence of the interacting particle system to the mean-field limit when number of particles, $N$, tend to $\infty$. This also implies convergence of the particle system towards the mean-field limit in $2-$Wasserstein metric. Let us emphasize that we prove this result for quadratically growing objective function. We also study convergence of the implementable discretized particle system towards the jump-diffusion SDEs as the discretization step, $h$, goes to $0$. Our results can be utilized for the earlier CBO models \cite{cbo1,cbo2,cbo3}.  
    
    \item[(iii)] As illustrated in the numerical experiments, the addition of a jump-diffusion process in the particle system leads to a more effective exploration of the energy landscape. This particularly relevant when a good prior knowledge of the optimal solution for initialization of the CBO is not available.  
\end{itemize}

As was highlighted in \cite[Remark 3.2]{cbo2}, it is not straightforward to prove convergence of the interacting particle system towards its mean-field limit, even after proving uniform in $N$ moment bound of the solutions of the SDEs driving particles system. Convergence results of this type have been proved for special cases of compact manifolds (see \cite{cbo34} for compact hypersurfaces and \cite{cbo51} for Stiefel manifolds)  and  globally Lipschitz continuous objective functions.  In this case, not only the objective function is bounded but also particles are evolving on a compact set. Under the assumptions on the objective function as in our paper, in the diffusion case weak convergence of the empirical measure of a particle system to the law of the corresponding mean field SDEs has been proved in \cite{cbo41, cbo52}  exploiting Prokhorov's theorem. 
Here we prove  convergence of the particle system to the mean-field SDEs  in the mean-square sense for a quadratically growing locally-Lipschitz objective function defined on $\mathbb{R}^{d}$.  

Furthermore, practical implementation of the particle system corresponding to a CBO model needs a numerical approximation in the mean-square sense. We utilize an explicit Euler scheme to implement the proposed jump-diffusion CBO model. This leads to  the question whether the Euler scheme converges to the CBO model taking into account that the coefficients of the particle system are not globally Lipschitz and the Lipschitz constants grow exponentially when the objective function is not bounded. At the same time, the coefficients of the particle system have linear growth at infinity.  In the case of jump-diffusion SDEs, earlier works either showed convergence of the Euler scheme in the case of globally Lipschitz coefficients \cite{cbo28} or proposed special schemes in the case of non-globally Lipschitz coefficients with super-linear growth, e.g. a tamed Euler scheme \cite{cbo15}. Here we prove mean-square convergence of the Euler scheme and we show that this convergence is uniform in the number of particles $N$, i.e. the choice of a discretization time-step $h$ is independent of $N$. Our convergence result also holds for earlier CBO models \cite{cbo1,cbo2,cbo3}.

In Section \ref{sec_lit_rev}, we first present a review of  existing CBO models and then describe our CBO model driven by jump-diffusion SDEs. We also formally introduce  mean-field limit of the new CBO model. In Section~\ref{sec_wel_pos}, we focus on  well-posedness of the interacting particle system behind the new CBO model and its mean-field limit. In Section~\ref{cbo_conv_res}, we discuss convergence of the mean field limit towards a point in $\mathbb{R}^{d}$ which approximates the global minimum, convergence of the interacting particle system towards mean field limit, and convergence of the implementable discretized particle system towards the  particle system.  We present results of numerical experiments in Section~\ref{cbo_num_exp}  to compare performance of our model and  the existing CBO models.   

Throughout the paper, $C$ is a floating constant which may vary at different places. We denote $(a\cdot b)$ as dot product between two vectors, $a,b \in \mathbb{R}^{d}$. We will omit brackets $()$ wherever it does not lead to any confusion.

\section{ CBO models : existing and new}\label{sec_lit_rev}
In Section~\ref{sec_ex_cbo}, we review the existing CBO models. In Section~\ref{sec_our_mod}, we introduce a new CBO model driven by jump-diffusion SDEs and and discuss potential advantages of adding jumps to CBO models which are confirmed by numerical experiments in Section~\ref{cbo_num_exp}. The numerical experiments of Section~\ref{cbo_num_exp} are conducted using the Euler scheme presented in Section~\ref{sec_our_mod}.   
\subsection{Review of the existing CBO models}\label{sec_ex_cbo}
Let $N \in \mathbb{N}$ denote the number of agents with position vector, $X^{i}_{N}(t) \in \mathbb{R}^{d}$, $i=1,\dots,N$. The following model was proposed in \cite{cbo1}:
\begin{align}\label{cbos1.2}
    dX^{i}_{N}(t) &= -\beta(X^{i}_{N}(t) - \bar{X}^{\alpha,f}_{N}(t))H^{\epsilon}(f(X^{i}_{N}(t)) - f(\bar{X}^{\alpha,f}_{N}(t)))dt \nonumber \\ & \;\;\;\; + \sqrt{2}\sigma \vert X^{i}_{N}(t) -\bar{X}^{\alpha,f}_{N}(t)\vert dW^{i}(t),\;\;\;\;i = 1,\dots,N,
\end{align}
where $H^{\epsilon} : \mathbb{R} \rightarrow \mathbb{R}$ is a smooth regularization of the Heaviside function, $W^{i}(t)$, $i=1,\ldots , N,$ represent $N-$independent $d$-dimensional standard Wiener processes, $\beta> 0$, $\sigma > 0 $, and $\bar{X}^{\alpha,f}_{N}(t)$  is given by
\begin{equation} \label{cbo2.2}
    \bar{X}^{\alpha,f}_{N}(t) = \frac{\sum_{i =1}^{N}X^{i}_{N}(t)w_{f}^{\alpha}(X^{i}_{N}(t))}{\sum_{i =1}^{N}w_{f}^{\alpha}(X^{i}_{N}(t))},
\end{equation}
with $w_{f}^{\alpha}(x) = \exp{(-\alpha f(x))}$, $\alpha > 0$.   

Each particle $X^{i}_{N}$ at time $t$ is assigned an opinion $f(X^{i}_{N}(t))$. The lesser the value of $f$ for a particle, the more is the influence of that particle, i.e. the more weight is assigned to that  particle at that time as can be seen in (\ref{cbo2.2}) of the instantaneous weighted average. If the value $f(X^{i}_{N}(t))$ of a particle $X^{i}_{N}$ at time $t$ is greater than the value $f(\bar{X}_{N}^{\alpha,f}(t))$ at the instantaneous weighted average $\bar{X}_{N}^{\alpha, f}(t)$ then the regularised Heaviside function forces the particle $X^{i}_{N}$ to drift  towards $\bar{X}_{N}^{\alpha,f}$.  If the opinion of $i$-th particle matters more among the interacting particles, i.e. the value $f(X^{i}_{N}(t))$ is less than $f(\bar{X}^{i}_{N}(t))$, then it is not  beneficial for it to move towards $\bar{X}_{N}^{\alpha, f}$. The noise term is added to explore the space $\mathbb{R}^{d}$ and to avoid non-uniform consensus.
The noise intensity induced in the dynamics of the $i-$th  particle at time $t$  takes into account the distance of the particle from the instantaneous weighted average, $\bar{X}_{N}^{\alpha, f}(t)$. Over a period of time as the particles start moving towards a consensus opinion, the coefficients in (\ref{cbos1.2}) go to zero. 

One can observe that the more influential opinion a particular particle has, the higher is the weight assigned to that particle in the instantaneous weighted average (\ref{cbo2.2}). Based on this logic, in \cite{cbo2} the authors dropped the regularised Heaviside function in the drift coefficient and the model (\ref{cbos1.2}) was simplified as follows:
\begin{equation}\label{cbos1.3}
    dX^{i}_{N}(t) = -\beta (X^{i}_{N}(t) -\bar{X}_{N}^{\alpha,f}(t)) dt + \sigma \vert X^{i}_{N}(t) - \bar{X}_{N}^{\alpha,f}(t)\vert dW^{i}(t),\;\;\; i = 1,\dots,N,
\end{equation}
with $\beta$, $ \sigma$,  $\bar{X}_{N}^{\alpha,f}$ as in (\ref{cbos1.2})-(\ref{cbo2.2}).

The major drawback of the consensus based models (\ref{cbos1.2}) and (\ref{cbos1.3}) is that the parameters $\beta$ and $\sigma$ are dependent on the dimension $d$. To illustrate this fact, we replace $\bar{X}_{N}^{\alpha,f}$ in (\ref{cbos1.3}) by a fixed vector $V \in \mathbb{R}^{d}$. Then, using Ito's formula, we have
\begin{equation}
\frac{d}{dt}\mathbb{E}|X^{i}_{N}(t)-V|^{2}  = (-2\beta + \sigma^{2}d)\mathbb{E}|X^{i}_{N}(t)-V|^{2},\;\;\;\; i = 1,\dots,N.
\end{equation}
As one can notice, for particles to reach the consensus point whose position  vector is $V$, one needs $2\beta > d\sigma^{2}$. To overcome this deficiency, the authors of \cite{cbo3} proposed the following model which is based on component-wise noise intensity instead of isotropic noise used in (\ref{cbos1.2}) and (\ref{cbos1.3}):
\begin{equation}\label{cbos1.5}
    dX^{i}_{N}(t) = -\beta (X^{i}_{N}(t) - \bar{X}_{N}^{\alpha,f}(t)) dt + \sqrt{2}\sigma\diag(X^{i}_{N}(t) - \bar{X}_{N}^{\alpha,f}(t)) dW^{i}(t), \;\;\;\; i =1,\dots,N,
\end{equation}
where $\beta, \sigma$, and $\bar{X}_{N}^{\alpha,f} $ are as in (\ref{cbos1.2})-(\ref{cbo2.2}), and $\diag(U)$ is a diagonal matrix whose diagonal is a vector $U \in \mathbb{R}^{d}$.
Now, if we replace $\bar{X}_{N}^{\alpha, f}$ by a fixed vector $V$ and then use Ito's formula for (\ref{cbos1.5}), we get
\begin{align}
    \frac{d}{dt}\mathbb{E}|X^{i}_{N}(t) -V|^{2} & = -2\beta\mathbb{E}|X^{i}_{N}(t) -V|^{2} + \sigma^{2}\mathbb{E}\sum\limits_{j=1}^{d}(X^{i}_{N}(t) - V)_{j}^{2} \nonumber \\ & =(-2\beta + \sigma^{2})\mathbb{E}|X^{i}_{N}(t) - V|^{2},\;\;\;\;i=1,\dots,N,
\end{align}
where $(X_{N}^{i}(t) - V)_{j} $ denotes the  $j-$th component of $(X_{N}^{i}(t) -V)$. It is clear that in this model there is no dimensional restriction on $\beta$ and $\sigma$.

Other CBO models \cite{cbo4,cbo5} are based on interacting particles driven by common noise. Since the same noise drives all the particles, the exploration is not effective. Therefore, they are not scalable with respect to dimension and do not perform well in contrast to the CBO models (\ref{cbos1.2}), (\ref{cbos1.3}), (\ref{cbos1.5}) and model introduced in Section~\ref{sec_our_mod}. This fact is demonstrated in experiments in Section~\ref{cbo_num_exp}.

\subsection{Jump-diffusion CBO models}\label{sec_our_mod}


Let us consider the following jump-diffusion model:
\begin{align}\label{cbos1.6}
    dX^{i}_{N}(t) &= -\beta(t)(X^{i}_{N}(t^{}) - \bar{X}_{N}(t^{}))dt + \sqrt{2}\sigma(t) \diag(X^{i}_{N}(t^{})-\bar{X}_{N}(t^{}))dW^{i}(t) \nonumber  \\ &\;\;\;\; + \gamma(t)\diag(X^{i}_{N}(t^{-}) -\bar{X}_{N}(t^{-}))dJ^{i}(t), 
    \;\; i=1,\dots,N, 
\end{align}
with
\begin{equation} \label{cbo_neweq_2.8}
J^{i}(t) = \sum\limits_{j=1}^{N^{i}(t)}Z^{i}_{j},
\end{equation}
where $N^{i}(t)$, $i=1\dots,N$ are $N-$independent Poisson processes with jump intensity $\lambda$ and $Z_{j}^{i} = (Z_{j,1}^{i},\dots,Z_{j,d}^{i})^{\top}$  are i.i.d. $d$-dimensional random variables denoting $j-$th jump by $i-$th particle and $Z_{j}^{i} \sim Z$. The distribution of $Z$ is called as jump size distribution. For the sake of convenience, we write $Z_{l}$ as the $l$-th component of vector $Z$. We assume that each component $Z_{l}$ of $Z$ is also i.i.d. random variable and distributed as
\begin{equation}
    Z_{l} \sim \Zstroke,
\end{equation}
where $\Zstroke $ is an $\mathbb{R}-$valued random variable whose  probability density is given by $\rho_{\zstroke}(\zstroke)$ such that $\mathbb{E}(\Zstroke) = \int_{\mathbb{R}}\zstroke \rho_{\zstroke}(\zstroke)d\zstroke = 0$. We also denote the probability density of $Z$ as $\rho_{z}(z) = \prod_{l=1}^{d}\rho_{\zstroke}(z_{l})
$. Note that $\mathbb{E}(Z)$ is a $d-$dimensional zero vector, since each $Z_{l}$ is distributed as $\Zstroke$. The Wiener processes $W^{i}(t)$, the Poisson processes $N^{i}(t)$, $i = 1\dots, N$ and the jump sizes $Z$ are assumed to be mutually independent (see further theoretical details concerning L\'{e}vy-driven SDEs in \cite{cbos11}).
 Also, $\beta(t)$, $\sigma(t), \gamma(t)$ are continuous functions and
\begin{equation} \label{cbos1.7}
    \bar{X}_{N}(t) = (\bar{X}^{1}_{N}(t),\dots, \bar{X}^{d}_{N}(t)) := \frac{\sum_{i=1}^{N}X^{i}_{N}(t)e^{-\alpha f(X^{i}_{N}(t))}}{\sum_{i=1}^{N}e^{-\alpha f(X^{i}_{N}(t))}},
\end{equation}
with $\alpha > 0$. Note that we have omitted $\alpha $ and $f$ of $\bar{X}_{N}^{\alpha,f}$ in the notation used in  (\ref{cbos1.6})  for the simplicity of writing.

We recall the meaning of the jump term
\begin{equation*}
\int_{0}^{t}\gamma(s)\diag(X^{i}(s^{-}) -\bar{X}_{N}(s^{-}))dJ^{i}(s)=
\sum_{j=1}^{N^{i}(t)}\gamma(\tau_{j})\diag(X^{i}(\tau_{j}^{-}) - \bar{X}_{N}(\tau_{j}^{-}))Z^{i}_{j} ,
  \end{equation*}where $\tau_{j}$ denotes the time of $j$-th jump of the Poisson process $N^{i}(t)$.  Thanks to the assumption  that $\mathbb{E}(\Zstroke) = 0$ \big(which in turn implies $\mathbb{E}(Z^{i}_{j,l}) = 0$,  $j=1,\dots,N^{i}(t)$, $i =1,\dots,N$, $l =1,\dots,d$\big), the above integral is a martingale, and hence (similar to Ito's integral term in (\ref{cbos1.6})) it does not bias trajectories of $X_{N}^{i}(t)$, $i=1,\dots,N$.

The jump diffusion SDEs (\ref{cbos1.6}) are different from (\ref{cbos1.5}) in the two ways:
\begin{itemize}
    \item The SDEs (\ref{cbos1.6}) are a consequence of interlacing of Ito's diffusion  by jumps arriving according to the Poisson process whose jump intensity is given by $\lambda$. 
    \item We take $\beta(t)$ as a continuous positive non-decreasing function of $t$ such that $\beta(t) \rightarrow \beta > 0$ as $t \rightarrow \infty$, $\sigma(t)$ as a continuous positive non-increasing function of $t$ such that $\sigma(t) \rightarrow \sigma > 0$ as $t \rightarrow \infty$ and $\gamma(t)$ as a continuous non-negative non-increasing function of $t$ such that $\gamma(t) \rightarrow \gamma \geq 0$ as $t \rightarrow \infty$. 
\end{itemize}
Although we analyse CBO model (\ref{cbos1.6}) with time-dependent parameters, a decision to take parameters time-dependent or not is problem specific. Note that the particles driven by SDEs (\ref{cbos1.6}) jump at different times with different jump sizes and jumps arrive according to the Poisson process with intensity $\lambda$.

We can also write the jump-diffusion SDEs (\ref{cbos1.6}) in terms of Poisson random measure \cite{cbos11} as
\begin{align}\label{cboeq1.8}
    dX^{i}_{N}(t) &= -\beta(t)(X^{i}_{N}(t^{}) -\bar{X}_{N}(t^{}))dt + \sqrt{2}\sigma(t)\diag(X^{i}_{N}(t^{}) -\bar{X}_{N}(t^{}))dW^{i}(t) \nonumber\\ &  \;\;\;\;+\int_{\mathbb{R}^{d}}\gamma(t)\diag(X^{i}_{N}(t^{-}) -\bar{X}_{N}(t^{-}))z\mathcal{N}^{i}(dt,dz),
\end{align}
where $\mathcal{N}^{i}(dt,dz)$, $i=1,\dots, N$, represent the independent Poisson random measures with intensity measure $\nu(dz)dt$ and $\nu(dz)$ is a L\'{e}vy measure which is finite in our case (\ref{cbos1.6}). Although for simplicity we introduced our model as (\ref{cbos1.6}), in proving well-posedness and convergence results we will make use of (\ref{cboeq1.8}).

We can formally write the mean field limit of the model (\ref{cbos1.6}) as the following McKean-Vlasov SDEs:
\begin{align}\label{cbomfsde}
    dX(t) &= -\beta(t)(X(t^{}) -\bar{X}(t^{}))dt + \sqrt{2}\sigma(t) \diag(X(t^{})-\bar{X}(t^{}))dW(t) \nonumber \\
    &\;\;\;\; +\gamma(t)\diag(X(t^{-})   -\bar{X}(t^{-}))dJ(t),
\end{align}
where $J(t) = \sum_{j=1}^{N(t)}Z_{j}$, $N(t)$ is a Poisson process with intensity $\lambda$, and
\begin{align}\label{eqcbo2.12}
\bar{X}(t) := \bar{X}^{\mathcal{L}_{X(t)}} = \frac{\int_{\mathbb{R}^{d}} xe^{-\alpha f(x)}\mathcal{L}_{X(t)}(dx)}{\int_{\mathbb{R}^{d}}e^{-\alpha f(x)}\mathcal{L}_{X(t)}(dx)} =  \frac{\mathbb{E}\big(X(t)e^{-\alpha f(X(t))}\big)}{\mathbb{E}\big(e^{-\alpha f(X(t))}\big)}, 
\end{align}
with $\mathcal{L}_{X(t)} := \text{Law}(X(t))$.
We can rewrite the mean field jump diffusion SDEs (\ref{cbomfsde}) in terms of Poisson random measure as
\begin{align}\label{cbomfsdep}
    dX(t) &= -\beta(t)(X(t^{}) - \bar{X}(t^{}))dt + \sqrt{2}\sigma(t)\diag(X(t^{}) - \bar{X}(t^{}))dW(t) \nonumber \\ 
    &\;\;\;\; + \gamma(t) \int_{\mathbb{R}^{d}}\diag(X(t^{-}) - \bar{X}(t^{-}))z\mathcal{N}(dt,dz).
\end{align}

\subsubsection{Other jump-diffusion CBO models}

Although the aim of the paper is it to analyse the CBO model (\ref{cboeq1.8}), we discuss three other jump-diffusion CBO models of interest.

\textbf{Additional Model 1 :} Writing (\ref{cbos1.6}) in terms of Poisson random measure suggests that we can also consider an infinite activity L\'{e}vy process, e.g. an $\alpha-$stable process,  to introduce jumps in dynamics of particles. We can write the CBO model as
\begin{align}\label{}
    dX^{i}_{N}(t) &= -\beta(t)(X^{i}_{N}(t^{}) -\bar{X}_{N}(t^{}))dt + \sqrt{2}\sigma(t)\diag(X^{i}_{N}(t^{}) -\bar{X}_{N}(t^{}))dW^{i}(t) \nonumber\\ &  \;\;\;\;+\int_{\mathbb{R}^{d}}\gamma(t)\diag(X^{i}_{N}(t^{-}) -\bar{X}_{N}(t^{-}))z\mathcal{N}^{i}(dt,dz),
\end{align}

However, numerical approximation of SDEs driven by infinite activity L\'{e}vy processes is computationally more expensive (see e.g. \cite{cbo28, cbos12}), hence  it can be detrimental for the overall CBO performance.

\textbf{Additional Model 2 :} In the SDEs (\ref{cbos1.6}), the intensity of Poisson process $\lambda$ is constant. If we take jump intensity as $\lambda(t) $, i.e. a function of $t$ then the corresponding SDEs will be as follows:
\begin{align}\label{cbos1.9}
    dX^{i}(t) &= -\beta(t)(X^{i}_{N}(t^{}) - \bar{X}_{N}(t^{}))dt + \sqrt{2}\sigma(t) \diag(X^{i}_{N}(t^{})-\bar{X}(t^{}))dW^{i}(t) \nonumber  \\ &\;\;\;\; + \diag(X^{i}_{N}(t^{-}) -\bar{X}_{N}(t^{-}))dJ^{i}(t), 
    \;\; i=1,\dots,N, 
\end{align}
where all the notation are as in (\ref{cbos1.6}) and (\ref{cbos1.7}) except here the intensity of the Poisson processes $N^{i}(t)$ is a time-dependent function $\lambda(t)$. It is assumed that $\lambda(t)$ is a decreasing function such that $\lambda(t) \rightarrow 0$ as $t \rightarrow \infty$. Also, in comparison with (\ref{cbos1.6}), there is no $\gamma(t)$ in the jump component of (\ref{cbos1.9}). 
 Note that, the compound Poisson process with constant jump intensity $\lambda $ is a L\'{e}vy process but with time-dependent jump intensity $\lambda(t)$, it is not a L\'{e}vy process, rather it is an additive process. Additive process is a generalization of L\'{e}vy process which satisfies all conditions of L\'{e}vy process except stationarity of increments \cite{cbos14}. The SDEs (\ref{cbos1.9}) present another jump-diffusion CBO model driven by additive process. The analysis of model (\ref{cbos1.9})  follows similar arguments since the jump-diffusion SDEs (\ref{cbos1.9}) can also be written in terms of the Poisson random measure with intensity measure $\nu_{t}(dz)dt $, where $(\nu_{t})_{t\geq 0}$ is a family of L\'{e}vy measures.

\textbf{Additional Model 3 :} In model (\ref{cboeq1.8}), the particles have idiosyncratic noise which means they are driven by different Wiener processes and different  compound Poisson processes. Instead, we can have a different jump-diffusion model in which the same Poisson noise drives particle system but jumps sizes still independently vary for all particles. This means jumps arrive at the same time for all particles,  but particles jump with different jump-sizes. We can write CBO model as
\begin{align} \label{cbo_neweq_2.17}
     dX^{i}_{N}(t) &= -\beta(t)(X^{i}_{N}(t^{}) -\bar{X}_{N}(t^{}))dt + \sqrt{2}\sigma(t)\diag(X^{i}_{N}(t^{}) -\bar{X}_{N}(t^{}))dW^{i}(t) \nonumber\\ &  \;\;\;\;+\int_{\mathbb{R}^{d}}\gamma(t)\diag(X^{i}_{N}(t^{-}) -\bar{X}_{N}(t^{-}))z\mathcal{N}^{}(dt,dz).
\end{align}
We compare performance of the jump-diffusion CBO models (\ref{cboeq1.8}) and (\ref{cbo_neweq_2.17}) in Section~\ref{cbo_num_exp}. 


\subsubsection{Discussion}\label{cbo_sec_disc}
Firstly, we will discuss dependence  of the parameters $\beta(t)$, $\sigma(t)$, $\gamma(t)$ and $\lambda$ on dimension $d$. The independent and identical distribution of $Z_{l}$, which denotes the $l-$th component of $Z$, result in the non-dependency of parameters on dimension in the similar manner as for the model (\ref{cbos1.5}). We illustrate this fact by fixing a vector $V \in \mathbb{R}^{d}$ and replacing $\bar{X}_{N}$ in (\ref{cboeq1.8}) by $V$ then using Ito's formula and the assumption made on $\rho_{\zstroke}(\zstroke)$, we have
\begin{align}
    \frac{d}{dt}\mathbb{E}|X^{i}_{N}(t) - V|^{2} &= -2 \beta(t)\mathbb{E}|X^{i}_{N}(t) - V|^{2} + \sigma^{2}(t)\sum\limits_{j =1}^{d}\mathbb{E}(X^{i}_{N}(t) - V)_{j}^{2} \nonumber \\ & \;\;\;\; + \lambda \int_{\mathbb{R}^{d}}\big(|X^{i}_{N}(t) - V + \gamma(t)\diag(X^{i}_{N}(t) - V)z|^{2} - |X^{i}_{N}(t) -V|^{2}\big)\rho_{z}(z)dz \nonumber \\ & =  (-2 \beta(t) + \sigma^{2}(t))\mathbb{E}|X^{i}_{N}(t) - V|^{2} + \lambda\int_{\mathbb{R}^{d}}\gamma^{2}(t)|\diag(X^{i}_{N}(t)-V)z|^{2}\rho_{z}(z)dz 
    \nonumber \\ & = (-2 \beta(t) + \sigma^{2}(t))\mathbb{E}|X^{i}_{N}(t) - V|^{2} + \lambda \gamma^{2}(t)\sum\limits_{j=1}^{d}\int_{\mathbb{R}^{d}}(X^{i}_{N}(t)-V)_{j}^{2}z_{j}^{2}\prod_{l=1}^{d}\rho_{\zstroke}(z_{l})dz
    \nonumber \\ & = \big(-2 \beta(t) + \sigma^{2}(t) + \lambda \gamma^{2}(t)\mathbb{E}(\Zstroke^{2})\big)\mathbb{E}|X^{i}_{N}(t) - V|^{2}. \label{cboeq2.16}
\end{align}
We can choose $\beta(t)$, $\sigma(t)$, $\gamma(t)$, $\lambda$ and distribution of $\Zstroke$ guaranteeing that there is a $t_{*} \geq 0$ such that $-2\beta(t) + \sigma^{2}(t)+ \lambda \gamma^{2}(t)\mathbb{E}(\Zstroke^{2}) < 0 $ for all $t \geq t_{*}$ and such a choice is independent of $d$. It is clear from (\ref{cboeq2.16}) that with this choice, $\mathbb{E}|X^{i}_{N}(t)-V|^{2}$, $i =1,\dots,N$, decay in time as $t\rightarrow \infty$. 

In the previous CBO models, there were only two terms namely, the drift term and the diffusion term. The drift tries to take the particles towards their instantaneous weighted average. The diffusion term helps in exploration of the state space with the aim to find a state with better weighted average than the current one. The model (\ref{cbos1.6}) 
contains one extra term, which we call the jump term. 
Jumps help in intensifying the search in a search space and aids in avoiding premature convergence or trapping in  local minima. This results in more effective use of the interaction of particles. 

Moreover, the effect of jumps decays with time in (\ref{cbos1.6}) by virtue of decreasing $\gamma (t)$. 
The reason for considering the model (\ref{cbos1.6}) where jumps affect only the initial period of time is that we want particles to explore more space faster at the beginning of simulation and, as soon as the weighted average of particles is in a vicinity of the global minimum, we do not want jumps to affect convergence of particles towards that consensus point lying in the close neighbourhood of the global minimum. Therefore, the time-dependent parameters and degeneracy of the coefficients  help in exploiting the searched space.  

As a consequence, the jump-diffusion noise and degenerate time-dependent coefficients in  model (\ref{cbos1.6}) 
may help in keeping the  balance of \textbf{\textit{exploration}} and \textbf{\textit{exploitation}} by interacting particles over a period of time. We will continue this discussion on exploration and exploitation 
in Section~\ref{cbo_num_exp}, where the proposed CBO method is tested. 

\subsubsection{Implementation}\label{subsec_implemen}

Let $0=t_{0}<\dots<t_{n}=T$ be a uniform partition of the time interval $[0,T]$ into $n $ sub-intervals such that $h:= t_{k+1} -t_{k}$, $k =0,\dots, n-1$ and $T = nh$. To approximate (\ref{cbos1.6}), we construct a Markov chain $(Y_{N}^{i}(t_{k}))$, $ k = 1,\dots, n$, using the following Euler scheme:
\begin{align}\label{cbo_dis_ns}
    Y^{i}_{N}(t_{k+1}) &= Y_{N}^{i}(t_{k}) - \beta(t_{k})(Y^{i}_{N}(t_{k}) - \bar{Y}_{N}(t_{k}) ) h + \sigma(t_{k})\diag(Y^{i}_{N}(t_{k})- \bar{Y}_{N}(t_{k}))\Delta W(t_{k})\nonumber \\& \;\;\;\;+ \gamma(t_{k})\sum\limits_{j = N^{i}(t_{k})+1}^{N^{i}(t_{k+1})}\diag(Y^{i}_{N}(t_{k}) -\bar{Y}_{N}(t_{k})) Z^{i}_{j},
\end{align}
where $\Delta W(t_{k}) = W(t_{k+1}) - W(t_{k})$ has Gaussian distribution with mean $0$ and variance $h$, $Z^{i}_{j}$ denotes $j-$th jump size of the $i-$th particle, $N^i(t)$ are independent Poisson processes with jump intensity $\lambda$, 
and
\begin{align}\label{cbo_e2.18}
    \bar{Y}_{N}(t) = \sum\limits_{i=1}^{N}Y^{i}_{N}(t)\frac{e^{-\alpha f(Y^{i}_{N}(t))}}{\sum_{j=1}^{N}e^{-\alpha f(Y^{i}_{N}(t))}}.
\end{align}
To implement the discretization scheme we initialize the $N\times d$ matrix $Y$ at time $t_0=0$, and update it for $n$ iterations using (\ref{cbo_dis_ns}) by calculating (\ref{cbo_e2.18}) at each iteration.  
The code to implement above numerical scheme utilizing $N\times d$ matrix, which allows to save memory and time in computations, is available on  \href{https://github.com/akashspace/Consensus-based-opmization}{github}. We will discuss the convergence of scheme (\ref{cbo_dis_ns}) in Subsection~\ref{cbo_conv_ns}.


  


\section{Well-posedness results}\label{sec_wel_pos}
In Section~\ref{sec_well_pos_1}, we discuss well-posedness of the interacting particle system (\ref{cboeq1.8}) and prove moment bound for this system. In Section~\ref{sec_well_pos_2}, we prove well-posedness and moment bound of the mean field limit (\ref{cbomfsdep}) of the particle system (\ref{cboeq1.8}). 
\subsection{Well-posedness of the jump-diffusion particle system}\label{sec_well_pos_1}
This section is focused on showing existence and uniqueness of the solution of (\ref{cboeq1.8}).
    We first introduce the notation which are required in this section. 
    
    Let us denote $\textbf{x}_{N} := (x_{N}^{1},\dots,x_{N}^{N})^{\top} \in \mathbb{R}^{Nd}$,  $\bar{\textbf{x}}_{N} = \sum_{i=1}^{N}x^{i}_{N}e^{-\alpha f(x^{i}_{N})}/\sum_{j=1}^{N}e^{-\alpha f(x^{j}_{N})}$, $\textbf{W}(t) := (W^{1}_{}(t),\dots,W_{N}^{}(t))^{\top}$, $\textbf{F}_{N}(\textbf{x}_{N}) := \big( F^{1}_{N}(\textbf{x}_{N}),\dots,F^{N}_{N}(\textbf{x}_{N})\big)^{\top} \in \mathbb{R}^{Nd}$ with $F_{N}^{i}(\textbf{x}_{N}) = (x_{N}^{i} - \bar{x}_{N}) \in \mathbb{R}^{d}$ for all $i = 1,\dots,N$,  $\textbf{G}_{N}(\textbf{x}_{N}) : = \diag(\textbf{F}_{N}(\textbf{x}_{N})) \in \mathbb{R}^{Nd\times Nd}$ and $\textbf{J}(t) = ({J}^{1}(t),\dots,{J}^{N}(t))$, where $J^{i}(t)$ is from (\ref{cbo_neweq_2.8}) which implies $\int_{0}^{t}\gamma(t)\diag(F^{i}_{N}(\textbf{x}_{N}^{i}))d{J}^{i}(t) = \int_{0}^{t}\int_{\mathbb{R}^{d}}\diag(F^{i}_{N}(\textbf{x}_{N}))z\mathcal{N}^{i}(dt,dz)$. 
    Let us represent $\ell(dz)$ as the Lebesgue measure of $dz$, and for the sake of convenience we will use $dz$ in place of $\ell(dz)$ whenever there is no confusion.
We can write the particle system (\ref{cboeq1.8}) using the above notation as
\begin{align}\label{cboeq3.1}
    d\textbf{X}_{N}(t) = \beta(t)\textbf{F}_{N}(\textbf{X}_{N}(t^{-}))dt + \sqrt{2}\sigma(t)\textbf{G}_{N}(\textbf{X}_{N}(t^{-}))d\textbf{W}(t) + \gamma(t)\textbf{G}_{N}(\textbf{X}_{N}(t^{-}))d\textbf{J}(t).
\end{align}
In order to show well-posedness of (\ref{cboeq3.1}), we need the following natural assumptions on the objective function $f$. Let 
\begin{equation}\label{cbo_eq_fm}
f_{m} := \inf f.
\end{equation}
\begin{assumption}\label{cboh3.1}
$f_{m} > 0$. 
\end{assumption}
\begin{assumption}\label{cboasu1.1}
$f : \mathbb{R}^{d} \rightarrow \mathbb{R}$ is locally Lipschtiz continuous, i.e. there exists a positive function $L(R)$ such that
\begin{equation*}
    |f(x) - f(y) | \leq L(R)|x-y|,
\end{equation*}
whenever $|x|$, $|y| \leq R$, $x$, $y \in \mathbb{R}^{d}$, $R>0$.
\end{assumption}
Assumption~\ref{cboasu1.1} is used for proving local Lipschitz continuity and linear growth of $F^{i}_{N}$ and $G^{i}_{N}$, $i=1,\dots,N$.  Let $B(R) = \{ x\in \mathbb{R}^{d}\;;\;|x| \leq R\}$. 
\begin{lemma}\label{cbolemma3.1}
Under Assumptions~\ref{cboh3.1}-\ref{cboasu1.1}, the following inequalities hold for any $\textbf{x}_{N}$, $\textbf{y}_{N} \in \mathbb{R}^{Nd}$ satisfying $\sup_{i=1,\dots,N}|x^{i}_{N}|, 
\sup_{i=1,\dots,N}|y^{i}_{N}| \leq R$ and for all $i = 1,\dots,N$:
\begin{enumerate}
    \item $ |F^{i}_{N}(\textbf{x}_{N}) -F^{i}_{N}(\textbf{y}_{N})|  \leq |x^{i}_{N} - y^{i}_{N}| + \frac{C(R)}{N^{1/2}}|\textbf{x}_{N} - \textbf{y}_{N}|,$
        \item $ |F^{i}_{N}(\textbf{x}_{N})|^{2} \leq 2(|x_{N}^{i}|^{2} + |\textbf{x}_{N}|^{2}), $
\end{enumerate}
where $C(R) =  e^{\alpha (|f|_{L_{\infty}(B(R))} - f_{m}})\big( 1+ \alpha R L(R)+ \alpha R L(R) e^{\alpha (|f|_{L_{\infty}(B(R))} - f_{m})})$.
\end{lemma}
\begin{proof}
Let us deal with the first inequality above. We have
\begin{align*}
    |F^{i}_{N}(\textbf{x}_{N}) &- F^{i}_{N}(\textbf{y}_{N})| \leq |x^{i}_{N} - y^{i}_{N}| + \Bigg| \frac{\sum_{i=1}^{N}x^{i}_{N}e^{-\alpha f(x^{i}_{N})}}{\sum_{i=1}^{N}e^{-\alpha f(x^{i}_{N})}} -  \frac{\sum_{i=1}^{N}y^{i}_{N}e^{-\alpha f(y^{i}_{N})}}{\sum_{i=1}^{N}e^{-\alpha f(y^{i}_{N})}}\Bigg| \\ & \leq |x^{i}_{N} - y^{i}_{N}| + \frac{1}{\sum_{j=1}^{N}e^{-\alpha f(x^{j}_{N})}}\Bigg|\sum\limits_{i=1}^{N}\bigg(x^{i}_{N}e^{-\alpha f(x^{i}_{N})} - y^{i}_{N}e^{-\alpha f(y^{i}_{N})}\bigg)\Bigg| \\ & \;\;\;\; + \sum\limits_{i=1}^{N}|y^{i}_{N}|e^{-\alpha f(y^{i}_{N})}\Bigg| \frac{1}{\sum_{j=1}^{N}e^{-\alpha f(x^{j}_{N})}} - \frac{1}{\sum_{j=1}^{N}e^{-\alpha f(y^{j}_{N})}}\Bigg| \\ & \leq  |x^{i}_{N} - y^{i}_{N}| + \frac{1}{\sum_{j=1}^{N}e^{-\alpha f(x^{j}_{N})}}\Bigg(\Bigg|\sum\limits_{i=1}^{N}(x^{i}_{N} - y^{i}_{N})e^{-\alpha f(x^{i}_{N})}\Bigg| + \Bigg|\sum\limits_{i=1}^{N}y^{i}_{N}(e^{-\alpha f(x^{i}_{N})} - e^{-\alpha f(y^{i}_{N})})\Bigg|\Bigg) \\ & \;\;\;\; + \sum\limits_{i=1}^{N}|y^{i}_{N}|e^{-\alpha f(y^{i}_{N})}\Bigg| \frac{1}{\sum_{j=1}^{N}e^{-\alpha f(x^{j}_{N})}} - \frac{1}{\sum_{j=1}^{N}e^{-\alpha f(y^{j}_{N})}}\Bigg|. 
\end{align*}
Using Jensen's inequality, we have
\begin{align*}
    \frac{1}{\frac{1}{N}\sum_{i=1}^{N}e^{-\alpha f(x^{i}_{N})}} &\leq e^{\alpha \frac{1}{N}\sum_{i=1}^{N}f(x^{i}_{N})}.
\end{align*}
Using the Cauchy-Bunyakowsky-Shwartz inequality, we get
\begin{align*}
     &|F^{i}_{N}(\textbf{x}_{N}) - F^{i}_{N}(\textbf{y}_{N})| \leq |x^{i}_{N} - y^{i}_{N}| + e^{\alpha |f|_{L_{\infty}(B(R))}}e^{-\alpha f_{m}}\frac{1}{N}\sum_{i=1}^{N}\big|x^{i}_{N} - y^{i}_{N}\big| + \alpha e^{-\alpha f_{m}}e^{\alpha |f|_{L_{\infty}(B(R))}}L(R)\\ &\times\bigg(\frac{1}{N}\sum\limits_{i=1}^{N}|y^{i}_{N}|^{2}\bigg)^{1/2}\bigg(\frac{1}{N}\sum\limits_{i=1}^{N}|x^{i}_{N} - y^{i}_{N}|^{2}\bigg)^{1/2} +  \alpha e^{-2\alpha f_{m}}e^{2\alpha |f|_{L_{\infty}(B(R))}}\frac{L(R)}{N}\sum\limits_{i=1}^{N}|y^{i}_{N}| \sum\limits_{i=1}^{N}|x^{i}_{N} - y^{i}_{N}|
     \\ & \leq |x^{i}_{N} - y^{i}_{N}| + e^{\alpha |f|_{L_{\infty}(B(R))}}e^{-\alpha f_{m}}\frac{1}{N}\sum_{i=1}^{N}\big|x^{i}_{N} - y^{i}_{N}\big| + \alpha e^{-\alpha f_{m}}e^{\alpha |f|_{L_{\infty}(B(R))}}R L(R)\bigg(\frac{1}{N}\sum\limits_{i=1}^{N}|x^{i}_{N} - y^{i}_{N}|^{2}\bigg)^{1/2}  \\ & +  \alpha e^{-2\alpha f_{m}}e^{2\alpha |f|_{L_{\infty}(B(R))}}R L(R)\bigg(\frac{1}{N}\sum\limits_{i=1}^{N}|x^{i}_{N} - y^{i}_{N}|^{2}\bigg)^{1/2}
    \\ & \leq |x^{i}_{N} - y^{i}_{N}| + e^{\alpha (|f|_{L_{\infty}(B(R))} - f_{m})})\big( 1+ \alpha R L(R)+ \alpha R L(R) e^{\alpha (|f|_{L_{\infty}(B(R))} - f_{m})})\frac{1}{N^{1/2}}|\textbf{x}_{N} - \textbf{y}_{N}|.
\end{align*}
The second inequality directly follows from
\begin{align*}
    |F^{i}_{N}(\textbf{x}_{N})| \leq |x^{i}_{N}| + |\textbf{x}_{N}|.
\end{align*}
\end{proof}

\begin{theorem}\label{cbo_thrm_3.2}
Let the initial condition $\textbf{X}_{N}(0)$ of the jump-diffusion SDE (\ref{cbos1.6})  satisfy $\mathbb{E}|\textbf{X}_{N}(0)|^2  < \infty$ and  $\mathbb{E}|\Zstroke|^{2} < \infty$, then
the $Nd-$dimensional system (\ref{cbos1.6}) has a unique strong solution $\textbf{X}_{N}(t)$ under Assumptions~\ref{cboh3.1}-\ref{cboasu1.1}.
\end{theorem}
\begin{proof}
Note that 
$|G^{i}_{N}(\textbf{x}_{N}) - G^{i}_{N}(\textbf{y}_{N})| = |F^{i}_{N}(\textbf{x}_{N}) - F^{i}_{N}(\textbf{y}_{N})|$ and for all $i=1\dots,N$, 
\begin{align*} 
\int_{\mathbb{R}^{d}}|{F}^{i}_{N}(\textbf{x}_{N}){z}|^{2}\rho_{{z}}({z})d{z} &=\int_{\mathbb{R}^{d}}\sum\limits_{l=1}^{d}|(x_{N}^{i})_{l} - (y_{N}^{i})_{l}|^{2}|z^{}_{l}|^{2}\prod\limits_{k=1}^{d}\rho_{\zstroke}(z^{}_{k})d{z} \\ &= \sum\limits_{l=1}^{d}|(x_{N}^{i})_{l} - (y_{N}^{i})_{l}|^{2}\int_{\mathbb{R}^{d}}|z^{}_{l}|^{2}\prod\limits_{k=1}^{d}\rho_{\zstroke}(z^{}_{k})d{z} =
|{F}^{i}_{N}(\textbf{x}_{N})|^{2} \mathbb{E}(\Zstroke)^{2},
\end{align*}
where $(x^{i}_{N})_{l}$ means the $l-$th component of $d$-dimensional  vector $x^{i}_{N}$ and $z^{}_{l}$ means the $l-$th component of $d-$dimensional vector $z^{}$. Therefore, from Lemma~\ref{cbolemma3.1}, we can say that we have a positive function $K(R)$ of $R > 0$ such that
\begin{align*}
    |\textbf{F}_{N}(\textbf{x}_{N}) - \textbf{F}_{N}(\textbf{y}_{N}) |^{2} + |\textbf{G}_{N}(\textbf{x}_{N}) - \textbf{G}_{N}(\textbf{y}_{N}) |^{2}& + \sum_{i=1}^{N}\int_{\mathbb{R}^{d}}|\diag({F}^{i}_{N}(\textbf{x}_{N})-{F}^{i}_{N}(\textbf{y}_{N})){z}|^{2}\rho_{{z}}({z})d{z} \\ & \leq K(R) |\textbf{x}_{N}-\textbf{y}_{N}|,
\end{align*}
whenever $|\textbf{x}_{N}|$, $|\textbf{y}_{N}| \leq  R$. Moreover,
\begin{align*}
    |\textbf{F}_{N}(\textbf{x}_{N})|^{2} + |\textbf{G}_{N}(\textbf{x}_{N})|^{2} + \sum_{i=1}^{N}\int_{\mathbb{R}^{d}}|\diag({F}^{i}_{N}({x}_{N})){z}|^{2}\rho_{{z}}({z})d{z} \leq C|\textbf{x}_{N}|^{2},
\end{align*}
where $C$ is some positive constant independent of $|\textbf{x}_{N}| $.
Then the proof immediately follows from \cite[Theorem 1]{cbo19}. 

Consequently, by \cite[Lemma 2.3]{cbo15},  the following moment bound, provided $\mathbb{E}|\textbf{X}_{N}(0)|^{2p} <\infty$ and $\mathbb{E}|\textbf{Z}|^{2p} < \infty$, holds:
\begin{align}\label{cbo_eqn_3.2}
    \mathbb{E}\sup_{0\leq t\leq T}|\textbf{X}_{N}(t)|^{2p} \leq C_{N},
\end{align}
where $C_{N}$ may depend on $N$ and $p \geq 1$.
\end{proof}
In the last step of proof above, we highlighted that  $C_{N}$ may depend on $N$.  However, for convergence analysis in later sections we need an uniform in $N$ bound for $\sup_{i=1,\dots,N}\mathbb{E}\big(\sup_{t\in[0,T]}|X^{i}_{N}(t)|^{2p}\big)$, $p \geq 1$ which we prove under the following assumptions as in \cite{cbo2}.

\begin{assumption}\label{cboh3.2}
  There exists a positive constant $K_{f}$ such that
\begin{align*}
      |f(x) - f(y)| &\leq K_{f}(1+|x| + |y|)|x-y|, \;\;\text{for all}\;x, y , \in \mathbb{R}^{d}.
\end{align*}
\end{assumption}
\begin{assumption}\label{cboassu3.4}
There is a constant $K_{u} > 0$
\begin{align*}
f(x) - f_{m} &\leq K_{u}(1+ |x|^{2}), \;\; \text{for all}\; x \in \mathbb{R}^{d}.
\end{align*}
\end{assumption}
\begin{assumption}\label{cboasm1.4}
There exists constants $R>0$ and $K_{l} > 0$ such that
\begin{equation*}
    f(x) - f_{m} \geq K_{l}|x|^{2},\;\; |x|\geq R.
\end{equation*}
\end{assumption}
As one can see, we need a stronger Assumption~\ref{cboh3.2} as compared to Assumption~\ref{cboasu1.1} to obtain a moment bound uniform in $N$. The Assumptions~\ref{cboassu3.4}-\ref{cboasm1.4} are to make sure that objective function $f$ has quadratic growth at infinity. 

From \cite[Lemma 3.3]{cbo2}, we have the following result under Assumptions~\ref{cboh3.1}, \ref{cboh3.2}-\ref{cboasm1.4}:
\begin{align}\label{y4.2}
    \sum_{i=1}^{N}|x_{N}^{i}|^{2} \frac{e^{-\alpha f(x_{N}^{i})}}{\sum_{j=1}^{N}e^{-\alpha f(x_{N}^{j})}} \leq L_{1} + L_{2}\frac{1}{N}\sum_{i=1}^{N}|x_{N}^{i}|^{2},
\end{align}
where $L_{1} = R^{2}  + L_{2}$ and $L_{2} = 2\frac{K_{u}}{K_{l}}\Big(1 + \frac{1}{\alpha K_{l} R^{2}}\Big) $, $R$ is from Assumption~\ref{cboasm1.4}.

\begin{lemma}\label{cbolemma3.3}
Let Assumptions~\ref{cboh3.1}, \ref{cboh3.2}-\ref{cboasm1.4} be satisfied. Let $p\geq 1$, $\sup_{i=1,\dots,N}\mathbb{E}|X^{i}_{N}(0)|^{2p} < \infty $ and $\mathbb{E}|Z|^{2p} < \infty$. Then
\begin{equation*}
    \sup_{i\in\{1,\dots,N\}}\mathbb{E}\sup_{0\leq t\leq T}|X^{i}_{N}(t)|^{2p} \leq K_{m},
\end{equation*}
where $X_{N}^{i}(t)$ is from (\ref{cboeq1.8}) and $K_{m}$ is a positive constant independent of $N$.
\end{lemma}
\begin{proof}
Let $p$ be a positive integer. Using Ito's formula, we have
\begin{align*}
    |X_{N}^{i}(t)|^{2p} &= |X^{i}_{N}(0)|^{2p} -2p \mathbb{E}\int_{0}^{t}\beta(s)|X_{N}^{i}(s)|^{2p-2}\big(X_{N}^{i}(s)\cdot(X_{N}^{i}(s) - \bar{X}_{N}(s))\big)ds \\ & \;\;\;\;+
   2 \sqrt{2}p\int_{0}^{t}\sigma(s)|X^{i}_{N}(s)|^{2p-2}\big(X_{N}^{i}(s) \cdot \diag(X_{N}^{i}(s) - \bar{X}_{N}(s))dW^{i}(s)\big) \\ &
    \;\;\;\;+4p(p-1)\int_{0}^{t}\sigma^{2}(s)|X_{N}^{i}(s)|^{2p-4}|\diag(X_{N}^{i}(s)-\bar{X}_{N}(s))X_{N}^{i}(s)|^{2}ds \\ &\;\;\;\;
    +2 p\int_{0}^{t}\sigma^{2}(s)|X_{N}^{i}(s)|^{2p-2}|\diag(X_{N}^{i}(s) - \bar{X}_{N}(s)|^{2}ds \\ & \;\;\;\;
    + \int_{0}^{t}\int_{\mathbb{R}^{d}}\big(|X_{N}^{i}(s^{-}) + \gamma(s)\diag(X_{N}^{i}(s^{-}) - \bar{X}_{N}(s^{-}))z|^{2p} - |X_{N}^{i}(s^{-})|^{2p}\big)\mathcal{N}^{i}(ds,dz).
\end{align*}
First taking supremum over $0\leq t\leq T$ and then taking expectation, we get 
\begin{align}\label{cbo_eq_3.3}
    &\mathbb{E}\sup_{0 \leq t\leq T}|X^{i}_{N}(t)|^{2p} \leq  \mathbb{E}|X^{i}_{N}(0)|^{2p} + C \mathbb{E}\int_{0}^{T}|X_{N}^{i}(s)|^{2p-2}\big|X_{N}^{i}(s)\cdot(X_{N}^{i}(s) - \bar{X}_{N}(s))\big|ds \nonumber \\ & \;\;\;\; +      C\mathbb{E}\sup_{0 \leq t\leq T}\bigg|\int_{0}^{t}|X^{i}_{N}(s)|^{2p-2}\big(X_{N}^{i}(s) \cdot \diag(X_{N}^{i}(s) - \bar{X}_{N}(s))dW^{i}(s)\big)\bigg| 
    \nonumber\\ & \;\;\;\;+
     C\mathbb{E}\int_{0}^{T}|X_{N}^{i}(s)|^{2p-4}|\diag(X_{N}^{i}(s)-\bar{X}_{N}(s))X_{N}^{i}(s)|^{2}ds \nonumber\\ &\;\;\;\;
    + C\mathbb{E}\int_{0}^{T}|X_{N}^{i}(s)|^{2p-2}|\diag(X_{N}^{i}(s) - \bar{X}_{N}(s)|^{2}ds \nonumber \\ &
    \;\;\;\;+ C\mathbb{E}\sup_{0\leq t\leq T}\int_{0}^{t}\int_{\mathbb{R}^{d}}\big(|X_{N}^{i}(s^{-}) + \gamma(s)\diag(X_{N}^{i}(s^{-}) - \bar{X}_{N}(s^{-}))z|^{2p} - |X_{N}^{i}(s^{-})|^{2p}\big)\mathcal{N}^{i}(ds,dz). 
    \end{align}

 To deal with the second term in (\ref{cbo_eq_3.3}), we use Young's inequality and obtain 
    \begin{align*}
        |X_{N}^{i}(s)|^{2p-2}\big|X_{N}^{i}(s)\cdot(X_{N}^{i}(s) - \bar{X}_{N}(s))\big| 
         &\leq |X_{N}^{i}(s)|^{2p} + |X_{N}^{i}(s)|^{2p-1}|\bar{X}_{N}(s)|
         \\ & \leq \frac{4p-1}{2p}|X_{N}^{i}(s)|^{2p} + \frac{1}{2p}|\bar{X}_{N}(s)|^{2p}. 
    \end{align*}
To ascertain a bound on $|\bar{X}_{N}(s)|^{2p}$, we first apply Jensen's inequality  to $ |\bar{X}_{N}(s)|^{2}$ to get
\begin{equation*}
    |\bar{X}_{N}(s)|^{2} = \Bigg|\sum_{i = 1}^{N}X_{N}^{i}(s)\frac{e^{-\alpha f(X_{N}^{i}(s))}}{\sum_{j=1}^{N}e^{-\alpha f(X_{N}^{j}(s))}}\Bigg|^{2} \leq \sum_{i=1}^{N}|X_{N}^{i}(s)|^{2}\frac{e^{-\alpha f(X_{N}^{i}(s))}}{\sum_{j=1}^{N}e^{-\alpha f(X_{N}^{j}(s))}},
\end{equation*}
then using (\ref{y4.2}), we obtain $
    |\bar{X}_{N}(s)|^{2} \leq L_{1} + L_{2}\frac{1}{N}\sum\limits_{i=1}^{N}|X_{N}^{i}(s)|^{2},
$
which on applying the elementary inequality, $ (a + b )^{p} \leq 2^{p-1}(a^{p} + b^{p}), \; a,b \in \mathbb{R}_{+}$ and Jensen's inequality, gives
\begin{align*}
    |\bar{X}_{N}(s)|^{2p} \leq 2^{p-1}\Big(L_{1}^{p} + L_{2}^{p}\frac{1}{N}\sum\limits_{i=1}^{N}|X_{N}^{i}(s)|^{2p}\Big).
\end{align*}
As a consequence of the above calculations, we get
\begin{align}\label{cbo_eq_3.4}
    |X_{N}^{i}(s)|^{2p-2}\big|X_{N}^{i}(s)\cdot(X_{N}^{i}(s) - \bar{X}_{N}(s))\big|    \leq C\Big(1 + |X^{i}_{N}(s)|^{2p} + \frac{1}{N}\sum\limits_{i=1}^{N}|X_{N}^{i}(s)|^{2p}\Big),
\end{align}
where $C$ is a positive constant independent of $N$. 

Using the Burkholder-Davis-Gundy inequality, we get
\begin{align}
    \mathbb{E}&\sup_{0 \leq t\leq T}\bigg|\int_{0}^{t}|X^{i}_{N}(s)|^{2p-2}\big(X_{N}^{i}(s) \cdot \diag(X_{N}^{i}(s) - \bar{X}_{N}(s))dW^{i}(s)\big)\bigg|\nonumber \\ & \leq  \mathbb{E}\bigg(\int_{0}^{T} \big(|X^{i}_{N}(s)|^{2p-2}\big(X_{N}^{i}(s) \cdot \diag(X_{N}^{i}(s) - \bar{X}_{N}(s))\big)\big)^{2}ds\bigg)^{1/2} \nonumber \\ & \leq \mathbb{E}\Bigg(\sup_{0\leq t \leq T } |X_{N}^{i}(t)|^{2p-1}\bigg(\int_{0}^{T}|X_{N}^{i}(s) - \bar{X}_{N}(s))|^{2}ds\bigg)^{1/2}\Bigg),\nonumber 
\end{align}
which on applying generalized Young's inequality ($ab \leq (\epsilon a^{q_{1}})/q_{1} + b^{q_{2}}/(\epsilon^{q_{2}/q_{1}}q_{2}),\; \epsilon, q_{1}, q_{2} >0, 1/q_{1} + 1/q_{2} = 1$)  yields
\begin{align}
    \mathbb{E}&\sup_{0 \leq t\leq T}\bigg|\int_{0}^{t}|X^{i}_{N}(s)|^{2p-2}\big(X_{N}^{i}(s) \cdot \diag(X_{N}^{i}(s) - \bar{X}_{N}(s))dW^{i}(s)\big)\bigg|\nonumber \\ & \leq \frac{1}{2}\mathbb{E}\sup_{0\leq t\leq T}|X^{i}_{N}(t)|^{2p} + C\mathbb{E}\bigg(\int_{0}^{T}|X_{N}^{i}(s) - \bar{X}_{N}(s))|^{2}ds\bigg)^{p}\nonumber \\ &  \leq \frac{1}{2}\mathbb{E}\sup_{0\leq t\leq T}|X^{i}_{N}(t)|^{2p} +  C\mathbb{E}\bigg(\int_{0}^{T}|X_{N}^{i}(s) - \bar{X}_{N}(s))|^{2p}ds\bigg),\label{cbo_eq_3.5}
\end{align}
where in the last step we have utilized Holder's inequality.

Now, we move on to obtain estimates which are required to deal with fourth and fifth term in (\ref{cbo_eq_3.3}). Using Young's inequality, we have
\begin{align}
 A_{1} :=  |X_{N}^{i}(s)|^{2p-4}(|X_{N}^{i}(s)|^{2} &- (X_{N}^{i}(s)\cdot\bar{X}_{N}(s)))^{2} 
    \leq 2|X_{N}^{i}(s)|^{2p} + 2|X_{N}^{i}(s)|^{2p-2}|\bar{X}_{N}(s)|^{2}\nonumber  \\ &
    \leq \frac{4p-2}{p}|X_{N}^{i}(s)|^{2p} + \frac{2}{p}|\bar{X}_{N}(s)|^{2p}.
\end{align}
In the same way, applying Young's inequality, we obtain
\begin{align}
A_{2} :=     |X_{N}^{i}(s)|^{2p-2}|\diag(X_{N}^{i}(s) &- \bar{X}_{N}(s))|^{2}  \leq 2|X_{N}^{i}(s)|^{2p} + 2|X_{N}^{i}(s)|^{2p-2}|\bar{X}_{N}(s)|^{2} \nonumber \\ & \leq \frac{4p-2}{p}|X_{N}^{i}(s)|^{2p} + \frac{2}{p}|\bar{X}_{N}(s)|^{2p}.
\end{align}
Following the same procedure based on (\ref{y4.2}), which we followed to obtain  bound (\ref{cbo_eq_3.4}), we also get
\begin{align}\label{cbo_eq_3.8}
A_{1} + A_{2} \leq C\Big(1 + |X_{N}^{i}(s)|^{2p} + \frac{1}{N}\sum\limits_{i=1}^{N}|X_{N}^{i}(s)|^{2p} \Big),   
\end{align}
where $C$ is a positive constant independent of $N$.

It is left to deal with the last term in (\ref{cbo_eq_3.3}). Using the Cauchy-Bunyakowsky-Schwartz inequality, we get
\begin{align*}
   &\mathbb{E}\sup_{0\leq t\leq T}\int_{0}^{t}\int_{\mathbb{R}^{d}}\big(|X_{N}^{i}(s^{-}) + \gamma(s)\diag(X_{N}^{i}(s^{-}) - \bar{X}_{N}(s^{-}))z|^{2p} - |X_{N}^{i}(s^{-})|^{2p}\big)\mathcal{N}^{i}(ds,dz)
 \\ & \leq \mathbb{E}\sup_{0\leq t\leq T}\int_{0}^{t}\int_{\mathbb{R}^{d}}\bigg(2^{2p-1}\big(|X_{N}^{i}(s^{-})|^{2p} + |\gamma(s)\diag(X_{N}^{i}(s^{-}) - \bar{X}_{N}(s^{-}))z|^{2p}\big) - |X_{N}^{i}(s^{-})|^{2p}\bigg)\mathcal{N}^{i}(ds,dz)
\\ &  \leq C\mathbb{E}\int_{0}^{T}\int_{\mathbb{R}^{d}}\big(|X_{N}^{i}(s^{-})|^{2p} + |\gamma(s)\diag(X_{N}^{i}(s^{-}) - \bar{X}_{N}(s^{-}))z|^{2p}\big) \mathcal{N}^{i}(ds,dz)
\\  & \leq C\mathbb{E}\int_{0}^{T}\int_{\mathbb{R}^{d}}(|X_{N}^{i}(s)|^{2p} + |\gamma(s)\diag(X_{N}^{i}(s)-\bar{X}_{N}(s))z|^{2p}\big)\rho_{z}(z)dz 
   \\ &
   \leq C\mathbb{E}\int_{0}^{T}\Big(|X_{N}^{i}(s)|^{2p}  + |X_{N}^{i}(s) - \bar{X}_{N}(s)|^{2p}\int_{\mathbb{R}^{d}}|z|^{2p}\rho_{z}(z)dz\Big)ds.
\end{align*}
We have
\begin{align*}
    |X_{N}^{i}(s) - \bar{X}_{N}(s)|^{2p} &\leq 2^{2p-1}\big(|X_{N}^{i}(s)|^{2p} + |\bar{X}_{N}^{i}(s)|^{2p}\big) \leq  C\Big( 1 + |X_{N}^{i}(s)|^{2p} + \frac{1}{N}\sum\limits_{i=1}^{N}|X_{N}^{i}(s)|^{2p}\Big),
\end{align*}
and hence
\begin{align}
   &\mathbb{E}\sup_{0\leq t\leq T}\int_{0}^{t}\int_{\mathbb{R}^{d}}\big(|X_{N}^{i}(s^{-}) + \gamma(s)\diag(X_{N}^{i}(s^{-}) - \bar{X}_{N}(s^{-}))z|^{2p} - |X_{N}^{i}(s^{-})|^{2p}\big)\mathcal{N}^{i}(ds,dz) \nonumber \\ & \leq C\mathbb{E}\int_{0}^{T}\Big(1+ |X_{N}^{i}(s)|^{2p} + \frac{1}{N}\sum\limits_{i=1}^{N}|X_{N}^{i}(s)|^{2p}\Big)ds,\label{cbo_eq_3.9}
\end{align}
where $C >0$ does not depend on $N$.

Using (\ref{cbo_eq_3.4}), (\ref{cbo_eq_3.5}), (\ref{cbo_eq_3.8}) and (\ref{cbo_eq_3.9}) in (\ref{cbo_eq_3.3}), we get
\begin{align*}
    \frac{1}{2}\mathbb{E}\sup_{0\leq t\leq T}|X_{N}^{i}(t)|^{2p} &\leq \mathbb{E}|X_{N}^{i}(0)|^{2p} +  C\mathbb{E}\int_{0}^{T}\Big(1 +  |X_{N}^{i}(s)|^{2p} + \frac{1}{N}\sum\limits_{i=1}^{N}|X_{N}^{i}(s)|^{2p}\Big)ds 
\end{align*}
and
\begin{align*}
     \mathbb{E}\sup_{0\leq t\leq T}|X_{N}^{i}(t)|^{2p} &\leq 2\mathbb{E}|X_{N}^{i}(0)|^{2p} +  C\mathbb{E}\int_{0}^{T}\Big(1 +  \sup_{0\leq u\leq s}|X_{N}^{i}(u)|^{2p} + \frac{1}{N}\sum\limits_{i=1}^{N}\sup_{0\leq u\leq s}|X_{N}^{i}(u)|^{2p}\Big)ds.
    \end{align*}
Taking supremum over $\{1,\dots,N\}$, we obtain  
    \begin{align*} 
     \sup_{i=1,\dots,N}\mathbb{E}\sup_{0\leq t\leq T}|X_{N}^{i}(t)|^{2p} &\leq  2\sup_{i=\{1,\dots,N\}}\mathbb{E}|X_{N}^{i}(0)|^{2p} + C \bigg(1 + \int_{0}^{T}\sup_{i = 1,\dots,N}\mathbb{E}\sup_{0\leq u \leq s}|X_{N}^{i}(u)|^{2p} ds\bigg),
\end{align*}
which gives our targeted result for positive integer valued $p$ by applying Gr\"{o}nwall's lemma (note that we can apply Gr\"{o}nwall's lemma due to (\ref{cbo_eqn_3.2})). We can extend the result to non-integer values of $p \geq 1$ using Holder's inequality. 
\end{proof}

\subsection{Well-posedness of mean-field jump-diffusion SDEs}
\label{sec_well_pos_2}
In this section, we first introduce Wasserstein metric and state Lemma~\ref{cboblw} which is crucial for establishing well-posedness of the mean-field limit. Then, we  prove existence and uniqueness of the McKean-Vlasov jump-diffusion SDEs (\ref{cbomfsde}) in Theorem~\ref{mf_wel_pos_th}.  

Let $\mathbb{D}([0,T];\mathbb{R}^{d})$ be the space of $\mathbb{R}^{d}$ valued c\'{a}dl\'{a}g functions and $\mathcal{P}_{p}(\mathbb{R}^{d}),\; p\geq 1$, be the space of probability measures on the measurable space $(\mathbb{R}^{d},\mathcal{B}(\mathbb{R}^{d}))$ such that for any $\mu \in \mathcal{P}_{p}(\mathbb{R}^{d})$,  $\int_{\mathbb{R}^{d}}|x|^{p}\mu(dx)<  \infty$, and which is equipped with the  $p$-Wasserstein metric 
\begin{equation*}
    \mathcal{W}_{p}(\mu,\vartheta) := \inf_{\pi \in \prod(\mu,\vartheta)}\Big( \int_{\mathbb{R}^{d}\times \mathbb{R}^{d}}|x-y|^{p}\pi(dx,dy)\Big)^{\frac{1}{p}},
\end{equation*}
where $\prod(\mu,\vartheta)$ is the set of couplings of $\mu,\vartheta \in \mathcal{P}_{p}(\mathbb{R}^{d})$ \cite{cbo33}.

Let $\mu \in \mathcal{P}_{2}(\mathbb{R}^{d})$ with $\int_{\mathbb{R}^{d}}|x|^{2}\mu(dx) \leq K$. Then, using Jensen's inequality, we have
\begin{align*}
    e^{-\alpha \int_{\mathbb{R}^{d}}f(x)\mu(dx) } \leq \int_{\mathbb{R}^{d}}e^{-\alpha f(x)}\mu(dx), 
\end{align*}
and the simple rearrangement together with  Assumption~\ref{cboassu3.4}, gives
\begin{align}\label{cbol3.4}
    \frac{e^{-\alpha f_{m}}}{\int_{\mathbb{R}^{d}}e^{-\alpha f(x)}\mu(dx)} \leq e^{\alpha(\int_{\mathbb{R}^{d}}f(x)\mu(dx) - f_{m})} \leq e^{\alpha K_{u}\int_{\mathbb{R}^{d}}(1 + |x|^{2})\mu(dx)} \leq  C_{K},
\end{align}
where $C_{K} > 0$ is a constant.  We will also need the following notation:
\begin{align*}
    \bar{X}^{\mu} = \frac{\int_{\mathbb{R}^{d}} xe^{-\alpha f(x)}\mu(dx)}{\int_{\mathbb{R}^{d}}e^{-\alpha f(x)}\mu(dx)},
\end{align*}
where $\mu \in \mathcal{P}_{4}(\mathbb{R}^{d})$.

The next lemma is required for proving well-posedness of the McKean-Vlasov SDEs (\ref{cbomfsdep}). Its proof is available in \cite[Lemma 3.2]{cbo2}. 
\begin{lemma}\label{cboblw}
Let Assumptions~\ref{cboh3.1}, \ref{cboh3.2}-\ref{cboasm1.4} hold and  there exists a constant $K>0$ such that  $\int |x|^{4}\mu(dx) \leq K$ and $\int |y|^{4} \vartheta(dy) \leq K$ for all $\mu,\vartheta \in \mathcal{P}_{4}(\mathbb{R}^{d})$, then the following inequality is satisfied: 
\begin{equation*}
    |\bar{X}^{\mu} - \bar{X}^{\vartheta}| \leq C\mathcal{W}_{2}(\mu,\vartheta),
\end{equation*}
where $C>0$ is independent of $\mu$ and $\vartheta$.
\end{lemma}
\begin{theorem}\label{mf_wel_pos_th}
Let Assumptions~\ref{cboh3.1}, \ref{cboh3.2}-\ref{cboasm1.4} hold, and let $\mathbb{E}|X(0)|^{4} < \infty $  and $\int_{\mathbb{R}^{d}}|z|^{4}\rho_{z}(z)dz < \infty$. Then, there exists a unique nonlinear process $X \in \mathbb{D}([0,T];\mathbb{R}^{d})$, $T>0$ which satisfies the McKean-Vlasov SDEs (\ref{cbomfsdep}) in the strong sense.
\end{theorem}
\begin{proof}
Let $v \in C([0,T];\mathbb{R}^{d})$. Consider the following SDEs:
\begin{align}
       dX_{v}(t) &= -\beta(t)(X_{v}(t) - v(t))dt  + \sigma(t)\diag(X_{v}(t) - v(t))dW(t) \nonumber \\ & \;\;\;\;+ \gamma(t)\int_{\mathbb{R}^{d}}\diag(X_{v}(t^{-}) - v(t)))z\mathcal{N}(dt,dz) \label{cbo_neweq_3.14}
\end{align}
for any $t \in[0,T]$.

Note that $v(t)$ is a deterministic function of $t$, therefore the coefficients of SDEs (\ref{cbo_neweq_3.14}) only depend on $x$ and $t$. The coefficients are globally Lipschitz continuous and have linear growth in $x$. The existence and uniqueness of a process $X_{v} \in \mathbb{D}([0,T];\mathbb{R}^{d})$ satisfying SDEs  with  L\'{e}vy noise (\ref{cbo_neweq_3.14}) follows from \cite[pp. 311-312]{cbos11}.  We also have $\int_{\mathbb{R}^{d}}|x|^{4}\mathcal{L}_{X_{v}(t)}(dx) = \mathbb{E}|X_{v}(t)|^{4} \leq \sup_{t\in[0,T]}\mathbb{E}|X_{v}(t)|^{4} \leq K$, where $K$ is a positive constant depending on $v$ and $T$, and  $\mathcal{L}_{X_{v}(t)}$ represents the law of $X_{v}(t)$. 

We define a mapping 
\begin{align}
\mathbb{T} : C([0,T];\mathbb{R}^{d}) \rightarrow C([0,T];\mathbb{R}^{d}),\;\;\mathbb{T}(v)  = \bar{X}_{v},
\end{align}
where
\begin{align*}
\mathbb{T}v(t) & = \bar{X}_{v}(t) = \mathbb{E}(X_{v}(t)e^{-\alpha f(X_{v}(t))})\Big/\mathbb{E}(e^{-\alpha f(X_{v}(t))}) \\ & =  \int_{\mathbb{R}^{d}}xe^{-\alpha f(x)}\mathcal{L}_{X_{v}(t)}(dx) \bigg/\int_{\mathbb{R}^{d}}e^{-\alpha f(x)}\mathcal{L}_{X_{v}(t)}(dx)= \bar{X}^{\mathcal{L}_{X_{v}(t)}}(t).
\end{align*}

Let $\delta \in (0,1)$. For all $t, t+\delta \in (0,T)$, Ito's isometry provides
\begin{align}
    \mathbb{E}|X_{v}(t + \delta) - X_{v}(t)|^{2} &\leq C\int_{t}^{t+\delta}\mathbb{E}|X_{v}(s) - v(s)|^{2}ds   \nonumber  \\ & \;\;\;\;+ \int_{t}^{t+\delta}\int_{\mathbb{R}^{d}}\mathbb{E}|X_{v}(s) - v(s)|^{2}|z|^{2}\rho(z)dzds \leq C \delta, \label{cbo_neweq_3.17}
\end{align}
where $C$ is a positive constant independent of $\delta$. Using Lemma~\ref{cboblw} and (\ref{cbo_neweq_3.17}), we obtain
\begin{align*}
    |\bar{X}_{v}(t+\delta ) - \bar{X}_{v}(t)| &= |\bar{X}^{\mathcal{L}_{X_{v}(t+\delta)}}(t+\delta) - \bar{X}^{\mathcal{L}_{X_{v}(t)}}(t)| \leq C\mathcal{W}_{2}(\mathcal{L}_{X_{v}(t+\delta)}, \mathcal{L}_{X_{v}(t)}) \\ & \leq C\big(\mathbb{E}|X_{v}(t+\delta) - X_{v}(t)|^{2}\big)^{1/2} \leq C|\delta|^{1/2},
\end{align*}
 where $C$ is a positive constant independent $\delta$. This implies the H\"{o}lder continuity of the map $t \rightarrow \bar{X}_{v}(t)$. Therefore, the compactness of $\mathbb{T}$ follows from the compact embedding $C^{0,\frac{1}{2}}([0,T];\mathbb{R}^{d}) \hookrightarrow C([0,T];\mathbb{R}^{d}) $.
 
 Using Ito's isometry, we have
\begin{align}
    \mathbb{E}|X_{v}(t)|^{2} &\leq  4\bigg(\mathbb{E}|X_{v}(0)|^{2} + \mathbb{E}\bigg|\int_{0}^{t}\beta(s)(X_{v}(s) - v(s))ds\bigg|^{2} +  \mathbb{E}\bigg|\int_{0}^{t}\sigma(s)\diag(X_{v}(s) - v(s))dW(s)\bigg|^{2} \nonumber \\ & \;\;\;\; + \mathbb{E}\bigg|\int_{0}^{t}\gamma(s)\diag(X_{v}(s^-) - v(s))z\mathcal{N}(ds,dz)\bigg|^{2}\bigg) \nonumber \\ & \leq  C\bigg(1 + \int_{0}^{t}\mathbb{E}|X_{v}(s) - v(s)|^{2}ds\bigg) \leq C\bigg(1+ \int_{0}^{t}(\mathbb{E}|X_{v}(s)|^{2} + |v(s)|^{2}) ds\bigg), \label{cbo_eq_3.17}
\end{align}
where $C$ is a positive constant independent of $v$. 
Moreover, we have the following result under Assumptions~\ref{cboh3.1}, \ref{cboh3.2}-\ref{cboasm1.4} \cite[Lemma 3.3]{cbo2}:
\begin{align}
    |\bar{X}_{v}(t)|^{2} \leq L_{1} + L_{2}\mathbb{E}|X_{v}(t)|^{2}, \label{cbo_neweq_3.18}
\end{align}
where $L_{1}$ and $L_{2}$ are from (\ref{y4.2}). 
Consider a set $\mathcal{S} = \{ v\in C([0,T];\mathbb{R}^{d}) : v = \epsilon \mathbb{T}v, \; 0\leq \epsilon \leq 1\} $. The set $\mathcal{S}$ is non-empty due to the fact that $\mathbb{T}$ is compact (see the remark after Theorem~10.3 in \cite{104}).  Therefore, for any $v \in \mathcal{S}$, we have the corresponding unique process $X_{v}(t) \in \mathbb{D}([0,T];\mathbb{R}^{d})$ satisfying (\ref{cbo_neweq_3.14}), and  $\mathcal{L}_{X_{v}(t)}$ represents the law of $X_{v}(t)$, such that the following holds due to (\ref{cbo_neweq_3.18}):
\begin{align}
    |v(s)|^{2} = \epsilon^{2} |\mathbb{T}v(s)|^{2} = \epsilon^{2} |\bar{X}_{v}(s)|^{2} \leq  \epsilon^{2}  \big(L_{1} + L_{2}\mathbb{E}|X(s)|^{2}) \label{cbo_neweq_3.19}
\end{align}
for all $s \in [0,T]$.  Substituting (\ref{cbo_neweq_3.19}) in (\ref{cbo_eq_3.17}), we get
\begin{align*}
    \mathbb{E}|X_{v}(t)|^{2} \leq C\bigg(1+\int_{0}^{t}\mathbb{E}|X_{v}(s)|^{2}ds\bigg),
\end{align*}
which on applying Gr\"{o}nwall's lemma gives
\begin{align}
    \mathbb{E}|X_{v}(t)|^{2} \leq C, \label{cbo_neweq_3.20}
\end{align}
where $C$ is independent of $v$. Due to (\ref{cbo_neweq_3.19}) and (\ref{cbo_neweq_3.20}), we can claim the boundedness of the set $\mathcal{S}$. Therefore, from the Leray-Schauder theorem \cite[Theorem~10.3]{104} there exists a fixed point of the mapping $\mathbb{T}$. This proves existence of the solution of (\ref{cbomfsdep}).

Let $v_{1}$ and $v_{2}$ be two fixed points of the mapping $\mathbb{T}$ and let us denote the corresponding solutions of (\ref{cbo_neweq_3.14}) as $X_{v_{1}}$ and $X_{v_{2}}$. Using Ito's isometry, we can get
\begin{align}
    \mathbb{E}|X_{v_{1}}(t) - X_{v_{2}}(t)|^{2} \leq \mathbb{E}|X_{v_{1}}(0) - X_{v_{2}}(0)|^{2} + C\int_{0}^{t}\big(\mathbb{E}|X_{v_{1}}(s) -X_{v_{2}}(s)|^{2} + |v_{1}(s) - v_{2}(s)|^{2}\big)ds.  \label{cbo_neweq_3.21}
\end{align}
Note that $\mathcal{S}$ is a bounded set and by definiiton $v_{1}$ and $v_{2}$ belong to $\mathcal{S}$. Then, we can apply Lemma~\ref{cboblw} to ascertain 
\begin{align*}
    |v_{1}(s) - v_{2}(s)|^{2} = |\bar{X}_{v_{1}}(s) - \bar{X}_{v_{2}}(s)|^{2} \leq C\mathcal{W}_{2}(\mathcal{L}_{X_{v_{1}}(s)} , \mathcal{L}_{X_{v_{2}}(s)}) \leq C \mathbb{E}|X_{v_{1}}(s) - X_{v_{2}}(s)|^{2}.
\end{align*}
Using the above estimate, Gr\"{o}nwall's lemma and the fact $X_{v_{1}}(0) = X_{v_{2}}(0)$ in (\ref{cbo_neweq_3.21}), we get uniqueness of the solution of (\ref{cbomfsdep}).
\end{proof}
\begin{theorem}\label{cbolem3.6}
Let Assumptions~\ref{cboh3.1}, \ref{cboh3.2}-\ref{cboasm1.4} are satisfied. Let $p\geq 1$, $\mathbb{E}|X(0)|^{2p} < \infty $ and $\mathbb{E}|Z|^{2p}< \infty$, then the following holds:
\begin{align*}
   \mathbb{E} \sup_{0\leq t \leq T}|X(t)|^{2p} \leq K_{p},
\end{align*}
where $X(t)$ satisfies (\ref{cbomfsdep}) and $K_{p}$ is a positive constant.
\end{theorem}
\begin{proof}
Recall that under the assumptions of this theorem, Theorem~\ref{mf_wel_pos_th} guarantees existence of a strong solution of (\ref{cbomfsdep}).

Let $p$ be a positive integer. Let us denote $ \theta_{R} = \inf\{s \geq 0\; ; \; |X(s)| \geq R\}$. Using Ito's formula, we obtain
\begin{align}
    |X(t)|^{2p} &= |X(0)|^{2p} - 2p \int_{0}^{t}\beta(s)|X(s)|^{2p-2}\big(X(s)\cdot(X(s) - \bar{X}(s))\big)ds \nonumber \\ &
    \;\;\;\; + 2\sqrt{2}p \int_{0}^{t}\sigma(s)|X(s)|^{2p-2}\big(X(s)\cdot(\diag(X(s)- \bar{X}(s))dW(s))\big) \nonumber\\ &
    \;\;\;\; + 4p(p-1)\int_{0}^{t}\sigma^{2}(s)|X(s)|^{2p-4}|\diag(X(s) -\bar{X}(s))X(s)|^{2}ds \nonumber\\ & 
   \;\;\;\; + 2p\int_{0}^{t}\sigma^{2}(s)|X(s)|^{2p-2}|\diag(X(s) - \bar{X}(s))|^{2}ds \nonumber\\ &
    \;\;\;\; + \int_{0}^{t}\int_{\mathbb{R}^{d}}(|X(s^{-}) + \gamma(s)\diag(X(s^{-})-\bar{X}(s^{-}))z|^{2p} - |X(s^{-})|^{2p})\mathcal{N}(ds,dz).
   \nonumber
\end{align} 
First taking suprema over $0\leq t\leq T\wedge \theta_{R}$ and then taking expectation on both sides, we get
\begin{align}
\mathbb{E}&\sup_{0\leq t\leq T \wedge \theta_{R}}|X(t)|^{2p} \leq  \mathbb{E}|X(0)|^{2p} + C\mathbb{E}\int_{0}^{T\wedge \theta_{R}}|X(s)|^{2p-2}\big|X(s)\cdot(X(s) - \bar{X}(s))\big|ds \nonumber \\ & \;\;\;\; + C\mathbb{E}\sup_{0\leq t\leq T\wedge\theta_{R}}\bigg|\int_{0}^{t}|X(s)|^{2p-2}\big(X(s)\cdot(\diag(X(s)- \bar{X}(s))dW(s))\big)\bigg| \nonumber \\ &
    \;\;\;\;+ C\mathbb{E}\int_{0}^{T\wedge \theta_{R}}|X(s)|^{2p-4}|\diag(X(s) -\bar{X}(s))X(s)|^{2}ds\nonumber
    \\ &  \;\;\;\; +C\mathbb{E}\int_{0}^{T\wedge \theta_{R}}|X(s)|^{2p-2}|\diag(X(s) - \bar{X}(s))|^{2}ds \nonumber\\ &
    \;\;\;\;+ \mathbb{E}\sup_{0\leq t\leq T\wedge \theta_{R}}\int_{0}^{t}\int_{\mathbb{R}^{d}}(|X(s^{-}) + \gamma(s)\diag(X(s^{-})-\bar{X}(s^{-}))z|^{2p} - |X(s^{-})|^{2p})\mathcal{N}(ds,dz).\label{w3.5}
\end{align}
To deal with the second term in (\ref{w3.5}), we use Young's inequality and ascertain
\begin{align}
    &|X(s)|^{2p-2}\big|X(s)\cdot(X(s) -\bar{X}(s))\big|
     \leq  |X(s)|^{2p} + |X(s)|^{2p-1}|\bar{X}(s)| \nonumber \\&
     \leq \frac{4p-1}{2p}|X(s)|^{2p} + \frac{1}{2p}|\bar{X}(s)|^{2p} \leq C(|X(s)|^{2p} + |\bar{X}(s)|^{2p}).\label{cbo_eq_3.14}
\end{align}

Using Burkholder-Davis-Gundy inequality, we have
\begin{align}
    \mathbb{E}&\sup_{0\leq t\leq T\wedge\theta_{R}}\bigg|\int_{0}^{t}|X(s)|^{2p-2}\big(X(s)\cdot(\diag(X(s)- \bar{X}(s))dW(s))\big)\bigg| \nonumber \\ & \leq 
     \mathbb{E}\bigg(\int_{0}^{T\wedge \theta_{R}}|X(s)|^{4p-2}|X(s)- \bar{X}(s)|^{2}ds\bigg)^{1/2} \nonumber  \\ & \leq \mathbb{E}\Bigg(\sup_{0\leq t \leq T \wedge \theta_{R}} |X(t)|^{2p-1}\bigg(\int_{0}^{T\wedge \theta_{R}}|X(s) - \bar{X}(s))|^{2}ds\bigg)^{1/2}\Bigg).\label{cbo_eq_3.15}
\end{align}
We apply generalized Young's inequality $\big(ab \leq (\epsilon a^{q_{1}})/q_{1} + b^{q_{2}}/(\epsilon^{q_{2}/q_{1}}q_{2}),\; \epsilon, q_{1},q_{2} >0, 1/q_{1} + 1/q_{2} = 1$\big) and Holder's inequality on the right hand side of (\ref{cbo_eq_3.15}) to get
\begin{align}
    \mathbb{E}&\sup_{0 \leq t\leq T\wedge \theta_{R}}\bigg|\int_{0}^{t}|X(s)|^{2p-2}\big(X(s) \cdot \diag(X(s) - \bar{X}(s))dW(s)\big)\bigg|\nonumber \\ & \leq \frac{1}{2}\mathbb{E}\sup_{0\leq t\leq T\wedge \theta_{R}}|X(t)|^{2p} + C\mathbb{E}\bigg(\int_{0}^{T\wedge \theta_{R}}|X(s) - \bar{X}(s)|^{2}ds\bigg)^{p}\nonumber \\ &  \leq \frac{1}{2}\mathbb{E}\sup_{0\leq t\leq T\wedge \theta_{R}}|X(t)|^{2p} + C\mathbb{E}\bigg(\int_{0}^{T\wedge \theta_{R}}|X(s) - \bar{X}(s)|^{2p}ds\bigg)\nonumber \\ & 
   \leq  \frac{1}{2}\mathbb{E}\sup_{0\leq t\leq T\wedge \theta_{R}}|X(t)|^{2p} +  C\mathbb{E}\bigg(\int_{0}^{T\wedge \theta_{R}} \big(|X(s)|^{2p} + |\bar{X}(s)|^{2p} \big) ds\bigg).\label{cbo_eq_3.16}
\end{align}

We have the following estimate to use in the fourth term in (\ref{w3.5}):
\begin{align}
    |X(s)|^{2p-4}&|\diag(X(s)- \bar{X}(s))X(s)|^{2}  \leq |X(s)|^{2p-4}(|X(s)|^{2} + (X(s)\cdot\bar{X}(s)))^{2}
    \nonumber \\ &\leq 2|X(s)|^{2p} + 2|X(s)|^{2p-2}|\bar{X}(s)|^{2} \leq C\big(|X(s)|^{2p} + |\bar{X}(s)|^{2p}\big).\label{w3.8}
\end{align}

We make use of Minkowski's inequality to get
\begin{align*}
  |X(s)|^{2p-2}|\diag(X(s) - \bar{X}(s))|^{2} = |X(s)|^{2p-2}|X(s) - \bar{X}(s)|^{2}  \leq 2|X(s)|^{2p} + 2|X(s)|^{2p-2}|\bar{X}(s)|^{2},
\end{align*}
then Young's inequality implies
\begin{align}
    |X(s)|^{2p-2}|X(s) - \bar{X}(s)|^{2}  \leq C(|X(s)|^{2p} + |\bar{X}(s)|^{2p}). \label{w3.9}
\end{align}

Now, we find an estimate for the last term in (\ref{w3.5}).  Using the Cauchy-Bunyakowsky-Schwartz inequality, we obtain
\begin{align}
   \mathbb{E}&\sup_{0\leq t\leq T\wedge \theta_{R}}\int_{0}^{t}\int_{\mathbb{R}^{d}}(|X(s^{-}) + \gamma(s)\diag(X(s^{-})-\bar{X}(s^{-}))z|^{2p} - |X(s^{-})|^{2p})\mathcal{N}(ds,dz)\nonumber \\ & \leq \mathbb{E}\sup_{0\leq t\leq T\wedge \theta_{R}}\int_{0}^{t}\int_{\mathbb{R}^{d}}2^{2p-1}(|X(s^{-})|^{2p} +  |\gamma(s)\diag(X(s^{-})-\bar{X}(s^{-}))z|^{2p})- |X(s^{-})|^{2p}\mathcal{N}(ds,dz) \nonumber 
   \\ & \leq  C\mathbb{E}\int_{0}^{T\wedge \theta_{R}}\int_{\mathbb{R}^{d}}(|X(s^{-})|^{2p} + |\gamma(s)\diag(X(s^{-})-\bar{X}(s^{-}))z|^{2p})\mathcal{N}(ds,dz). \nonumber
\end{align} 
Using Doob's optional stopping theorem \cite[Theorem 2.2.1]{cbos11}, we get
\begin{align}
   \mathbb{E}&\sup_{0\leq t\leq T\wedge \theta_{R}}\int_{0}^{t}\int_{\mathbb{R}^{d}}(|X(s^{-}) + \gamma(s)\diag(X(s^{-})-\bar{X}(s^{-}))z|^{2p} - |X(s^{-})|^{2p})\mathcal{N}(ds,dz)\nonumber  \\ & \leq  C\mathbb{E}\int_{0}^{T\wedge \theta_{R}}\int_{\mathbb{R}^{d}}(|X(s)|^{2p} + |\gamma(s)\diag(X(s)-\bar{X}(s))z|^{2p})\rho_{z}(z)dzds \nonumber \\ & \leq  C\mathbb{E}\int_{0}^{T\wedge \theta_{R}}\Big(|X(s)|^{2p} + |\bar{X}(s)|^{2p}\Big)\Big(1+\int_{\mathbb{R}^{d}}|z|^{2p}\rho_{z}(z)dz\Big)ds \nonumber 
 \\ &
  \leq C\mathbb{E}\int_{0}^{T\wedge \theta_{R}}\big(|X(s)|^{2p} + |\bar{X}(s)|^{2p}\big)ds. \label{w3.10}
\end{align}

We have the following result under Assumptions~\ref{cboh3.1}, \ref{cboh3.2}-\ref{cboasm1.4} \cite[Lemma 3.3]{cbo2}:
\begin{align}
    |\bar{X}(s)|^{2} \leq L_{1} + L_{2}\mathbb{E}|X(s)|^{2}, \label{cbo_neweq_3.29}
\end{align}
where $L_{1}$ and $L_{2}$ are from (\ref{y4.2}). 

Substituting (\ref{cbo_eq_3.14}), (\ref{cbo_eq_3.16})-(\ref{cbo_neweq_3.29}) in (\ref{w3.5}), using Holder's inequality, we arrive at the following bound:
\begin{align*}
    \mathbb{E}\sup_{0\leq t\leq T\wedge \theta_{R}}|X(t)|^{2p} &\leq 2\mathbb{E}|X(0)|^{2p} +  C \mathbb{E}\int_{0}^{T\wedge \theta_{R}}(|X(s)|^{2p} + |\bar{X}(s)|^{2p})ds  
       \\ & \leq C  + C\mathbb{E}\int_{0}^{T\wedge \theta_{R}}(1 + |X(s)|^{2p} + \mathbb{E}|X(s)|^{2p})ds 
       \\ & \leq C  + C\int_{0}^{T} \mathbb{E}\sup_{0\leq u\leq s \wedge \theta_{R}}|X(u)|^{2p} ds,
\end{align*}
which on using Gr\"{o}nwall's lemma gives
\begin{align*}
    \mathbb{E}\sup_{0\leq t\leq T\wedge \theta_{R}}|X(t)|^{2p} \leq C,
\end{align*}
where $C$ is independent of $R$. Then, tending $R\rightarrow \infty$ and applying Fatau's lemma give the desired result.
\end{proof}

\section{Convergence results}\label{cbo_conv_res}
In Section \ref{cbo_sec_gl_min}, we prove the convergence of $X(t)$, which is the mean field limit of the particle system (\ref{cboeq1.8}), towards global minimizer.  This convergence proof is based on the Laplace principle. Our approach in Section~\ref{cbo_sec_gl_min} is similar to \cite[Appendix A]{cbo3}. The main result (Theorem~\ref{cbo_thrm_4.3}) of Section~\ref{cbo_sec_gl_min} differs from \cite{cbo3} in three respects. First, in our model (\ref{cboeq1.8}),  the parameters are time-dependent. Second, we need to treat the jump part of (\ref{cboeq1.8}). Third, the analysis in \cite{cbo3} is done for quadratic loss function but the assumptions that we impose on the objective function here are less restrictive.  
In Section~\ref{cbo_sec_mf}, we prove convergence of the interacting particle system (\ref{cboeq1.8}) towards the mean-field limit (\ref{cbomfsdep}) as $N\rightarrow \infty$. 
In Section~\ref{cbo_conv_ns}, we  prove uniform in $N$ convergence of the Euler scheme (\ref{cbo_dis_ns}) to (\ref{cboeq1.8}) as $h \rightarrow 0$, where $h$ is the discretization step.

\subsection{Convergence towards the global minimum}\label{cbo_sec_gl_min}
The aim of this section is to show that the non-linear process $X(t)$ driven by the distribution dependent SDEs (\ref{cbomfsde}) converges to a point $x^{*} $ which lies in a close vicinity of the global minimum which we denote as $x_{\min}$.  To this end, we will first prove that  $\var(t) := \mathbb{E}|X(t) - \mathbb{E}(X(t))|^{2}  $  satisfies a differential inequality which, with particular choice of parameters,  implies exponential decay of $\var(t)$ as $t \rightarrow \infty$. We also obtain a differential inequality for $M(t) := \mathbb{E}\big(e^{-\alpha f(X(t))}\big )$.

The approach that we follow in this section is along the lines of \cite{cbo2,cbo3} but with necessary adjustments for the jump term in (\ref{cbomfsde}).

\begin{lemma} \label{cbo_prop_4.1}
Under Assumptions~\ref{cboh3.1}, \ref{cboh3.2}-\ref{cboasm1.4}, the following inequality is satisfied for $\var(t)$:
\begin{align}
\frac{d}{dt}\var(t) &\leq - \bigg(2\beta(t) - \big(2\sigma^{2}(t) +\lambda\gamma^{2}(t)\mathbb{E}|\Zstroke|^{2}\big)\Big( 1+ \frac{e^{-\alpha f_{m}}}{M(t^{})} \Big) \bigg)\var(t^{}). \label{h4.1} 
\end{align}
\end{lemma}
\begin{proof}
Using Ito's formula, we have
\begin{align}
    |X(t) &- \mathbb{E}X(t)|^{2}  = |X(0) - \mathbb{E}X(0)|^{2}-2\int_{0}^{t}\beta(s)(X(s) - \mathbb{E}X(s))\cdot(X(s) - \bar{X}(s))ds  \nonumber \\ &- 2\int_{0}^{t}(X(s) - \mathbb{E}X(s))\cdot d\mathbb{E}X(s) + 2\int_{0}^{t}\sigma^{2}(s)|X(s) - \bar{X}(s)|^{2}ds  \nonumber \\ & + 2\sqrt{2}\int_{0}^{t}\sigma(s)(X(s) - \mathbb{E}X(s))\cdot \big(\diag(X(s) - \bar{X}(s)) dW(s)\big)\nonumber  \\ &  + \int_{0}^{t}\int_{\mathbb{R}^{d}}\big\{|X(s^{-}) - \mathbb{E}X(s^{-}) + \gamma(s)\diag(X(s^{-})  - \bar{X}(s^{-}))z|^{2} - |X(s^{-})- \mathbb{E}(X(s^{-}))|^{2}\big\} \mathcal{N}(ds,dz). \nonumber
\end{align}
Taking expectation on both sides, we get
\begin{align}
    &\var(t)  = \var(0) -2\mathbb{E}\int_{0}^{t}\beta(s)\mathbb{E}\big((X(s^{}) - \mathbb{E}X(s^{}))\cdot(X(s^{}) - \bar{X}(s^{}))\big)dt + 2\int_{0}^{t}\sigma^{2}(s)\mathbb{E}|X(s^{}) - \bar{X}(s^{})|^{2}ds \nonumber \\ 
    &  \;\;\;\; +  \lambda \gamma^{2}(s)\int_{0}^{t}\int_{\mathbb{R}^{d}}\mathbb{E}|\diag(X(s^{}) - \bar{X}(s^{}))z|^{2}\rho_{z}(z)dzds \nonumber \\ 
   & = \var(0) -2\int_{0}^{t} \big(\beta(s)\var(s^{}) + 2\sigma^{2}(s)\mathbb{E}|X(s^{}) - \bar{X}(s^{})|^{2}  +  \lambda \gamma^{2}(s)\mathbb{E}|\Zstroke|^{2}\mathbb{E}|X(s^{}) - \bar{X}(s^{})|^{2} \big) ds, \label{cbo_neweq_4.2}
\end{align}
since 
\begin{align*}
&\mathbb{E}\big((X(t^{}) - \mathbb{E}X(t^{}))\cdot(\mathbb{E}X(t^{}) - \bar{X}(t^{}))\big) = 0, \\ & |X(t^{}) - \mathbb{E}X(t^{}) + \diag(X(t^{}) -\bar{X}(t^{}))z|^{2} = |X(t^{}) - \mathbb{E}X(t^{})|^{2} + |\diag(X(t^{}) - \bar{X}(t^{}))z|^{2} \\ & \;\;\;\;\;\;\;\;\;\;\;\;\;\;\;\;\; + 2\big((X(t^{}) - \mathbb{E}X(t^{}))\cdot\diag(X(t^{}) - \bar{X}(t^{}))z\big),  \\ & \int_{\mathbb{R}^{d}}\big((X(t^{}) - \mathbb{E}X(t^{}))\cdot \diag(X(t^{})-\bar{X}(t^{}))z\big)\rho_{z}(z)dz = 0.
\end{align*}
Moreover, $ \int_{\mathbb{R}^{d}} \sum_{l=1}^{d} 
(X(t^{}) - \bar{X}(t^{}))_{l}^{2}z_{l}^{2} \rho_{z}(z)dz = \sum_{l=1}^{d} (X(t^{}) - \bar{X}(t^{}))_{l}^{2} \int_{\mathbb{R}{^d}}z_{l}^{2}\prod_{i=1}^{d}\rho_{\zstroke}(z_{i})dz = |X(t^{}) - \bar{X}(t^{})|^{2}\mathbb{E}|\Zstroke|^{2}   $, since each component $Z_{l}$ of $Z$ is distributed as $\Zstroke$. 

We also have 
\begin{align}\mathbb{E}|X(t^{})  - \bar{X}(t^{})|^{2} = \var(t)  + |\mathbb{E}X(t^{}) - \bar{X}(t^{})|^{2}. \label{cbo_eq_4.2}
\end{align} 
We estimate  the term $|\mathbb{E}(X(t^{})) - \bar{X}(t^{})|^{2}$ using Jensen's inequality as
\begin{align}\label{cboeq4.2}
    |\mathbb{E}X(t^{}) - \bar{X}(t^{})|^{2} & = \bigg| \mathbb{E}X(t^{}) - \frac{\mathbb{E}X(t^{})e^{-\alpha f(X(t^{}))}}{\mathbb{E}e^{-\alpha f(X(t^{}))}}\bigg|^{2} \nonumber  =
    \bigg|\mathbb{E} \bigg( \Big(\mathbb{E}X(t^{}) - X(t^{})\Big)\frac{e^{-\alpha f(X(t^{}))}}{\mathbb{E}e^{-\alpha f(X(t^{}))} }\bigg)\bigg|^{2} \nonumber\\  & = \bigg|\int_{\mathbb{R}^{d}}\big(\mathbb{E}X(t) - x\big) \vartheta_{X(t)}(dx)\bigg|^{2}  \leq \int_{\mathbb{R}^{d}}\big|\mathbb{E}X(t) - x\big|^{2} \vartheta_{X(t)}(dx)\nonumber \\ & = \mathbb{E}\bigg(|X(t^{}) - \mathbb{E}(X(t^{}))|^{2} \frac{e^{-\alpha f(X(t^{}))}}{\mathbb{E}e^{-\alpha f(X(t^{}))}}\bigg)\leq \frac{e^{-\alpha f_{m}}}{\M(t^{})}\var(t^{}),
\end{align}
where $\vartheta_{X(t)}(dx) = e^{-\alpha f(x)}/\mathbb{E}(e^{-\alpha f(X(t))}) \mathcal{L}_{X(t)}(dx) $ which implies $\int_{\mathbb{R}^{d}}\vartheta_{X(t)}(dx) = 1$.
Using (\ref{cbo_eq_4.2}) and (\ref{cboeq4.2}) in (\ref{cbo_neweq_4.2}) gives the targeted result.
\end{proof}

To prove the main result of this section, we need an additional  inequality, which is proved under the following assumption. 

\begin{assumption}\label{cbohas4.1}
$f \in C^{2}(\mathbb{R}^{d})$ and there exist three constants $K_{1}$,$K_{2}, K_{3} > 0$ such that the following inequalities are satisfied for sufficiently large $\alpha$:
\begin{itemize}
        \item[(i)] $(\nabla f(x) -\nabla f(y))\cdot (x-y) \geq -K_{1}|x-y|^{2}$ for all $x$, $ y \in \mathbb{R}^{d}$. 
    
    \item[(ii)] $ \alpha\Big(\frac{\partial f}{\partial x_{i}}\Big)^{2} -\frac{\partial^{2} f}{\partial x_{i}^{2}} \geq  -K_{2}$ for all $i = 1,\dots,d$ and $x \in \mathbb{R}^{d}$.
    
    \item[(iii)] $\mathbb{E}f(x+ \diag(x)Z) - f(x) \leq  K_{3} |x|^{2}\mathbb{E}|\Zstroke|^{2} $, \\ where $Z$ is a d-dimensional random vector  and $\Zstroke$ is real valued random variable introduced  in Section~\ref{sec_our_mod}.
\end{itemize}
\end{assumption}

We note that for $f(x) = 1+ |x|^{2}$, $x \in \mathbb{R}^{d}$, we have $\mathbb{E}|x+ \diag(x)Z|^{2} - |x|^{2} = \mathbb{E}|\diag(x)Z|^{2} = \sum_{l=1}^{d}\mathbb{E}(x_{l}Z_{l})^{2}$. However, each $Z_{l}$ is distributed as $\Zstroke$. Hence, $\mathbb{E}|x+ \diag(x)Z|^{2} - |x|^{2} = |x|^{2}\mathbb{E}|\Zstroke|^{2}$. The conditions $(i)$ and $(ii)$ are straightforward to verify for $1+|x|^{2}$. This implies the existence of a function satisfying the above assumption. This ensures that the class of functions satisfying the above assumption is not empty and is consistent with Assumptions~\ref{cboh3.1}, \ref{cboh3.2}-\ref{cboasm1.4}.  The most important implication is that the above assumption allows $f$ to have quadratic growth which is important for several loss functions in machine learning problems.

In \cite{cbo2}, the authors  assumed $f \in C^{2}(\mathbb{R}^{d})$, the norm of Hessian of $f$  being bounded by a constant, and the norm of gradient and Laplacian of $f$ satisfying the inequality, $\Delta f \leq c_{0} + c_{1}|\nabla f|^{2}$, where $c_{0}$ and $ c_{1}$ are positive constants. Therefore, in Assumption~\ref{cbohas4.1}, we have imposed restrictions on $f$ similar to \cite{cbo2} in the essence of regularity but adapted to our jump-diffusion case with component-wise Wiener noise.

\begin{lemma}\label{cbo_lem_4.2}
The following inequality holds under Assumptions~\ref{cboh3.1}, \ref{cboh3.2}-\ref{cboasm1.4} and \ref{cbohas4.1}:  
\begin{align}
    \frac{d}{dt}\M^{2}(t) &\geq  - 4\alpha e^{-\alpha f_{m}}\Big(\beta(t)K_{1} + \sigma^{2}(t)K_{2} + \lambda \gamma^{2}(t)K_{3}\mathbb{E}|\Zstroke|^{2}\Big)\var(t^{}).\label{h4.2}
\end{align}
\end{lemma}
\begin{proof}
Using Ito's formula, we get 
\begin{align*}
    e^{-\alpha f(X(t))}& =  \int_{0}^{t}\alpha \beta(s) e^{-\alpha f(X(s))}\nabla f(X(s))\cdot (X(s) -\bar{X}(s)) ds \\ & \;\;\;\;  - \sqrt{2} \int_{0}^{t}\alpha \sigma(s) e^{-\alpha f(X(s))}\nabla f(X(s))\cdot \big(\diag(X(s) -\bar{X}(s)) dW(s)\big)  
    \\ & \;\;\;\;+ \int_{0}^{t}\sigma^{2}(s)e^{-\alpha f(X(s))}\sum\limits_{j = 1}^{d}\bigg( \big(X(s) - \bar{X}(s)\big)^{2}_{j}  \Big(\alpha^{2} \Big(\frac{\partial f(X(s))}{\partial x_{j}}\Big)^{2} - \alpha\frac{\partial^{2}f(X(s))}{\partial x_{j}^{2}}\Big)\bigg)ds \\ & \;\;\;\;+ \int_{0}^{t}\int_{\mathbb{R}^{d}}\Big(e^{-\alpha f(X(s^{-}) + \gamma(s)\diag(X(s^{-}) - \bar{X}(s^{-}))z)} - e^{-\alpha f(X(s^{-}))}\Big) \mathcal{N}(ds,dz).
\end{align*}
Taking expectation on both sides and writing in the differential form  yield
\begin{align*}
    d\mathbb{E}e^{-\alpha f(X(t))} & = \alpha \beta(t)\mathbb{E}\big(e^{-\alpha f(X(t^{}))}(\nabla f(X(t^{})) -\nabla f(\bar{X}(t^{})))\cdot (X(t^{}) - \bar{X}(t^{}))\big) dt \\ 
    & +\sigma^{2}(t)\mathbb{E}\Bigg(e^{-\alpha f(X(t^{}))}\sum\limits_{j = 1}^{d}\bigg( \big(X(t^{}) - \bar{X}(t^{})\big)^{2}_{j}  \Big(\alpha^{2} \Big(\frac{\partial f(X(t^{}))}{\partial x_{j}}\Big)^{2} - \alpha\frac{\partial^{2}f(X(t^{}))}{\partial x_{j}^{2}}\Big)\bigg)\Bigg)dt \\ 
    & +\lambda \int_{\mathbb{R}^{d}}\mathbb{E}\Big(e^{-\alpha f(X(t^{}) + \gamma(t)\diag(X(t^{}) - \bar{X}(t^{}))z)} - e^{-\alpha f(X(t^{}))}\Big) \rho_{z}(z)dz dt,
\end{align*}
where we have used the fact $ \mathbb{E}\big[e^{-\alpha f(X(t))}(\nabla f(\bar{X}(t))\cdot (X(t) -\bar{X}(t)))\big] = 0$. 

Note that $|e^{-\alpha f(x)} - e^{-\alpha f(y)}| \leq \alpha e^{-\alpha f_{m}}|f(x) - f(y)|  $ which means $e^{-\alpha f(x)} - e^{-\alpha f(y)} \geq -\alpha e^{-\alpha f_{m}} |f(x) -f(y)| $. Using Assumption~\ref{cbohas4.1}, we get
\begin{align*}
    d\mathbb{E}e^{-\alpha f(X(t))} \geq - \alpha e^{-\alpha f_{m}}\big(\beta(t)K_{1} + \sigma^{2}(t)K_{2} + \lambda \gamma^{2}(t)K_{3}\mathbb{E}|\Zstroke|^{2}\big)\mathbb{E}|X(t^{}) - \bar{X}(t^{})|^{2}.
\end{align*}
From (\ref{cbo_eq_4.2}) and (\ref{cboeq4.2}), we have
\begin{align*}
    \mathbb{E}|X(t) - \bar{X}(t)|^{2} \leq \var(t^{}) + \frac{e^{-\alpha f_{m}}}{\M(t^{})}\var(t^{}) \leq 2 \frac{e^{-\alpha f_{m}}}{\M(t^{})}\var(t^{}).
\end{align*}
This implies 
\begin{align*}
    d \M(t) \geq -2 \alpha e^{-\alpha f_{m}}\big(\beta(t)K_{1} + \sigma^{2}(t)K_{2} + \lambda \gamma^{2}(t) K_{3}\mathbb{E}|\Zstroke|^{2}\big) \frac{e^{-\alpha f_{m}}}{\M(t^{})}\var(t^{}) dt,
\end{align*}
which is  what we aimed to prove in this lemma.
\end{proof}
Our next objective is to show that $\mathbb{E}(X(t))$ converges to $x^{*}$ as $t \rightarrow \infty$, where $x^{*}$ is close to $x_{\min}$, i.e. the point at which  $f(x)$ attains its minimum value, $f_{m}$. Applying Laplace's method (see e.g. \cite[Chap. 3]{cbo38} and also \cite{cbo1,cbo2}), we can calculate the following asymptotics: for any compactly supported probability measure $\rho \in \mathcal{P}(\mathbb{R}^{d})$ with $x_{\min} \in \text{supp}(\rho)$, we have
\begin{align}
    \lim\limits_{\alpha \rightarrow \infty}\Bigg(-\frac{1}{\alpha}\log\bigg(\int_{\mathbb{R}^{d}}e^{-\alpha f(x)}d\rho(x)\bigg)\Bigg) = f_{m} > 0. \label{cbo_neweq_4.6}
\end{align}
Based on the above asymptotics, we aim to prove that
\begin{align*}
    f(x^{*}) \leq f_{m} + \Gamma(\alpha) +  \mathcal{O}\bigg(\frac{1}{\alpha}\bigg),
\end{align*}
where a function $\Gamma(\alpha) \rightarrow 0 $ as $ \alpha \rightarrow \infty$.


 We introduce the following function: 
\begin{align*}
    \chi(t) = 2\beta(t) - \big(2\sigma^{2}(t) +\lambda\gamma^{2}(t)\mathbb{E}|\Zstroke|^{2}\big)\Big( 1+ \frac{2e^{-\alpha f_{m}}}{M(0)} \Big).
\end{align*}
We choose $\alpha$, $\beta(t)$, $\sigma(t)$, $\gamma(t)$, $\lambda$, distribution of $\Zstroke$ 
such that
\begin{itemize}
\item[(i)] $\chi(t)$ is a continuous function of time $t$, 
\item[(ii)] $\chi(t) > 0$ for all $t \geq 0$, and
\item[(iii)] $ \chi(t) $ attains its minimum which we denote as $\chi_{\min}$.
\end{itemize}

We also introduce
\begin{align*}
    \eta &:= 4\alpha e^{-\alpha f_{m}}\var(0)\frac{K_{1} \beta + K_{2}\sigma^{2}(0)  + K_{3}\lambda  \gamma^{2}(0)\mathbb{E}|\Zstroke|^{2}}{ \M^{2}(0)\chi_{\min}},
\end{align*}
where $\beta $ is introduced in Section~\ref{sec_our_mod}, and $K_{1}$, $K_{2}$ and $K_{3}$ are from Assumption~\ref{cbohas4.1}.
 
The next theorem is the main result of this section.  We will be assuming that $\eta \leq 3/4$ which can always be achieved by choosing sufficiently small $\var(0)$.

\begin{theorem}\label{cbo_thrm_4.3}
Let Assumptions~\ref{cboh3.1}, \ref{cboh3.2}-\ref{cboasm1.4} and \ref{cbohas4.1} hold. Let us also assume that $\mathcal{L}_{X(0)}$ is compactly supported and $x_{\min} \in \text{supp}(\mathcal{L}_{X(0)})$. If $\eta \leq 3/4$, then  $\var(t)$ exponentially decays to zero as $t \rightarrow \infty$. Further, there exists an $x^{*} \in \mathbb{R}^{d}$ such that $X(t) \rightarrow x^{*}$ a.s., $\mathbb{E}(X(t)) \rightarrow x^{*}$, $\bar{X}(t) \rightarrow x^{*}$  as $ t \rightarrow \infty$ and the following inequality holds:
\begin{align*}
    f(x^{*}) \leq f_{m}  + \Gamma(\alpha) + \frac{\log{2}}{\alpha}, 
\end{align*}
where  function $\Gamma(\alpha) \rightarrow 0 $ as $ \alpha \rightarrow \infty$.
\end{theorem}

\begin{proof}
Let
$
    T^{*} = \sup\big\{ t \;;\; \M(s) > \frac{\M(0)}{2}, \text{for all}\; s \in [0,t]\big\}.
$
Observe that $T^{*} > 0$ by definition.

Let us assume that $T^{*} < \infty$.  We can deduce that the following holds by definition of $T^{*}$ for all $t\in [0,T^{*}]$:
\begin{align*}
   2\beta(t) - \big(2\sigma^{2}(t) +\lambda\gamma^{2}(t)\mathbb{E}|\Zstroke|^{2}\big)\Big( 1+ \frac{e^{-\alpha f_{m}}}{M(t^{})} \Big) \geq   2\beta(t) - \big(2\sigma^{2}(t) +\lambda\gamma^{2}(t)\mathbb{E}|\Zstroke|^{2}\big)\Big( 1+ \frac{2e^{-\alpha f_{m}}}{M(0)} \Big) = \chi(t),
\end{align*}
where the left hand side of the above inequality is from (\ref{h4.1}). Using Lemma~\ref{cbo_prop_4.1}, the fact that $\chi(t)$ is continuous and $\chi(t) > 0 $ for all $t \geq 0$, we get for all $t \in [0, T^{*}]$:
\begin{align*}
    \var(t) \leq \var(0)e^{-\chi(t)t} \leq \var(0)e^{-\chi_{\min}t}. 
\end{align*}
We have from Lemma~\ref{cbo_lem_4.2} for all $t \in (0,T^{*}]$:
\begin{align*}
\M^{2}(t) &\geq  \M^{2}(0) - 4\alpha e^{-\alpha f_{m}}\int_{0}^{t} \big(K_{1} \beta(s) + K_{2}\sigma^{2}(s) + K_{3}\lambda \gamma^{2}(s) \mathbb{E}|\Zstroke|^{2}\big)\var(s)ds \\ 
& \geq \M^{2}(0) - 4\alpha e^{-\alpha f_{m}}\big(K_{1} \beta + K_{2}\sigma^{2}(0) + K_{3}\lambda\gamma^{2}(0)\mathbb{E}|\Zstroke|^{2}\big)\var(0)  \int_{0}^{t}e^{-\chi_{\min}s}ds 
\\ 
& =  \M^{2}(0) - 4\alpha e^{-\alpha f_{m}}\big(K_{1} \beta + K_{2}\sigma^{2}(0) + K_{3}\lambda\gamma^{2}(0)\mathbb{E}|\Zstroke|^{2}\big)\frac{\var(0)}{\chi_{\min}}\big(1 - e^{-\chi_{\min}t}\big)\\ 
& > \M^{2}(0) - 4\alpha e^{-\alpha f_{m}}\big(K_{1} \beta + K_{2}\sigma^{2}(0) + K_{3}\lambda\gamma^{2}(0)\mathbb{E}|\Zstroke|^{2}\big)\frac{\var(0)}{\chi_{\min}} \geq  \frac{\M^{2}(0)}{4},  
\end{align*}
where in the last step we have used the fact that $\eta \leq 3/4$.
This shows $\M(t) > \M(0)/2$ which implies $\M(t) -  \M(0)/2 > 0$ on the set $(0,T^{*}]$. Also, note that $M(t)$ is continuous in $t$, therefore there exists an $\epsilon > 0$ such that $\M(t) > \M(0)/2$ for all $t \in [T^{*},T^{*}+\epsilon)$. This creates a contradiction which implies $T^{*} = \infty$.  Hence, 
\begin{equation}
    \var(t) \leq \var(0) e^{- \chi_{\min}t}\;\; \text{and}\;\;    \M(t) > \M(0)/2 \; \text{ for all}\; t > 0. \label{cbo_neweq_4.7}
\end{equation}
This implies $\var(t)$ exponentially decays to zero as $t \rightarrow \infty$.
 From (\ref{cboeq4.2}) and (\ref{cbo_neweq_4.7}), we get
\begin{align} \label{cbo_eq_4.7}
    |\mathbb{E}X(t) - \bar{X}(t)|^{2} \leq e^{-\alpha f_{m}} \frac{\var(t)}{\M(t)} \leq Ce^{-\chi_{\min} t},\;\;\;\; t > 0,
\end{align}
where $C$ is a positive constant independent of $t$.

Taking expectation on both sides of (\ref{cbomfsdep}) (recall that $\mathbb{E}\Zstroke = 0$), applying Holder's inequality and using (\ref{cbo_eq_4.2}) gives
\begin{align}
    \bigg| \frac{d}{dt}\mathbb{E}X(t)\bigg| &\leq \beta \mathbb{E}|X(t^{})- \bar{X}(t^{})| \leq \beta (\mathbb{E}|X(t^{}) - \bar{X}(t)|^{2})^{1/2 } \leq  \beta \big(\var(t) + |\mathbb{E}X(t^{}) - \bar{X}(t^{})|^{2}\big)^{1/2} \nonumber \\  &  \leq Ce^{-\chi_{\min}t/2},\;\;\;\; t > 0, \label{cbo_eq_4.8}
\end{align}
where $C$ is a positive constant independent of $t$.

 It is clear from (\ref{cbo_eq_4.8}) that there exists an $x^{*} \in \mathbb{R}^{d}$ such that $ \mathbb{E}(X(t)) \rightarrow x^{*}$ as $t \rightarrow \infty$. Further,  $\bar{X}(t) \rightarrow x^{*}$  as $ t \rightarrow \infty$ due to (\ref{cbo_eq_4.7}).

Let $\ell > 0$. Using Chebyshev's inequality, we have
\begin{align*}
    \mathbb{P}(|X(t) - \mathbb{E}X(t)| \geq e^{-\ell t}) \leq \frac{\var{(t)}}{e^{-2\ell t}} \leq Ce^{-(\chi_{\min} - 2\ell )t},
\end{align*}
where $C>0$ is independent of $t$. 
 If we choose $\ell < \chi_{\min}/2$, then we can say $|X(t) - \mathbb{E}X(t)| \rightarrow 0$ as $t \rightarrow 0$ a.s. due to the Borel-Cantelli lemma. This  implies $X(t) \rightarrow x^{*}$ a.s. Application of the bounded convergence theorem gives the convergence result: $\mathbb{E}e^{-\alpha f(X(t))} \rightarrow e^{-\alpha f(x^{*})} $ as $t \rightarrow \infty$.  
Then, due to (\ref{cbo_neweq_4.7}), we obtain 
\begin{align*}
    e^{-2\alpha f(x^{*})} \geq  M^{2}(0)/4
\end{align*}
and hence
\begin{align*}
    f(x^{*}) \leq  - \frac{1}{\alpha}\log(\M(0)) + \frac{1}{\alpha}\log{2}.
\end{align*}
Then, using the asymptotics (\ref{cbo_neweq_4.6}), we get
\begin{align} \label{cbo_eqn_4.9}
    f(x^{*}) \leq f_{m} + \Gamma(\alpha) + \frac{1}{\alpha}\log{2}, 
\end{align}
where the function $\Gamma(\alpha) \rightarrow 0 $ as $ \alpha \rightarrow \infty$.
\end{proof}
\subsection{Convergence to the mean-field SDEs}\label{cbo_sec_mf}
In the previous section, we showed convergence of the non-linear process $X(t)$ from (\ref{cbomfsdep}) towards the global minimizer. However, the CBO method is based on the system (\ref{cbos1.6}) of finite particles. This means there is a missing link in the theoretical analysis which we fill in this section by showing convergence of the particle system (\ref{cbos1.6}) to the mean-field limit in mean-square sense (\ref{cbomfsdep}) as the number of particles tends to infinity. The proof of this result has some ingredients inspired from \cite{cbo36} (see also \cite{cbo37}), precisely where we partition the sample space (cf. Theorem~\ref{cbo_thrm4.5}). Further, it is clear from the proof that we need stronger moment bound result like in Lemmas~\ref{cbolemma3.3} and \ref{cbolem3.6}, as compared to \cite[Lemma 3.4]{cbo2}.

We first discuss some concepts necessary for later use in this section. We introduce the following notation for the empirical measure of i.i.d. particles driven by the McKean-Vlasov SDEs (\ref{cbomfsdep}):
\begin{align}
    \mathcal{E}_{t} : = \frac{1}{N}\sum\limits_{i=1}^{N}\delta_{X^{i}(t)},
\end{align}
where $\delta_{x}$ is the Dirac measure at $x \in \mathbb{R}^{d}$.
We will also need the following notation:
\begin{align}\label{cboeq5.2}
    \bar{X}^{\mathcal{E}_{t}}(t) = \frac{\int_{\mathbb{R}^{d}}x e^{-\alpha f(x)} \mathcal{E}_{t}(dx)}{\int_{\mathbb{R}^{d}} e^{-\alpha f(x)} \mathcal{E}_{t}(dx)} = \frac{\sum_{i=1}^{N}X^{i}(t)e^{-\alpha f(X^{i}(t))}}{\sum_{i=1}^{N}e^{-\alpha f(X^{i}(t))}}.
\end{align}
 
Using discrete Jensen's inequality, we have
\begin{align*}
    \exp{\bigg(-\alpha\frac{1}{N}\sum\limits_{i=1}^{N}f(X^{i}(t)) \bigg)} &\leq  \frac{1}{N}\sum\limits_{i=1}^{N}\exp{\Big(-\alpha f(X^{i}(t))\Big)},
\end{align*}
which, on rearrangement and multiplying both sides by $e^{-\alpha f_{m}}$, gives
\begin{align}\label{y4.5}
    \frac{e^{-\alpha f_{m}}}{\frac{1}{N}\sum_{i=1}^{N}e^{-\alpha f(X^{i}(t))}} &\leq \exp{\bigg(\alpha\Big(\frac{1}{N}\sum\limits_{i=1}^{N}f(X^{i}(t)) - f_{m}\Big)\bigg)}
    \leq e^{\alpha K_{u}} \exp{\Big(\frac{\alpha K_{u}}{N}\sum\limits_{i=1}^{N} |X^{i}(t)|^{2}\Big)},
\end{align}
where we have used Assumption~\ref{cboassu3.4} for the second inequality. 

We recall that a  random variable $ \zeta(\omega)$ is a.s. finite if there is an increasing sequence $\{e_{k}\}_{k\in \mathbb{N}}$ with $e_{k}\rightarrow \infty$ as $k \rightarrow \infty$ such that
\begin{align*}
    \mathbb{P}\big(\cup_{k=1}^{\infty}\{\omega \; : \; |\zeta(\omega)|  < e_{k} \}\big) = 1, 
\end{align*}
which means
\begin{align*}
    \mathbb{P}\big(\cap_{k=1}^{\infty}\{\omega \; : \; |\zeta(\omega)|  \geq  e_{k} \}\big) = 0, \;\;\;\;\text{i.e.}\;\;\;\;\;\; \mathbb{P}\big(\lim_{k\rightarrow \infty}\{ \omega\; :    \; |\zeta(\omega)| \geq e_{k}\} \big) =0.    
\end{align*}
Let $g(x)$ be an increasing continuous function of $x \in \mathbb{R}$ then $g(\zeta(\omega))$ is a.s. finite random variable as well. Also, if $\zeta_{1}(\omega)$ and $\zeta_{2}(\omega)$ are a.s. finite random variables then $\zeta_{1}(\omega) \vee \zeta_{2}(\omega)$ is also an a.s. finite random variable.
If $\zeta(\omega)$ is a.s. finite then by continuity of probability we have \cite{cbo35}:
\begin{align}
    \lim_{k\rightarrow \infty}\mathbb{P}(\{ \omega \; : \; |\zeta(\omega)| \geq e_{k}\}) = 0.
\end{align}

We know that $X^{i}(t)$, governed by the McKean-Vlasov SDEs (\ref{cbomfsdep}), are i.i.d. random variables for every $t\geq 0$, therefore using Chebyshev's inequality, we get
\begin{align*}
    &\mathbb{P}\Big(\frac{1}{N}\sum\limits_{i=1}^{N}|X^{i}(t)|^{2} - \mathbb{E}|X(t)|^{2} \geq N^{(\epsilon-1)/4}\Big)  \leq \frac{\mathbb{E}\Big|\frac{1}{N}\sum_{i=1}^{N}|X^{i}(t)|^{2} - \mathbb{E}|X(t)|^{2}\Big|^{4}}{N^{(\epsilon-1)}} \\ & = 
    \frac{\mathbb{E}\Big|\sum_{i=1}^{N}\big(|X^{i}(t)|^{2} - \mathbb{E}|X(t)|^{2}\big)\Big|^{4}}{N^{3+\epsilon}} = \frac{\sum_{i =1}^{N}\mathbb{E}U_{i}^{4}}{N^{3+\epsilon}} + \frac{\sum_{i=1}^{N}\mathbb{E}U_{i}^{2}\sum_{j=1}^{N}\mathbb{E}U_{j}^{2}}{N^{3+\epsilon}}
        \\ & \leq \frac{C}{N^{1+\epsilon}},
\end{align*}
where we have used Lemma~\ref{cbolem3.6}, $U_{i} = |X^{i}(t)|^{2} - \mathbb{E}|X(t)|^{2}$ and $C$ is independent of $N$.
We take $\epsilon \in (0,1)$ and define $E_{N} = \left\{ \frac{1}{N}\sum\limits_{i =1}^{N}|X^{i}(t)|^{2} - \mathbb{E}|X(t)|^{2} > \frac{1}{N^{(1-\epsilon)/4}}\right\}$  then
\begin{align*}
\sum\limits_{N =1}^{\infty}\mathbb{P}(E_{N}) < \infty.    
\end{align*}
The Borel-Cantelli lemma implies that the random variable 
\begin{align*}
    \zeta_{1}(t) := \sup_{N\in \mathbb{N}}N^{(1-\epsilon)/4}\Big(\frac{1}{N}\sum_{i=1}^{N}|X^{i}(t)|^{2} - \mathbb{E}|X(t)|^{2}\Big)
\end{align*}
is a.s. finite. Therefore,

\begin{align} \label{y4.6}
    \frac{1}{N}\sum\limits_{i=1}^{N}|X^{i}(t)|^{2} \leq \mathbb{E}|X(t)|^{2} + \zeta_{1}(t,\omega)N^{(-1 +\epsilon)/4},\;\;\;\; a.s.,
\end{align}
 for all $t \in [0,T]$. Using (\ref{y4.6}) in (\ref{y4.5}) and Lemma~\ref{cbolem3.6}, we get
\begin{align}\label{cboeq5.6}
    \frac{e^{-\alpha f_{m}}}{\frac{1}{N}\sum\limits_{i=1}^{N}e^{-\alpha f(X^{i}(t))}} \leq e^{\alpha K_{u}(1+ K_{p}+\zeta_{1}(t,\omega)N^{(-1+\epsilon)/4})  },\;\;\;\; a.s.
\end{align}
This show that
\begin{align}
    \lim\limits_{N \rightarrow \infty}\frac{e^{-\alpha f_{m}}}{\frac{1}{N}\sum\limits_{i=1}^{N}e^{-\alpha f(X^{i}(t))}} \leq e^{\alpha K_{u}(1+ K_{p})},\;\;\;\; a.s.
\end{align}
\begin{lemma}\label{cbolem5.1}
Let Assumptions~\ref{cboh3.1}, \ref{cboh3.2}-\ref{cboasm1.4} be  satisfied. Let $\mathbb{E}|X(0)|^{4} < \infty$ and $\mathbb{E}|Z|^{4} < \infty$. Then, the following bound holds for all $t\in [0,T]$ and sufficiently large $N$:
\begin{align} \label{cboeq5.7}
    |\bar{X}^{\mathcal{E}_{t}}(t) - \bar{X}(t)| \leq \frac{\zeta(t,\omega)}{N^{(1-\epsilon)/4}}, \;\;\;\; a.s.,
\end{align}
where $\bar{X}^{\mathcal{E}_{t}}(t)$ is from (\ref{cboeq5.2}), $\bar{X}(t)$ is from (\ref{eqcbo2.12}), $\zeta(t,\omega) $ is an $a.s.$ finite $\mathscr{F}_{t}-$ measurable random variable and $ \epsilon \in( 0,1)$.
\end{lemma}
\begin{proof}
We have
\begin{align}
    |\bar{X}^{\mathcal{E}_{t}}(t) &- \bar{X}(t)| = \bigg| \sum_{i=1}^{N}X^{i}(t)\frac{e^{-\alpha f(X^{i}(t))}}{\sum_{j=1}^{N}e^{-\alpha f(X^{j}(t))} } - \int_{\mathbb{R}^{d}}x\frac{e^{-\alpha f(x)}}{\int_{\mathbb{R}^{d}}e^{-\alpha f(x)} \mathcal{L}_{X(t)}(dx)}\mathcal{L}_{X(t)}(dx)\bigg| \nonumber\\ 
    & \leq \bigg| \frac{1}{\sum_{j=1}^{N}e^{-\alpha f(X^{j}(t))}}\bigg( \sum_{i=1}^{N}X^{i}(t) e^{-\alpha f(X^{i}(t))} - \int_{\mathbb{R}^{d}}xe^{-\alpha f(x)}\mathcal{L}_{X(t)}(dx)\bigg)\bigg| \nonumber \\ & \;\;\;\; + 
    \bigg|\int_{\mathbb{R}^{d}}xe^{-\alpha f(x)}\mathcal{L}_{X(t)}(dx)\bigg(\frac{1}{\sum_{j=1}^{N}e^{-\alpha f(X^{j}(t))}} - \frac{1}{\int_{\mathbb{R}^{d}}e^{-\alpha f(x)}\mathcal{L}_{X(t)}(dx)}\bigg)\bigg|. \label{cbo_eq_4.13}
\end{align}
Let $ Y^{i}(t) = X^{i}(t) e^{-\alpha f(X^{i}(t))} - \int_{\mathbb{R}^{d}}xe^{-\alpha f(x)}\mathcal{L}_{X(t)}(dx)$. Note that $\mathbb{E}Y^{i}(t)$ is a $d-$dimensional zero vector and $\mathbb{E}(Y^{i}(t)\cdot Y^{j}(t)) = 0$, $i\neq j$. Then, using Theorem~\ref{cbolem3.6}, we obtain
\begin{align}\label{cboeq4.14}
    \mathbb{E}\Big|\sum_{i=1}^{N}X^{i}(t) e^{-\alpha f(X^{i}(t))} &- \int_{\mathbb{R}^{d}}xe^{-\alpha f(x)}\mathcal{L}_{X(t)}(dx)\Big|^{4} = \frac{1}{N^{4}}\mathbb{E}\Big|\sum\limits_{i=1}^{N}Y^{i}(t)\Big|^{4}\nonumber\\
    & = \frac{1}{N^{4}}\mathbb{E}\bigg(\sum\limits_{i=1}^{N}|Y^{i}(t)|^{4}+ \sum_{i=1}^{N}|Y^{i}(t)|^{2}\sum_{j=1}^{N}|Y^{j}(t)|^{2}\bigg) \leq \frac{C}{N^{2}},
\end{align}
where $C$ is a positive constant independent of $N$. 
As a consequence of above estimate and using Chebyshev's inequality, we get
\begin{align*}
    \mathbb{P}\bigg(\Big|\sum_{i=1}^{N}X^{i}(t) e^{-\alpha f(X^{i}(t))} &- \int_{\mathbb{R}^{d}}X(t)e^{-\alpha f(X(t))}\mathcal{L}_{X(t)}(dx)\Big| \geq N^{(\epsilon-1)/4}\bigg) \leq \frac{C}{N^{1+\epsilon}}.
\end{align*}
Therefore, by the Borel-Cantelli lemma there exists an  a.s. finite $\mathcal{F}_{t}$-measurable random variable $\zeta_{2}(t,\omega)$ such that the following bound holds: 
\begin{align}\label{cboeq5.8}
    \Big|\sum_{i=1}^{N}X^{i}(t) e^{-\alpha f(X^{i}(t))} &- \int_{\mathbb{R}^{d}}X(t)e^{-\alpha f(X(t))}\mathcal{L}_{X(t)}(dx)\Big| \leq \frac{\zeta_{2}(t,\omega)}{N^{(1-\epsilon)/4}},\;\;\;\; a.s.
\end{align}
In the same manner, we can ascertain
\begin{align}\label{cboeq5.9}
    \Big|\sum_{i=1}^{N} e^{-\alpha f(X^{i}(t))} &- \int_{\mathbb{R}^{d}}e^{-\alpha f(X(t))}\mathcal{L}_{X(t)}(dx)\Big| \leq \frac{\zeta_{3}(t,\omega)}{N^{(1-\epsilon)/4}},\;\;\;\; a.s.,
\end{align}
where $\zeta_{3}(t,\omega)$ is an a.s. finite $\mathcal{F}_{t}$-measurable random variable. Substituting (\ref{cboeq5.6}), (\ref{cboeq5.8}) and (\ref{cboeq5.9}) in (\ref{cbo_eq_4.13}), we conclude that (\ref{cboeq5.7}) is true for sufficiently large $N$.
\end{proof}
\begin{remark}
From (\ref{y4.6}), we have $\lim_{N\rightarrow \infty}\int_{\mathbb{R}^{d}}|x|^{2}\mathcal{E}_{t}(dx) = \mathbb{E}|X(t)|^{2}$, $a.s.$, which is the strong law of large numbers for i.i.d. random variables $|X^{i}(t)|^{2}$. Also, the result of Lemma~\ref{cbolem5.1} can be treated as a law of large numbers which shows a.s. convergence of weighted average $\bar{X}^{\mathcal{E}_{t}}(t)$ (as compared to empirical average of (\ref{y4.6})) of i.i.d. particle system towards $\bar{X}(t)$ as $N \rightarrow \infty$. 
\end{remark}

Let $R>0$ be a sufficiently large real number. Let us fix a $t \in [0,T]$. Let us denote 
\begin{align}
    \tau_{1,R} = \inf\Big\{ s\geq 0\; ; \; \frac{1}{N}\sum\limits_{i=1}^{N}|X^{i}_{N}(s)|^{4}  \geq R \Big\},&\;\;\;\;      \tau_{2,R} = \inf\Big\{ s \geq 0\; ; \; \frac{1}{N}\sum\limits_{i=1}^{N}|X^{i}(s)|^{4} \geq R\Big\}, \\
    \tau_{R} & = \tau_{1,R}\wedge \tau_{2,R}, \label{cbo_neweq_4.23}
\end{align}
and 
\begin{align}
    \Omega_{1}(t) &= \{ \tau_{1,R} \leq t\} \cup \{ \tau_{2,R} \leq t \}, \label{cbo_eq_4.20}\\   \Omega_{2}(t) &= \Omega\backslash\Omega_{1}(t) = \{\tau_{1,R} > t\} \cap  \{ \tau_{2,R} >  t \}. \label{cbo_eq_4.21}
\end{align}

\begin{lemma}
Let Assumptions~\ref{cboh3.1}, \ref{cboh3.2}-\ref{cboasm1.4} be satisfied. Then, the following inequality holds for all $t \in [0,T]$:
\begin{align}
        \mathbb{E}\int_{0}^{t\wedge \tau_{R}} |\bar{X}_{N}(s) &- \bar{X}^{\mathcal{E}_{s}}(s)|^{2} ds 
        \leq CRe^{4\alpha K_{u}\sqrt{R}}\int_{0}^{t}\frac{1}{N}\sum\limits_{i=1}^{N}\mathbb{E}|X^{i}_{N}(s\wedge \tau_{R}) - X^{i}(s\wedge \tau_{R})|^{2}ds, \label{cbo_eq_4.23}
\end{align}
where $\tau_{R}$ is from (\ref{cbo_neweq_4.23}), $\bar{X}_{N}(s)$ is from (\ref{cbos1.7}), $\bar{X}^{\mathcal{E}_{s}}(s)$ is from (\ref{cboeq5.2}), $C > 0$ is independent of $N$ and $R$.
\end{lemma}
\begin{proof}
    We have
    \begin{align*}
         &|\bar{X}_{N}(s) - \bar{X}^{\mathcal{E}_{s}}(s)| = \bigg|\sum\limits_{i=1}^{N}X^{i}_{N}(s) \frac{e^{-\alpha f(X^{i}_{N}(s))}}{\sum_{j =1}^{N}e^{-\alpha f(X_{N}^{j}(s))}} - \sum\limits_{i=1}^{N}X^{i}(s) \frac{e^{-\alpha f(X^{i}(s))}}{\sum_{j =1}^{N}e^{-\alpha f(X^{j}(s))}}\bigg|\\ & \leq \Bigg|\frac{1}{N}\sum\limits_{i=1}^{N}\big(X_{N}^{i}(s) - X^{i}(s)\big)\frac{e^{-\alpha f(X_{N}^{i}(s))}}{\frac{1}{N}\sum_{j=1}^{N}e^{-\alpha f(X^{j}_{N}(s))}}\Bigg| +  \Bigg|\frac{\frac{1}{N}\sum_{i=1}^{N}X^{i}(s)\big(e^{-\alpha f(X_{N}^{i}(s))} - e^{-\alpha f(X^{i}(s))}\big)}{\frac{1}{N}\sum_{j=1}^{N}e^{-\alpha f(X^{j}_{N}(s))}}\Bigg| \\ & \;\;\;\;+\Bigg|\frac{1}{N}\sum_{i=1}^{N}X^{i}(s)e^{-\alpha f(X^{i}(s))}\bigg(\frac{1}{\frac{1}{N}\sum_{j=1}^{N}e^{-\alpha f(X^{j}_{N}(s))}} -  \frac{1}{\frac{1}{N}\sum_{j=1}^{N}e^{-\alpha f(X^{j}(s))}}\bigg)\Bigg|.
\end{align*}
Using the discrete Jensen inequality, we get
\begin{align}
             &|\bar{X}_{N}(s) - \bar{X}^{\mathcal{E}_{s}}(s)|  \leq C\Bigg(e^{\frac{\alpha}{N}\sum_{j=1}^{N}f(X^{j}_{N}(s))}\frac{1}{N}\sum_{i=1}^{N}|X^{i}_{N}(s) - X^{i}(s)|\nonumber \\ & \;\;\;\; +e^{\frac{\alpha }{N}\sum_{j=1}^{N}f(X^{j}_{N}(s))}\frac{1}{N}\sum_{i=1}^{N}|X^{i}(s)||e^{-\alpha f(X_{N}^{i}(s))} - e^{-\alpha f(X^{i}(s))}| \nonumber  \\ & \;\;\;\; +e^{\frac{\alpha }{N}\sum_{j=1}^{N}(f(X^{j}_{N}(s)) + f(X^{j}(s)))}\frac{1}{N}\sum_{i=1}^{N}|X^{i}(s)|\frac{1}{N}\sum_{j=1}^{N}|e^{-\alpha f(X_{N}^{j}(s))} - e^{-\alpha f(X^{j}(s))}|\Bigg),\label{cbo_eqn_4.26}
\end{align}
where $C $ is a positive constant independent of $N$.
Applying Assumptions~\ref{cboh3.2}-\ref{cboassu3.4}, the Cauchy-Bunyakowsky-Schwartz inequality and Young's inequality, $ab\leq a^{2}/2 + b^{2}/2$, $a,b>0$, we obtain
\begin{align}
             &|\bar{X}_{N}(s) - \bar{X}^{\mathcal{E}_{s}}(s)|  \leq  C\Bigg(e^{\frac{\alpha K_{u}}{N}\sum_{j=1}^{N}|X^{j}_{N}(s)|^{2}}\frac{1}{N}\sum_{i=1}^{N}|X^{i}_{N}(s) - X^{i}(s)| \nonumber\\ & \;\;\;\; +e^{\frac{\alpha K_{u}}{N}\sum_{j=1}^{N}|X^{j}_{N}(s)|^{2}}\frac{1}{N}\sum_{i=1}^{N}|X^{i}(s)|\big(1+ |X^{i}_{N}(s)| + |X^{i}(s)| \big)|X_{N}^{i}(s) - X^{i}(s)| \nonumber\\ & \;\;\;\; +e^{\frac{\alpha K_{u}}{N}\sum_{j=1}^{N}(|X^{j}_{N}(s)|^{2} + |X^{j}(s)|^{2})}\frac{1}{N}\sum_{i=1}^{N} |X^{i}(s)|\frac{1}{N}\sum_{j=1}^{N}\big(1+ |X^{j}_{N}(s)| + |X^{j}(s)| \big)|X_{N}^{j}(s) - X^{j}(s)| \Bigg)
            \nonumber\\ & \leq C\Bigg(e^{\frac{\alpha K_{u}}{N}\sum_{j=1}^{N}|X^{j}_{N}(s)|^{2}}\frac{1}{N}\sum_{i=1}^{N}|X^{i}_{N}(s) - X^{i}(s)| \nonumber\\ & \;\;\;\; + e^{\frac{\alpha K_{u}}{N}\sum_{j=1}^{N}(|X^{j}_{N}(s)|^{2} + |X^{j}(s)|^{2})}\frac{1}{N}\sum_{i=1}^{N}\big(1+|X^{i}_{N}(s)|^{2} + |X^{i}(s)|^{2}\big)|X_{N}^{i}(s) - X^{i}(s)|
          \nonumber  \\ & \;\;\;\; + e^{\frac{\alpha K_{u}}{N}\sum_{j=1}^{N}(|X^{j}_{N}(s)|^{2} + |X^{j}(s)|^{2})} \frac{1}{N}\sum_{i=1}^{N}|X^{i}(s)|^{2}\frac{1}{N}\sum\limits_{j=1}^{N}|X^{j}_{N}(s) - X^{j}(s)|
            \Bigg)\nonumber\\ & \leq C\Bigg(e^{\frac{\alpha K_{u}}{N}\sum_{j=1}^{N}|X^{j}_{N}(s)|^{2}}\frac{1}{N}\sum_{i=1}^{N}|X^{i}_{N}(s) - X^{i}(s)|   +  e^{\frac{\alpha K_{u}}{N}\sum_{j=1}^{N}(|X^{j}_{N}(s)|^{2} + |X^{j}(s)|^{2})}\nonumber\\ & \;\;\;\;\times\bigg(\frac{1}{N}\sum\limits_{i=1}^{N}\big(1+ |X_{N}^{i}(s)|^{2} + |X^{i}(s)|^{2}\big)^{2}\bigg)^{1/2}\bigg(\frac{1}{N}\sum\limits_{i=1}^{N}|X_{N}^{i}(s) - X^{i}(s)|^{2}\bigg)^{1/2}\Bigg). \label{cbo_neweq_4.28}
\end{align}
On squaring both sides, we ascertain
    \begin{align*}
         &|\bar{X}_{N}(s) - \bar{X}^{\mathcal{E}_{s}}(s)|^{2} \leq C\Bigg(e^{\frac{2\alpha K_{u}}{N}\sum_{j=1}^{N}|X^{j}_{N}(s)|^{2}}\frac{1}{N}\sum_{i=1}^{N}|X^{i}_{N}(s) - X^{i}(s)|^{2}   +  e^{\frac{2\alpha K_{u}}{N}\sum_{j=1}^{N}(|X^{j}_{N}(s)|^{2} + |X^{j}(s)|^{2})}\\ & \;\;\;\;\times\bigg(\frac{1}{N}\sum\limits_{i=1}^{N}\big(1+ |X_{N}^{i}(s)|^{2} + |X^{i}(s)|^{2}\big)^{2}\bigg)\bigg(\frac{1}{N}\sum\limits_{i=1}^{N}|X_{N}^{i}(s) - X^{i}(s)|^{2}\bigg)\Bigg). 
    \end{align*}
Using Holder's inequality, we have 
\begin{align*}
    \frac{1}{N}\sum_{j=1}^{N}(|X^{j}_{N}(s)|^{2} + |X^{j}(s)|^{2}) \leq \frac{2}{N^{1/2}}\bigg(\sum_{j=1}^{N}(|X^{j}_{N}(s)|^{4} + |X^{j}(s)|^{4})\bigg)^{1/2}.
\end{align*}
Therefore,
\begin{align*}
        \mathbb{E}\int_{0}^{t\wedge \tau_{R}} |\bar{X}_{N}(s) &- \bar{X}^{\mathcal{E}_{s}}(s)|^{2} ds 
        \leq CRe^{4\alpha K_{u}\sqrt{R}}\int_{0}^{t}\frac{1}{N}\sum\limits_{i=1}^{N}\mathbb{E}|X^{i}_{N}(s\wedge \tau_{R}) - X^{i}(s\wedge \tau_{R})|^{2}ds,
\end{align*}
where $C > 0$ is  independent of $N$ and $R$. 
\end{proof}

\begin{lemma}
Let Assumptions~\ref{cboh3.1}, \ref{cboh3.2}-\ref{cboasm1.4} be satisfied. Then, the following inequality holds for all $t \in [0,T]$:
 \begin{align}\label{cbo_eq_4.28}
 \mathbb{E}\int_{0}^{t\wedge \tau_{R}} |\bar{X}^{\mathcal{E}_{s}}(s) &- \bar{X}(s)|^{2}ds   \leq C\frac{e^{2 \alpha K_{u} \sqrt{R}}}{N},
 \end{align}
 where $\tau_{R}$ is from (\ref{cbo_neweq_4.23}), $\bar{X}^{\mathcal{E}_{s}}(s)$ is from (\ref{cboeq5.2}), $\bar{X}(s)$ is from (\ref{eqcbo2.12}), $C > 0$ is independent of $N$ and  $R$.
\end{lemma}
\begin{proof}
We have
\begin{align*}
    |\bar{X}^{\mathcal{E}_{s}}(s) &- \bar{X}(s)| 
    =  \bigg|\sum_{i=1}^{N}X^{i}(s) \frac{e^{-\alpha f(X^{i}(s))}}{\sum_{j=1}^{N}e^{-\alpha f(X^{j}(s))}} - \int_{\mathbb{R}^{d}}x\frac{e^{-\alpha f(x)}}{\int_{\mathbb{R}^{d}}e^{-\alpha f(x)} \mathcal{L}_{X(s)}(dx)}\mathcal{L}_{X(s)}(dx)\bigg| \\ & \leq   \frac{1}{\frac{1}{N}\sum_{j=1}^{N}e^{-\alpha f(X^{j}(s))}}\Bigg|\frac{1}{N}\sum_{i=1}^{N}\bigg(X^{i}(s)e^{-\alpha f(X^{i}(s))}- \int_{\mathbb{R}^{d}}x e^{-\alpha f(x)}\mathcal{L}_{X(s)}(dx)\bigg)\Bigg|\\ & \;\;\;\; + \Bigg|\int_{\mathbb{R}^{d}}x e^{-\alpha f(x)}\mathcal{L}_{X(s)}(dx)\frac{\frac{1}{N}\sum_{j=1}^{N}\Big(e^{-\alpha f(X^{j}(s))} -  \int_{\mathbb{R}^{d}}e^{-\alpha f(x)} \mathcal{L}_{X(s)}(dx)\Big)}{\frac{1}{N}\sum_{j=1}^{N}e^{-\alpha f(X^{j}(s))}   \int_{\mathbb{R}^{d}}e^{-\alpha f(x)} \mathcal{L}_{X(s)}(dx)}\Bigg|.  
\end{align*}
Using Jensen's inequality and squaring both sides, we get
\begin{align*}
    |\bar{X}^{\mathcal{E}_{s}}(s) &- \bar{X}(s)|^{2}  \leq   Ce^{\frac{2\alpha }{N}\sum_{j=1}^{N}f(X^{j}(s))} \bigg|\frac{1}{N}\sum_{i=1}^{N}\Big(X^{i}(s)e^{-\alpha f(X^{i}(s))}- \mathbb{E}\big(X(s)e^{-\alpha f(X(s))}\big)\Big)\bigg|^{2}\\ & \;\;\;\; + Ce^{\frac{2\alpha }{N}\sum_{j=1}^{N}f(X^{j}(s))}e^{2\alpha \mathbb{E}f(X(s))}(\mathbb{E}|X(s)|)^{2}\bigg|\frac{1}{N}\sum_{j=1}^{N}\Big(e^{-\alpha f(X^{j}(s))} -  \mathbb{E}\big(e^{-\alpha f(X(s))}\big)\Big)\bigg|^{2},  
\end{align*} 
where $C$ is a positive constant independent of $N$.
Applying Assumption~\ref{cboassu3.4}, we ascertain
\begin{align*}
    |\bar{X}^{\mathcal{E}_{s}}(s) &- \bar{X}(s)|^{2}  \leq   Ce^{\frac{2\alpha K_{u}}{N}\sum_{j=1}^{N}|X^{j}(s)|^{2}} \bigg|\frac{1}{N}\sum_{i=1}^{N}\Big(X^{i}(s)e^{-\alpha f(X^{i}(s))}- \mathbb{E}\big(X(s)e^{-\alpha f(X(s))}\big)\Big)\bigg|^{2}\\ & \;\;\;\; + Ce^{\frac{2\alpha K_{u}}{N}\sum_{j=1}^{N}|X^{j}(s)|^{2}}e^{2\alpha K_{u}\mathbb{E}|X(s)|^{2}}(\mathbb{E}|X(s)|)^{2}\bigg|\frac{1}{N}\sum_{j=1}^{N}\Big(e^{-\alpha f(X^{j}(s))} -  \mathbb{E}\big(e^{-\alpha f(X(s))}\big)\Big)\bigg|^{2}. 
\end{align*}    
Hence, using Theorem~\ref{cbolem3.6}, we obtain
\begin{align*}
    \mathbb{E}\int_{0}^{t\wedge \tau_{R}} |\bar{X}^{\mathcal{E}_{s}}(s) &- \bar{X}(s)|^{2}ds  \leq Ce^{2 \alpha K_{u} \sqrt{R}} \mathbb{E}\int_{0}^{t\wedge \tau_{R}} \bigg|\frac{1}{N}\sum_{i=1}^{N}\Big(X^{i}(s)e^{-\alpha f(X^{i}(s))}- \mathbb{E}\big(X(s)e^{-\alpha f(X(s))}\big)\Big)\bigg|^{2}ds\\ & \;\;\;\; + Ce^{2 \alpha K_{u} \sqrt{R}}\mathbb{E}\int_{0}^{t\wedge \tau_{R}}\bigg|\frac{1}{N}\sum_{j=1}^{N}\Big(e^{-\alpha f(X^{j}(s))} -  \mathbb{E}\big(e^{-\alpha f(X(s))}\big)\Big)\bigg|^{2}ds   
    \\ & \leq  Ce^{2 \alpha K_{u} \sqrt{R}}\int_{0}^{t} \mathbb{E}\bigg|\frac{1}{N}\sum_{i=1}^{N}U^{i}_{1}(s\wedge  \tau_{R})\bigg|^{2}ds + Ce^{2 \alpha K_{u} \sqrt{R}}\int_{0}^{t}\mathbb{E}\bigg|\frac{1}{N}\sum_{i=1}^{N} U^{i}_{2}(s\wedge \tau_{R})\bigg|^{2}ds,  
\end{align*}
where $U_{1}^{i}(s \wedge \tau_{R}) = X^{i}(s\wedge \tau_{R})e^{-\alpha f(X^{i}(s\wedge \tau_{R}))} - \mathbb{E}\big(X(s\wedge \tau_{R})e^{-\alpha f(X(s\wedge \tau_{R}))}\big) $, $U_{2}^{i}(s\wedge \tau_{R}) = e^{-\alpha f(X^{i}(s))} -  \mathbb{E}\big(e^{-\alpha f(X(s))}\big)$, and $C$  is independent of $N$ and $R$. 
We have
\begin{align*}
    \mathbb{E}\bigg|\frac{1}{N}\sum_{i=1}^{N}U^{i}_{1}(s\wedge  \tau_{R})\bigg|^{2} = 
      \frac{1}{N^{2}}\sum\limits_{i=1}^{N}\mathbb{E}|U_{1}^{i}(s\wedge \tau_{R})|^{2} + \frac{1}{N^{2}}\sum_{\substack{i,j=1 ,\; i\neq j }}^{N}\mathbb{E}\big(U^{i}_{1}(s\wedge \tau_{R})\cdot U_{1}^{j}(s\wedge \tau_{R})\big).
\end{align*}
Note that $\mathbb{E}\big(U^{i}_{1}(s)\cdot U_{1}^{j}(s)\big) = 0 $ for $i\neq j$ and $s\wedge \tau_{R}$ is a bounded stopping time then $\mathbb{E}\big(U^{i}_{1}(s\wedge \tau_{R})\cdot U_{1}^{j}(s\wedge \tau_{R})\big) = 0$ for $i\neq j$ because of Doob's optional stopping theorem \cite[Theorem 2.2.1]{cbos11}. Using Theorem~\ref{cbolem3.6}, we deduce
\begin{align}\label{cbo_eq_4.26}
    \mathbb{E}\bigg|\frac{1}{N}\sum_{i=1}^{N}U^{i}_{1}(s\wedge  \tau_{R})\bigg|^{2} \leq \frac{C}{N}, 
\end{align}
where $C$ is independent of $N$. In the similar manner, we can obtain
\begin{align}\label{cbo_eq_4.27}
    \mathbb{E}\bigg|\frac{1}{N}\sum_{i=1}^{N} U^{i}_{2}(s\wedge \tau_{R})\bigg|^{2} \leq \frac{C}{N},
\end{align}
where $C$ is independent of $N$. Using (\ref{cbo_eq_4.26}) and (\ref{cbo_eq_4.27}), we get the following estimate:
 \begin{align*}
 \mathbb{E}\int_{0}^{t\wedge \tau_{R}} |\bar{X}^{\mathcal{E}_{s}}(s) &- \bar{X}(s)|^{2}ds   \leq C\frac{e^{2 \alpha K_{u} \sqrt{R}}}{N},
 \end{align*}
 where $C$ is independent of $N$ and $R$.
\end{proof}

\begin{theorem}\label{cbo_thrm4.5}
Let Assumptions~\ref{cboh3.1}, \ref{cboh3.2}-\ref{cboasm1.4} be satisfied. Let $X_{N}^{i}(t)$  solve (\ref{cboeq1.8}). Let $X^{i}(t)$ represent independent processes which solve (\ref{cbomfsdep}). Let us assume that $X^{i}_{N}(0) = X^{i}(0) $, a.s., $i=1\dots,N$. Let  $\mathbb{E}|Z|^{4} \leq C$, $\sup_{i =1,\dots,N}\mathbb{E}|X^{i}(0)|^{4} \leq C$, and $\sup_{i=1,\dots,N}\mathbb{E}|X^{i}_{N}(0)|^{4} \leq C$. Then, the following mean-square convergence result holds for all $t \in [0,T]$:
\begin{align}
\lim\limits_{N \rightarrow \infty }\sup_{i =1,\dots,N}\mathbb{E}|X_{N}^{i}(t) - X^{i}(t)|^{2} = 0. 
\end{align}

\end{theorem}
\begin{proof}
Let $t \in (0,T]$. We can write
\begin{align*}
    \mathbb{E}|X_{N}^{i}(t) - X^{i}(t)|^{2} &= \mathbb{E}\big(|X_{N}^{i}(t) - X^{i}(t)|^{2}I_{\Omega_{1}(t)}\big) + \mathbb{E}\big( |X_{N}^{i}(t) - X^{i}(t)|^{2}I_{\Omega_{2}(t)}\big) \\ 
    & =: E_{1}(t) + E_{2}(t),
    \end{align*}
    where $\Omega_{1}(t)$ and $\Omega_{2}(t)$ are from (\ref{cbo_eq_4.20}) and (\ref{cbo_eq_4.21}), respectively.
 Using the Cauchy-Bunyakowsky-Shwartz inequality and Chebyshev's inequality, we obtain
\begin{align*}
    E_{1}(t) :&= \mathbb{E}\big(|X_{N}^{i}(t) - X^{i}(t)|^{2}I_{\Omega_{1}(t)}\big)
    \leq \big(\mathbb{E}|X_{N}^{i}(t) - X^{i}(t)|^{4}\big)^{1/2}\big(\mathbb{E}I_{\Omega_{1}(t)}\big)^{1/2}
    \\ & \leq C \big(\mathbb{E}|X_{N}^{i}(t)|^{4} + \mathbb{E}|X^{i}(t)|^{4}\big)^{1/2} \bigg(\frac{1}{RN}\sum\limits_{i=1}^{N}\mathbb{E}\sup_{0\leq s \leq t}|X^{i}_{N}(s)|^{4} + \frac{1}{RN}\sum\limits_{i=1}^{N}\mathbb{E}\sup_{0\leq s\leq t}|X^{i}(s)|^{4}\bigg)^{1/2}. 
\end{align*}
We get the following estimate for $E_{1}(t)$ by applying 
Lemma~\ref{cbolemma3.3} and Theorem~\ref{cbolem3.6}:
\begin{align}\label{cbo_neqeq_4.33}
    E_{1}(t) \leq  \frac{C}{R}, 
\end{align}
where $C$ is a positive constant independent of $N$ and $R$.

Now, we estimate $E_{2}(t)$. 
We have $\mathbb{E}(|X_{N}^{i}(t) - X^{i}(t)|^{2}I_{\Omega_{2}(t)}) \leq \mathbb{E}(|X_{N}^{i}(t\wedge \tau_{R}) -X^{i}(t \wedge \tau_{R})|^{2}) $. Using Ito's formula, we have
\begin{align}
|X_{N}^{i}&(t\wedge \tau_{R}) - X^{i}(t\wedge \tau_{R})|^{2} = |X^{i}_{N}(0) - X^{i}(0)|^{2} \nonumber\\ & \;\;-  2\mathbb{E}\int_{0}^{t\wedge \tau_{R}}\beta(s)(X_{N}^{i}(s) - X^{i}(s))\cdot (X_{N}^{i}(s) - \bar{X}_{N}(s) - X^{i}(s) + \bar{X}(s))ds \nonumber \\ &
\;\;    + 2\int_{0}^{t\wedge \tau_{R}}\sigma^{2}(s)|\diag(X_{N}^{i}(s) - \bar{X}_{N}(s) -X^{i}(s) + \bar{X}(s))|^{2}ds \nonumber\\ & \;\;+2\sqrt{2}\int_{0}^{t\wedge \tau_{R}}\sigma(s)\big((X_{N}^{i}(s) -X^{i}(s))\cdot\diag(X_{N}^{i}(s) - \bar{X}_{N}(s)- X^{i}(s) +\bar{X}(s))dW^{i}(s)\big) \nonumber        \\ & \;\;+\int_{0}^{t \wedge \tau_{R}}\int_{\mathbb{R}^{d}}\Big(|X_{N}^{i}(s^{-}) - X^{i}(s^{-}) + \gamma(s)\diag(X_{N}^{i}(s^{-}) - \bar{X}_{N}(s^{-}))z \nonumber \\ & \;\;\;\;\;\;\;\;- \gamma(s)\diag(X^{i}(s^{-}) - \bar{X}(s^{-}))z|^{2}- |X^{i}_{N}(s^{-}) - X^{i}(s^{-})|^{2}\Big)\mathcal{N}^{i}(ds,dz). \label{cbo_neweq_4.34} 
\end{align}
The Cauchy-Bunyakowsky-Schwartz inequality and Young's inequality provide the following estimates:
\begin{align}
    &(X_{N}^{i}(s) - X^{i}(s))\cdot (X_{N}^{i}(s) - \bar{X}_{N}(s) - X^{i}(s) + \bar{X}(s)) \leq C(|X^{i}_{N}(s) - X^{i}(s)|^{2} + |\bar{X}_{N}(s) - \bar{X}(s)|^{2}), \label{cbo_neweq_4.35}\\
    &|\diag(X_{N}^{i}(s) - \bar{X}_{N}(s) -X^{i}(s) + \bar{X}(s))|^{2} \leq C(|X^{i}_{N}(s) - X^{i}(s)|^{2} + |\bar{X}_{N}(s) - \bar{X}(s)|^{2}), 
\end{align}
and 
\begin{align}
    &\Big(|X_{N}^{i}(s^{-}) - X^{i}(s^{-}) + \gamma(s)\diag(X_{N}^{i}(s^{-}) - \bar{X}_{N}(s^{-}))z - \gamma(s)\diag(X^{i}(s^{-})\nonumber \\ & \;\;\;\; - \bar{X}(s^{-}))z|^{2}- |X^{i}_{N}(s^{-}) - X^{i}(s^{-})|^{2}\Big) = \gamma^{2}(s)|\big((X^{i}_{N}(s^{-}) - \bar{X}_{N}(s^{-}) - X^{i}(s^{-}) + \bar{X}(s^{-}))\cdot z\big)|^{2} \nonumber \\ & \;\;\;\; + 2\gamma(s)\Big( \big(X_{N}^{i}(s^{-}) - X^{i}(s^{-})\big)\cdot \big(\diag(X_{N}^{i}(s^{-}) - \bar{X}_{N}(s^{-}) - X^{i}(s^{-}) + \bar{X}(s^{-}))z\big)\Big) \nonumber\\ & \leq  C(|X^{i}_{N}(s^{-}) - X^{i}(s^{-})|^{2} + |\bar{X}_{N}(s^{-}) - \bar{X}(s^{-})|^{2})|z|^{2}\nonumber \\ & \;\;\;\; +  2\gamma(s)\Big( \big(X_{N}^{i}(s^{-}) - X^{i}(s^{-})\big)\cdot \big(\diag(X_{N}^{i}(s^{-}) - \bar{X}_{N}(s^{-}) - X^{i}(s^{-}) + \bar{X}(s^{-}))z\big)\Big). \label{cbo_neweq_4.38}
\end{align}
Taking expectations on both sides of (\ref{cbo_neweq_4.34}), using estimates (\ref{cbo_neweq_4.35})-(\ref{cbo_neweq_4.38}) and applying Doob's optional stopping theorem \cite[Theorem 2.2.1]{cbos11}, we get
\begin{align}
 &\mathbb{E}|X_{N}^{i}(t\wedge \tau_{R}) - X^{i}(t\wedge \tau_{R})|^{2} \leq \mathbb{E}|X_{N}^{i}(0) - X^{i}(0)|^{2} \nonumber 
 \\ & \;\;\;\;   + C\mathbb{E}\int_{0}^{t\wedge \tau_{R}}\big(|X_{N}^{i}(s) - X^{i}(s)|^{2} + |\bar{X}_{N}(s) - \bar{X}(s)|^{2}\big) ds   \nonumber 
\\ & \;\;\;\; + C\mathbb{E}\int_{0}^{t\wedge \tau_{R}}\int_{\mathbb{R}^{d}}(|X^{i}_{N}(s) - X^{i}(s)|^{2} + |\bar{X}_{N}(s) - \bar{X}(s)|^{2})|z|^{2}\rho_{z}(z)dz ds \nonumber
\\ & \leq \mathbb{E}|X_{N}^{i}(0) - X^{i}(0)|^{2}   + C\mathbb{E} \int_{0}^{t\wedge \tau_{R}}|X_{N}^{i}(s) - X^{i}(s)|^{2}ds \nonumber\\ & \;\;\;\;  + C\mathbb{E}\int_{0}^{t\wedge \tau_{R}} |\bar{X}_{N}(s) - \bar{X}^{\mathcal{E}_{s}}(s)|^{2} ds    + C\mathbb{E}\int_{0}^{t\wedge \tau_{R}}|\bar{X}^{\mathcal{E}_{s}}(s) - \bar{X}(s)|^{2}  ds.   \label{cbo_eq_4.22} \end{align}

Substituting (\ref{cbo_eq_4.23}) and (\ref{cbo_eq_4.28}) in (\ref{cbo_eq_4.22}), we obtain
\begin{align*}
    \mathbb{E}&\big(|X_{N}^{i}(t\wedge \tau_{R}) - X^{i}(t\wedge \tau_{R})|^{2}\big) \leq \mathbb{E}|X_{N}^{i}(0) - X^{i}(0)|^{2} \\ & \;\;\;\; +  CRe^{4\alpha K_{u}\sqrt{R}}\int_{0}^{t}\frac{1}{N}\sum\limits_{i=1}^{N}\mathbb{E}\big(|X^{i}_{N}(s\wedge \tau_{R}) - X^{i}(s\wedge \tau_{R})|^{2}\big)ds + C\frac{e^{2 \alpha K_{u} 
    \sqrt{R}}}{N},
\end{align*}
where $C>0$ is independent of $N$ and $R$.
Taking supremum over $i =1,\dots, N$, we get
    \begin{align*}
   \sup_{i=1,\dots,N}\mathbb{E}\big(|&X_{N}^{i}(t\wedge \tau_{R}) - X^{i}(t\wedge \tau_{R})|^{2}\big) \leq \sup_{i=1,\dots,N}\mathbb{E}|X_{N}^{i}(0) - X^{i}(0)|^{2} \\ & +  CRe^{4\alpha K_{u}\sqrt{R}}\int_{0}^{t}\sup_{i=1,\dots,N}\mathbb{E}\big(|X^{i}_{N}(s\wedge \tau_{R}) - X^{i}(s\wedge \tau_{R})|^{2}\big)ds  + C\frac{e^{2 \alpha K_{u} \sqrt{R}}}{N}.
   \end{align*}
Using Gr\"{o}nwall's inequality, we have
\begin{align}
     \sup_{i=1,\dots,N}\mathbb{E}\big(|X_{N}^{i}(t\wedge \tau_{R})& - X^{i}(t\wedge \tau_{R})|^{2}\big) \leq \frac{C}{N}e^{CRe^{4\alpha K_{u}\sqrt{R}}}e^{2 \alpha K_{u} R} \leq \frac{C}{N}e^{e^{C_{u}\sqrt{R}}},\label{cbo_eqn_4.30}
\end{align}
where $C>0$ and $C_{u}>0$ are constants independent of $N$ and $R$. In the above calculations, we have used the facts that $R < e^{2\alpha K_{u}\sqrt{R}}$ and $2\alpha K_{u}\sqrt{R} < e^{2\alpha K_{u}\sqrt{R}}$ for sufficiently large $R$. 

We choose $R = \frac{1}{C_{u}^{2}}(\ln{(\ln({N^{1/2})})})^{2} $.  Therefore,
\begin{align*}
\sup_{i=1,\dots,N}  \mathbb{E}(|X_{N}^{i}(t) - X^{i}(t)|^{2}I_{\Omega_{2}(t)}) \leq    \sup_{i=1,\dots,N}\mathbb{E}\big(|X_{N}^{i}(t\wedge \tau_{R})& - X^{i}(t\wedge \tau_{R})|^{2}\big) \leq  \frac{C}{N^{1/2}}, 
\end{align*}
which implies
\begin{align}\label{cbo_eq_4.31}
\lim\limits_{N\rightarrow \infty} \sup_{i=1,\dots,N}  \mathbb{E}(|X_{N}^{i}(t) - X^{i}(t)|^{2}I_{\Omega_{2}(t)}) = \lim\limits_{N\rightarrow \infty} \sup_{i=1,\dots,N}\mathbb{E}\big(|X_{N}^{i}(t\wedge \tau_{R})& - X^{i}(t\wedge \tau_{R})|^{2}\big) = 0.
\end{align}
The term (\ref{cbo_neqeq_4.33}) and the choice of $R$  provide the following estimate:
\begin{align*}
    \mathbb{E}\big(|X_{N}^{i}(t) - X^{i}(t)|^{2}I_{\Omega_{1}}(t)\big) \leq \frac{C}{R} \leq \frac{C}{(\ln{(\ln({N^{1/2})}))^{2}}},
\end{align*}
where $C>0$ is independent of $N$ and $R$. This yields
\begin{align}\label{cbo_eq_4.24}
    \lim_{N\rightarrow \infty}\sup_{i = 1,\dots,N }\mathbb{E}\big(|X^{i}_{N}(t) - X^{i}(t)|^{2}I_{\Omega_{1}(t)}\big) = 0.
\end{align}
As a  consequence of (\ref{cbo_eq_4.31}) and (\ref{cbo_eq_4.24}) , we get
\begin{align*}
    \lim_{N\rightarrow \infty}\sup_{i=1,\dots,N}\mathbb{E}|X_{N}^{i}(t) - X^{i}(t)|^{2} = 0,
\end{align*}
for all $t \in [0,T]$.
\end{proof}
\begin{remark}
It is not difficult to see from the above theorem that the empirical measure of the particle system (\ref{cboeq1.8}) converges to the law of the mean-field SDEs (\ref{cbomfsdep}) in $2-$Wasserstein metric, i.e. for all $t \in [0,T]$:
\begin{align}
    \lim_{N\rightarrow \infty}\mathcal{W}_{2}^{2}(\mathcal{E}_{t}^{N}, \mathcal{L}_{X(t)}) = 0,
\end{align}
where $\mathcal{E}_{t}^{N} = \frac{1}{N}\sum_{i=1}^{N}\delta_{X^{i}_{N}(t)} $.
\end{remark}
\begin{remark}
 Theorem~\ref{cbo_thrm4.5} implies weak convergence of the empirical measure, $\mathcal{E}_{t}^{N}$ of interacting particle system towards $\mathcal{L}_{X(t)}$ which is the law of the mean-field limit process $X(t)$ (see \cite{cbo35,cbo29}). 
\end{remark}

\subsection{Convergence of the numerical scheme}\label{cbo_conv_ns}
To implement  the particle system (\ref{cbos1.6}), we have proposed to utilize the Euler scheme introduced in Section~\ref{subsec_implemen}. The jump-diffusion SDEs  (\ref{cbos1.6}), governing interacting particle system, have locally Lipschitz and linearly growing coefficients. Due to non-global Lipschitzness of the coefficients, it is not straightforward to deduce  convergence of the Euler scheme to (\ref{cbos1.6}). In this section, we go one step further and prove this convergence result uniform in $N$. To this end, we introduce the function $\kappa_{h}(t) = t_{k}$, $t_{k} \leq t < t_{k+1}$, where $ 0=t_{0}<\dots<t_{n} = T$ is a uniform partition of $[0,T]$,  i.e. $t_{k+1} - t_{k} = h$ for all $k=0,\dots,n-1$. We write the continuous version of the numerical scheme (\ref{cbo_dis_ns}) as follows:
\begin{align}\label{cboeq5.20}
    dY^{i}_{N}(t) &= -\beta(t)(Y^{i}_{N}(\kappa_{h}(t)) - \bar{Y}_{N}(\kappa_{h}(t)))dt + \sqrt{2}\sigma(t)\diag(Y^{i}_{N}(\kappa_{h}(t)) - \bar{Y}_{N}(\kappa_{h}(t)))dW^{i}(t)\nonumber \\ & \;\;\;\; + \int_{\mathbb{R}^{d}}\diag(Y^{i}_{N}(\kappa_{h}(t)) - \bar{Y}_{N}(\kappa_{h}(t)))z\mathcal{N}^{i}(dt,dz).
\end{align}
In this section, our aim is to show mean-square convergence of $Y^{i}_{N}(t)$ to $X^{i}_{N}(t)$ uniformly in $N$, i.e.
\begin{align}
    \lim_{h\rightarrow 0}\sup_{i=1,\dots,N}\mathbb{E}|Y^{i}_{N}(t) - X^{i}_{N}(t)|^{2} = 0,
\end{align}
where $h \rightarrow 0$ means that keeping $T$ fixed the time-step of uniform partition of $[0,T]$ goes to zero.

Let Assumptions~\ref{cboh3.1}-\ref{cboasu1.1} hold. Let  $\mathbb{E}|Y^{i}_{N}(0)|^{2} < \infty$ and $\mathbb{E}|Z|^{2} < \infty$, then the particle system (\ref{cboeq5.20}) is well-posed (cf. Theorem~\ref{cbo_thrm_3.2}). Moreover, if $\mathbb{E}|Y^{i}_{N}(0)|^{2p} <\infty $ and $\mathbb{E}|Z|^{2p} < \infty$ for some $p \geq 1$, then, due to Lemma~\ref{cbolemma3.3}, the following holds:
\begin{align}
    \mathbb{E}\sup_{0\leq t\leq T}|Y^{i}_{N}(t)|^{2p} \leq K, \label{cbo_neweq_4.45}
\end{align}
where we cannot  say that $K$ is independent of $h$. However, to prove the convergence of numerical scheme we need the uniform in $h$ and $N$ moment bound, which we prove in the next lemma.
\begin{lemma}\label{cbo_lem4.6}
Let Assumptions~\ref{cboh3.1}, \ref{cboh3.2}-\ref{cboasm1.4} hold. Let  $p \geq 1$, $\mathbb{E}|Y^{i}_{N}(0)|^{2p} < \infty$ and $\mathbb{E}|Z|^{2p} < \infty$. Then, the following holds:
\begin{align}
    \sup_{i=1,\dots,N}\mathbb{E}\sup_{0\leq t\leq T}|Y^{i}_{N}(t)|^{2p} \leq K_{d},
\end{align}
where $K_{d}$ is a positive constant independent of $h$ and $N$.
\end{lemma}
\begin{proof}
Let $p$ be a positive integer. Using Ito's formula, the Cauchy-Bunyakowsky-Schwartz inequality and Young's inequality, we have
\begin{align*}
    |Y^{i}_{N}&(t)|^{2p} = |Y^{i}_{N}(0)|^{2p} - 2p\int_{0}^{t}\beta(s)|Y^{i}_{N}(s)|^{2p-2}\big(Y^{i}_{N}(s)\cdot (Y^{i}_{N}(\kappa_{h}(s)) - \bar{Y}_{N}(\kappa_{h}(s)))\big)ds \\
    & \;\;\;\; + 2\sqrt{2}p\int_{0}^{t}\sigma(s)|Y^{i}_{N}(s)|^{2p-2}\big(Y^{i}_{N}(s)\cdot \diag(Y^{i}_{N}(\kappa_{h}(s)) - \bar{Y}_{N}(\kappa_{h}(s)))dW^{i}(s)\big)\\ & \;\;\;\; + 4p(p-1)\int_{0}^{t}\sigma^{2}(s)|Y^{i}_{N}(s)|^{2p-4}|\diag(Y^{i}_{N}(\kappa_{h}(s)) - \bar{Y}_{N}(\kappa_{h}(s))) Y^{i}_{N}(s)|^{2}ds\\ & \;\;\;\;+ 2p\int_{0}^{t}\sigma^{2}(s)|Y^{i}_{N}(s)|^{2p-2}|\diag(Y^{i}_{N}(\kappa_{h}(s)) - \bar{Y}_{N}(s)|^{2}ds \\ & \;\;\;\; + \int_{0}^{t}\int_{\mathbb{R}^{d}}\Big(|Y^{i}_{N}(s^{-}) + \gamma(s)\diag(Y^{i}_{N}(\kappa_{h}(s)) - \bar{Y}_{N}(\kappa_{h}(s)))z|^{2p} - |Y^{i}_{N}(s^{-})|^{2p}\Big)\mathcal{N}^{i}(ds,dz)
    \\ & \leq |Y^{i}_{N}(0)|^{2p} + C \int_{0}^{t}(|Y^{i}_{N}(s)|^{2p} + |Y^{i}_{N}(\kappa_{h}(s))|^{2p}+|\bar{Y}_{N}(\kappa_{h}(s))|^{2p})ds\\ &\;\;\;\; + 2\sqrt{2}p\int_{0}^{t}\sigma(s)|Y^{i}_{N}(s)|^{2p-2}(Y^{i}_{N}(s)\cdot\diag(Y^{i}_{N}(\kappa_{h}(s)) - \bar{Y}_{N}(\kappa_{h}(s)))dW^{i}(s)) \\
    & \;\;\;\;+ C\int_{0}^{t}\int_{\mathbb{R}^{d}}\Big(|Y^{i}_{N}(s^{-})|^{2p} + (|Y^{i}_{N}(\kappa_{h}(s))|^{2p} + |\bar{Y}_{N}(\kappa_{h}(s))|^{2p})(1+|z|^{2p})\Big)\mathcal{N}^{i}(ds,dz).
\end{align*}
First taking supremum over $0\leq t\leq T$ and then expectation, we obtain
\begin{align*}
    \mathbb{E}&\sup_{0\leq t\leq T}|Y^{i}_{N}(t)|^{2p} \leq  \mathbb{E}|Y^{i}_{N}(0)|^{2p} + C\mathbb{E}\int_{0}^{T}\Big(|Y^{i}_{N}(s)|^{2p} + |Y^{i}_{N}(\kappa_{h}(s))|^{2p} + |\bar{Y}_{N}(\kappa_{h}(s))|^{2p}\Big) ds \\ & +2\sqrt{2}p\mathbb{E}\sup_{0\leq t\leq T}\bigg|\int_{0}^{t}\sigma(s)|Y^{i}_{N}(s)|^{2p-2}(Y^{i}_{N}(s)\cdot \diag(Y^{i}_{N}(\kappa_{h}(s))-\bar{Y}_{N}(\kappa_{h}(s)))dW^{i}(s))\bigg| \\ & +C\mathbb{E}\int_{0}^{T}\int_{\mathbb{R}^{d}}\Big(|Y^{i}_{N}(s^{-})|^{2p} + (|Y^{i}_{N}(\kappa_{h}(s))|^{2p} + |\bar{Y}_{N}(\kappa_{h}(s))|^{2p})(1+|z|^{2p})\Big)\mathcal{N}^{i}(ds,dz),
\end{align*}
where $C$ is independent of $h$ and $N$.
Using the Burkholder-Davis-Gundy inequality (note that we can apply this inequality due to (\ref{cbo_neweq_4.45})) and the fact that $\mathbb{E}|Z|^{2p} < \infty$, we get
\begin{align*}
    \mathbb{E}&\sup_{0\leq t\leq T}|Y^{i}_{N}(t)|^{2p} \leq  \mathbb{E}|Y^{i}_{N}(0)|^{2p} + C\mathbb{E}\int_{0}^{T}\Big(|Y^{i}_{N}(s)|^{2p} + |Y^{i}_{N}(\kappa_{h}(s))|^{2p} + |\bar{Y}_{N}(\kappa_{h}(s))|^{2p}\Big) ds \\ & \;\;\;\;+C\mathbb{E}\bigg(\int_{0}^{T}|Y^{i}_{N}(s)|^{4p-4}\big(Y^{i}_{N}(s)\cdot(Y^{i}_{N}(\kappa_{h}(s))-\bar{Y}_{N}(\kappa_{h}(s)))\big)^{2}ds\bigg)^{1/2} \\ & \;\;\;\;+C\mathbb{E}\int_{0}^{T}\int_{\mathbb{R}^{d}}\Big(|Y^{i}_{N}(s)|^{2p} + (|Y^{i}_{N}(\kappa_{h}(s))|^{2p} + |\bar{Y}_{N}(\kappa_{h}(s))|^{2p})(1+|z|^{2p})\Big)\rho_{z}(z)dzds
    \\  &\leq  \mathbb{E}|Y^{i}_{N}(0)|^{2p} + C\mathbb{E}\int_{0}^{T}\Big(|Y^{i}_{N}(s)|^{2p} + |Y^{i}_{N}(\kappa_{h}(s))|^{2p} + |\bar{Y}_{N}(\kappa_{h}(s))|^{2p}\Big) ds \\ &\;\;\;\; +\mathbb{E}\sup_{0\leq t\leq T}|Y^{i}_{N}(t)|^{2p-1}\bigg(\int_{0}^{T}|Y^{i}_{N}(\kappa_{h}(s))-\bar{Y}_{N}(\kappa_{h}(s))|^{2}ds\bigg)^{1/2}.
\end{align*}
Applying Young's inequality  and Holder's inequality,
we ascertain
\begin{align}
    &\mathbb{E}\sup\limits_{0\leq t\leq T}|Y^{i}_{N}(t)|^{2p} \leq \mathbb{E}|Y_{N}^{i}(0)|^{2p} + C\int_{0}^{T}(|Y^{i}_{N}(s)|^{2p} + |Y^{i}_{N}(\kappa_{h}(s))|^{2p} + |\bar{Y}_{N}(\kappa_{h}(s))|^{2p}) ds \nonumber \\ & \;\;\;\; + \frac{1}{2}\mathbb{E}\sup\limits_{0\leq t\leq T}|Y^{i}_{N}(t)|^{2p} + C\mathbb{E}\Big(\int_{0}^{T}|Y^{i}_{N}(\kappa_{h}(s)) - \bar{Y}_{N}(\kappa_{h}(s))|^{2}ds\Big)^{p} \nonumber\\ &  \leq \mathbb{E}|Y_{N}^{i}(0)|^{2p} + C\int_{0}^{t}(|Y^{i}_{N}(s)|^{2p} + |Y^{i}_{N}(\kappa_{h}(s))|^{2p} + |\bar{Y}_{N}(\kappa_{h}(s))|^{2p}) ds \nonumber \\ & \;\;\;\; + \frac{1}{2}\mathbb{E}\sup\limits_{0\leq t\leq T}|Y^{i}_{N}(t)|^{2p} + C\mathbb{E}\int_{0}^{T}|Y^{i}_{N}(\kappa_{h}(s)) - \bar{Y}_{N}(\kappa_{h}(s))|^{2p}ds. \label{cbo_neweq_4.47}
\end{align}
Using Jensen's inequality and (\ref{y4.2}), we have
\begin{align}
    |\bar{Y}_{N}(\kappa_{h}(s))|^{2} &\leq \sum\limits_{i=1}^{N}|Y^{i}_{N}(\kappa_{h}(s))|^{2}\frac{e^{-\alpha f(Y^{i}_{N}(\kappa_{h}(s)))}}{\sum_{j=1}^{N}e^{-\alpha f(Y^{j}_{N}(\kappa_{h}(s)))}} \leq L_{1} + \frac{L_{2}}{N}\sum\limits_{i=1}^{N}|Y^{i}_{N}(\kappa_{h}(s))|^{2}. \label{cbo_neweq_4.48}
\end{align}
Therefore, substituting (\ref{cbo_neweq_4.48}) in (\ref{cbo_neweq_4.47}) yields
\begin{align*}
    &\mathbb{E}\sup\limits_{0\leq t\leq T}|Y^{i}_{N}(t)|^{2p} \leq 2\mathbb{E}|Y_{N}^{i}(0)|^{2p} + C +  C\mathbb{E}\int_{0}^{T}\Big(|Y^{i}_{N}(s)|^{2p} + |Y^{i}_{N}(\kappa_{h}(s))|^{p} +  \frac{1}{N}\sum\limits_{i=1}^{N}|Y_{N}^{i}(\kappa_{h}(s))|^{2p}\Big)ds \\ & \leq 2\mathbb{E}|Y_{N}^{i}(0)|^{2p} +C+ C\int_{0}^{T}\Big(\mathbb{E}\sup_{0\leq u\leq s}|Y^{i}_{N}(u)|^{2p}  + \frac{1}{N}\sum\limits_{i=1}^{N}\mathbb{E}\sup_{0\leq u\leq s} |Y_{N}^{i}(u)|^{2p}\Big)ds,
\end{align*}
where $C>0$ is independent of $h$ and $N$. 
Taking supremum over $ i =1,\dots, N$, we get
\begin{align*}
     \sup\limits_{i=1,\dots,N}\mathbb{E}\sup\limits_{0\leq t\leq T}|Y^{i}_{N}(t)|^{2p}  \leq 2\mathbb{E}|Y^{i}_{N}(0)|^{2p}+ C + C\int_{0}^{T}\sup_{i=1,\dots,N}\mathbb{E}\sup_{0\leq u\leq s}|Y^{i}_{N}(u)|^{2p}ds,
\end{align*}
where $C>0$ is independent of $h$ and $N$.
Using Gr\"{o}nwall's lemma, we have the desired result.
\end{proof}

\begin{lemma}\label{cbo_lem4.7}
Let Assumptions~\ref{cboh3.1}, \ref{cboh3.2}-\ref{cboasm1.4} hold. Let  $\sup_{i=1,\dots,N}\mathbb{E}|X^{i}_{N}(0)|^{4} < \infty$, $ \sup_{i=1,\dots,N} \mathbb{E}|Y^{i}_{N}(0)|^{4} < \infty$, $\mathbb{E}|Z|^{4} < \infty$. Then
\begin{align*}
   \sup_{i=1,\dots,N} \mathbb{E}|Y^{i}_{N}(t) - Y^{i}_{N}(\kappa_{h}(t))|^{2}  \leq Ch,
\end{align*}
where $C$ is a positive constant independent of $N$ and $h$.
\end{lemma}
\begin{proof}
We have
\begin{align*}
    |Y^{i}_{N}(t) &- Y^{i}_{N}(\kappa_{h}(t))|^{2} \leq C\bigg(\bigg|\int_{\kappa_{h}(t)}^{t}( Y^{i}_{N}(\kappa_{h}(s)) - \bar{Y}_{N}(\kappa_{h}(s)))ds\bigg|^{2} \\ & \;\;\;\; +    \bigg| \int_{\kappa_{h}(t)}^{t}\diag ( Y^{i}_{N}(\kappa_{h}(s)) - \bar{Y}_{N}(\kappa_{h}(s)))dW^{i}(s)\bigg|^{2}\\ & \;\;\;\; + \bigg|\int_{\kappa_{h}(t)}^{t}\int_{\mathbb{R}^{d}}\diag( Y^{i}_{N}(\kappa_{h}(s)) - \bar{Y}_{N}(\kappa_{h}(s)))z^{}\mathcal{N}^{i}(ds,dz^{})\bigg|^{2}\bigg),
\end{align*}
where $C$ is independent of $h$ and $N$.
Taking expectation and using Ito's isometry (note that we can apply Ito's isometry due to Lemma~\ref{cbo_lem4.6}), we get
\begin{align*}
    \mathbb{E}|Y^{i}_{N}(t) &- Y^{i}_{N}(\kappa_{h}(t))|^{2} \leq C(1+\mathbb{E}|Z|^{2})\bigg(\int_{\kappa_{h}(t)}^{t}\mathbb{E}| Y^{i}_{N}(\kappa_{h}(s)) - \bar{Y}_{N}(\kappa_{h}(s))|^{2}ds\bigg).
\end{align*}
Therefore, use of (\ref{cbo_neweq_4.48}) gives  
\begin{align*}
 \sup_{i=1,\dots,N}\mathbb{E}&|Y^{i}_{N}(t) - Y^{i}_{N}(\kappa_{h}(t))|^{2} \leq C(1+\mathbb{E}|Z|^{2})\bigg(\int_{\kappa_{h}(t)}^{t} \sup_{i=1,\dots,N}\mathbb{E}|Y^{i}_{N}(\kappa_{h}(s))|^{2} \\ &\;\;\;\; + 2L_{1} + \frac{L_{2}}{N}\sum\limits_{i=1}^{N}\sup_{i=1,\dots,N}\big( \mathbb{E}|Y^{i}_{N}(\kappa_{h}(s))|^{2})ds\bigg).
 \end{align*}
 Using Lemma~\ref{cbolemma3.3} and Lemma~\ref{cbo_lem4.6}, we get
 \begin{align*}
   \sup_{i=1,\dots,N} \mathbb{E}|Y^{i}_{N}(t) - Y^{i}_{N}(\kappa_{h}(t))|^{2} \leq C(t -\kappa_{h}(t)) \leq Ch,  
 \end{align*}
 where $C$ is independent of $N$ and $h$.
\end{proof}

\begin{theorem}
Let Assumptions~\ref{cboh3.1}, \ref{cboh3.2}-\ref{cboasm1.4} hold. Let $\mathbb{E}|Z|^{4} < \infty$, $\sup_{i=1,\dots,N}\mathbb{E}|X^{i}_{N}(0)|^{4} < \infty$, $ \sup_{i=1,\dots,N} \mathbb{E}|Y^{i}_{N}(0)|^{4} < \infty$ and $Y^{i}_{N}(0) = X^{i}_{N}(0) $, $i=1,\dots, N$.  Then
\begin{align}
    \lim\limits_{h \rightarrow 0}\lim\limits_{N\rightarrow \infty}\sup_{i=1,\dots,N}\mathbb{E}|Y^{i}_{N}(t) - X^{i}_{N}(t)|^{2} = \lim\limits_{N \rightarrow \infty}\lim\limits_{h\rightarrow 0}\sup_{i=1,\dots,N}\mathbb{E}|Y^{i}_{N}(t) - X^{i}_{N}(t)|^{2}= 0,
\end{align}
for all $t \in [0,T]$.
\end{theorem}
\begin{proof}
Let
\begin{align*}
   \tau_{1.R} = \inf\Big\{ t\geq 0 \; ; \; \frac{1}{N}\sum\limits_{i=1}^{N}|X^{i}_{N}(t)|^{4} \geq R\Big\}&,\;\;\;\; 
   \tau_{3,R} = \inf\Big\{ t \geq 0\; ; \; \frac{1}{N}\sum\limits_{i=1}^{N}|Y^{i}_{N}(t)|^{4} \geq R \Big\}, \\ 
   \tau^{h}_{R} & = \tau_{1,R} \wedge \tau_{3,R},
\end{align*}
and
\begin{align*}
    \Omega_{3}(t) & = \{ \tau_{1,R} \leq t\} \cup \{ \tau_{3,R} \leq t\}, \;\;\;
    \Omega_{4}(t) = \Omega \backslash \Omega_{3}(t) =  \{ \tau_{1,R} \geq t\} \cap \{ \tau_{3,R} \geq  t\} .
\end{align*}
We have
\begin{align*}
    \mathbb{E}|Y^{i}_{N}(t) - X^{i}_{N}(t)|^{2} &=  \mathbb{E}\big(|Y^{i}_{N}(t) - X^{i}_{N}(t)|^{2}I_{\Omega_{3}(t)}\big) \nonumber 
        +  \mathbb{E}\big(|Y^{i}_{N}(t) - X^{i}_{N}(t)|^{2}I_{\Omega_{4}(t)}\big)\\ & =: E_{3}(t) + E_{4}(t).
\end{align*}
Let us first estimate the term $E_{3}(t)$. Using Cauchy-Bunyakowsky-Schwartz inequality, Chebyshev's inequality, Lemma~\ref{cbolemma3.3} and Lemma~\ref{cbo_lem4.6}, we get
\begin{align}
    \mathbb{E}\big(|Y^{i}_{N}(t) - X^{i}_{N}(t)|^{2}I_{\Omega_{3}(t)}\big) &\leq \big(\mathbb{E}|Y^{i}_{N}(t) - X^{i}_{N}(t)|^{4}\big)^{1/2}\big(\mathbb{E}I_{\Omega_{3}(t)}\big)^{1/2} \nonumber \\ &\leq C \bigg( \frac{1}{RN}\sum\limits_{i=1}^{N}\mathbb{E}\sup_{0\leq s\leq t}|Y^{i}_{N}(s)|^{4} + \frac{1}{RN}\sum\limits_{i=1}^{N}\mathbb{E}\sup_{0\leq s\leq t}|X^{i}_{N}(s)|^{4} \bigg) \leq \frac{C}{R},\label{cbo_neweq_4.49}
\end{align}
where $C$ is independent of $h$, $N$ and $R$.

Note that $ \mathbb{E}\big(|Y^{i}_{N}(t) - X^{i}_{N}(t)|^{2}I_{\Omega_{4}(t)}\big) \leq  \mathbb{E}|Y^{i}_{N}(t \wedge \tau^{h}_{R}) - X^{i}_{N}(t \wedge \tau^{h}_{R})|^{2} $. Using Ito's formula, we obtain
\begin{align*}
    &|Y^{i}_{N}(t \wedge \tau^{h}_{R}) - X^{i}_{N}(t \wedge \tau^{h}_{R})|^{2} = |Y^{i}_{N}(0) - X^{i}_{N}(0)|^{2} \\&  - 2\int_{0}^{t\wedge \tau^{h}_{R}} \beta(s)\big((Y^{i}_{N}(s) - X^{i}_{N}(s))\cdot (Y^{i}_{N}(\kappa_{h}(s)) - \bar{Y}_{N}(\kappa_{h}(s)) - X^{i}_{N}(s) + \bar{X}_{N}(s))\big)ds \\ & +2 \sqrt{2}\int_{0}^{t \wedge \tau^{h}_{R}}\sigma(s)\big((Y^{i}_{N}(s) - X^{i}_{N}(s))\cdot \diag(Y^{i}_{N}(\kappa_{h}(s)) - \bar{Y}_{N}(\kappa_{h}(s)) - X^{i}_{N}(s) + \bar{X}_{N}(s))dW^{i}(s)\big)\\ & +
    2\int_{0}^{t\wedge \tau^{h}_{R}}\sigma^{2}(s)|Y^{i}_{N}(\kappa_{h}(s))- \bar{Y}_{N}(\kappa_{h}(s)) - X^{i}_{N}(s) + \bar{X}_{N}(s)|^{2} ds \\ &
    + \int_{0}^{t\wedge \tau^{h}_{R}}\int_{\mathbb{R}^{d}}\big(|Y^{i}_{N}(s^{-}) - X^{i}_{N}(s^{-}) + \diag(Y^{i}_{N}(\kappa_{h}(s)) - \bar{Y}_{N}(\kappa_{h}(s)))z - \diag(X^{i}_{N}(s) - \bar{X}_{N}(s))z|^{2} \\ & \;\;\;\;\;\;- |Y^{i}_{N}(s^{-}) - X^{i}_{N}(s^{-})|^{2}\big)\mathcal{N}^{i}(ds,dz).
\end{align*}
Taking expectation on both sides, and using the Cauchy-Bunyakowsky-Schwartz inequality, Young's inequality, Ito's isometry (note that we can apply Ito's isometry due to Lemma~\ref{cbo_lem4.6}) and Doob's optional stopping theorem \cite[Theorem 2.2.1]{cbos11}, we get
\begin{align}
    \mathbb{E}\big(|Y^{i}_{N}(t\wedge \tau^{h}_{R}) - X^{i}_{N}(t\wedge \tau^{h}_{R})|^{2}\big) &\leq  C h +C(1+|z|^{2})\mathbb{E}\int_{0}^{t\wedge \tau^{h}_{R}}\Big(|Y^{i}_{N}(\kappa_{h}(s)) - X^{i}_{N}(s)|^{2} \nonumber  \\ & \;\;\;\; \;\;\;\;\;\;\;+ |\bar{Y}_{N}(\kappa_{h}(s)) - \bar{X}_{N}(s)|^{2}\Big)ds \nonumber \\ &  \leq C\mathbb{E}\int_{0}^{t\wedge \tau^{h}_{R}} \Big(| Y^{i}_{N}(\kappa_{h}(s)) - Y^{i}_{N}(s)|^{2}+|Y^{i}_{N}(s) - X^{i}_{N}(s)|^{2} \nonumber  \\ & \;\;\;\; \;\;\;\;\;\;\;+ |  \bar{Y}_{N}(\kappa_{h}(s))-\bar{Y}_{N}(s)|^{2}   + |\bar{Y}_{N}(s) - \bar{X}_{N}(s)|^{2}\Big) ds. \label{cbo_eq4.30} 
\end{align}
Due to Lemma~\ref{cbo_lem4.7}, we have
\begin{align}
    \sup_{i=1,\dots,N}\mathbb{E}|Y^{i}_{N}(\kappa_{h}(s)) - Y^{i}_{N}(s )|^{2} \leq Ch, \label{cbo_eq4.31}
\end{align}
where $C$ is independent of $h$ and $N$.

Now, we will estimate the term $|\bar{Y}_{N}(s) - \bar{Y}_{N}(\kappa_{h}(s))|  $. Recall that we used discrete Jensen's inequality, Assumptions~\ref{cboh3.2}-\ref{cboassu3.4} and Cauchy-Bunyakowsky-Schwartz inequality to obtain (\ref{cbo_neweq_4.28}). We apply the same set of arguments as before to get 
\begin{align*}
|\bar{Y}_{N}(s)& - \bar{Y}_{N}(\kappa_{h}(s))| = \bigg|\sum\limits_{i=1}^{N}Y^{i}_{N}(s)\frac{e^{-\alpha f(Y^{i}_{N}(s))}}{\sum_{j=1}^{N}e^{-\alpha f(Y^{j}_{N}(s))}} -\sum\limits_{i=1}^{N}Y^{i}_{N}(\kappa_{h}(s))\frac{e^{-\alpha f(Y^{i}_{N}(\kappa_{h}(s)))}}{\sum_{j=1}^{N}e^{-\alpha f(Y^{j}_{N}(\kappa_{h}(s)))}}\bigg|
\\ & \leq \frac{1}{\frac{1}{N}\sum_{j=1}^{N}e^{-\alpha f(Y^{j}_{N}(s))}}\bigg|\frac{1}{N}\sum\limits_{i=1}^{N} \big(Y^{i}_{N}(s) - Y^{i}_{N}(\kappa_{h}(s))\big)e^{-\alpha f(Y^{i}_{N}(s))}\bigg| \\ & \;\;\;\; 
 +\frac{1}{\frac{1}{N}\sum_{j=1}^{N}e^{-\alpha f(Y^{j}_{N}(s))}}\bigg|\frac{1}{N}\sum\limits_{i=1}^{N}Y^{i}_{N}(\kappa_{h}(s))\Big(e^{-\alpha f(Y^{i}_{N}(s))} - e^{-\alpha f(Y^{i}_{N}(\kappa_{h}(s)))}\Big)\bigg|\\ & \;\;\;\;+ \bigg|\frac{1}{N}\sum\limits_{i=1}^{N}Y^{i}_{N}(\kappa_{h}(s))e^{-\alpha f(Y^{i}_{N}(\kappa_{h}(s)))}\bigg(\frac{1}{\frac{1}{N}\sum_{j=1}^{N}e^{-\alpha f(Y^{j}_{N}(s))}} - \frac{1}{\frac{1}{N}\sum_{j=1}^{N}e^{-\alpha f(Y^{j}_{N}(\kappa_{h}(s)))}}\bigg)\bigg| 
 \\ & \leq C\Bigg(e^{\frac{\alpha K_{u}}{N}\sum_{j=1}^{N}|Y^{j}_{N}(s)|^{2}}\frac{1}{N}\sum\limits_{i=1}^{N}|Y^{i}_{N}(s) - Y^{i}_{N}(\kappa_{h}(s))| \\ & \;\;\;\; +e^{\frac{\alpha K_{u}}{N}\sum_{j=1}^{N}(|Y^{j}_{N}(s)|^{2} + |Y^{j}_{N}(\kappa_{h}(s))|^{2})} \times\bigg(\frac{1}{N}\sum\limits_{i=1}^{N}(1+ |Y^{i}_{N}(s)|^{2} + |Y^{i}_{N}(\kappa_{h}(s))|^{2})^{2}\bigg)^{1/2}\\ & \;\;\;\;\; \times\bigg(\frac{1}{N}\sum\limits_{i=1}^{N}|Y^{i}_{N}(s) - Y^{i}_{N}(\kappa_{h}(s))|^{2}\bigg)^{1/2}\Bigg), 
 \end{align*}
 where $C > 0$ is independent of $h$ and $N$.
Squaring both sides, we ascertain
\begin{align}
&|\bar{Y}_{N}(s) - \bar{Y}_{N}(\kappa_{h}(s))|^{2} \leq C\Bigg( e^{\frac{2\alpha K_{u}}{N}\sum_{j=1}^{N}|Y^{j}_{N}(s)|^{2}}\frac{1}{N}\sum\limits_{i=1}^{N}|Y^{i}_{N}(s) - Y^{i}_{N}(\kappa_{h}(s))|^{2} \nonumber \\ & \;\;\;\; +e^{\frac{2\alpha K_{u}}{N}\sum_{j=1}^{N}(|Y^{j}_{N}(s)|^{2} + |Y^{j}_{N}(\kappa_{h}(s))|^{2})} \times\bigg(\frac{1}{N}\sum\limits_{i=1}^{N}(1+ |Y^{i}_{N}(s)|^{2} + |Y^{i}_{N}(\kappa_{h}(s))|^{2})^{2}\bigg) \nonumber \\ & \;\;\;\;\; \times\bigg(\frac{1}{N}\sum\limits_{i=1}^{N}|Y^{i}_{N}(s) - Y^{i}_{N}(\kappa_{h}(s))|^{2}\bigg)\Bigg). \label{cbo_eq4.32}
\end{align}
In the similar manner, we can obtain the following bound:
\begin{align}
    &|\bar{X}_{N}(s) - \bar{Y}_{N}(s)|^{2} \leq C\Bigg(e^{ \frac{2\alpha K_{u}}{N}\sum_{j=1}^{N}|X^{j}_{N}(s)|^{2}}\frac{1}{N}\sum\limits_{i=1}^{N}|X^{i}_{N}(s) - Y^{i}_{N}(s)|^{2} \nonumber \\ & \;\;\;\; +e^{\frac{2\alpha K_{u}}{N}\sum_{j=1}^{N}(|X^{j}_{N}(s)|^{2} + |Y^{j}_{N}(s)|^{2})} \times\bigg(\frac{1}{N}\sum\limits_{i=1}^{N}(1+ |X^{i}_{N}(s)|^{2} + |Y^{i}_{N}(s)|^{2})^{2}\bigg) \nonumber \\ & \;\;\;\;\; \times\bigg(\frac{1}{N}\sum\limits_{i=1}^{N}|X^{i}_{N}(s) - Y^{i}_{N}(s)|^{2}\bigg) \Bigg), \label{cbo_eq4.33}
\end{align}
where $C>0$ is independent of $h$ and $N$.
We substitute (\ref{cbo_eq4.31}), (\ref{cbo_eq4.32}) and (\ref{cbo_eq4.33}) in (\ref{cbo_eq4.30}) to get
\begin{align*}
&\mathbb{E}\big(|Y^{i}_{N}(t\wedge \tau^{h}_{R}) - X^{i}_{N}(t\wedge \tau^{h}_{R})|^{2}\big)     \leq C\mathbb{E}\int_{0}^{t\wedge \tau_{R}^{h}}\big(|X^{i}_{N}(s) - Y^{i}_{N}(s)|^{2}\big)ds + Ch \\ &  
\;\; +    CRe^{4\alpha K_{u}\sqrt{R}}\bigg( \mathbb{E}\int_{0}^{t\wedge \tau_{R}^{h}}\frac{1}{N}\sum\limits_{i=1}^{N}\big(|Y^{i}_{N}(s) - Y^{i}_{N}(\kappa_{h}(s))|^{2}\big) ds +  \mathbb{E}\int_{0}^{t\wedge \tau_{R}^{h}}\frac{1}{N}\sum\limits_{i=1}^{N}\big(|X^{i}_{N}(s) - Y^{i}_{N}(s)|^{2}\big) ds \bigg) \\ &  \leq C\int_{0}^{t}\mathbb{E}\big(|X^{i}_{N}(s\wedge \tau_{R}^{h}) - Y^{i}_{N}(s\wedge \tau_{R}^{h})|^{2}\big)ds + Ch  +
    CRe^{4\alpha K_{u}\sqrt{R}}\int_{0}^{t}\frac{1}{N}\sum\limits_{i=1}^{N}\mathbb{E}\big(|Y^{i}_{N}(s) - Y^{i}_{N}(\kappa_{h}(s))|^{2}\big) ds \\ & \;\; +  CRe^{4\alpha K_{u}\sqrt{R}}\int_{0}^{t}\frac{1}{N}\sum\limits_{i=1}^{N}\mathbb{E}\big(|X^{i}_{N}(s\wedge \tau_{R}^{h}) - Y^{i}_{N}(s\wedge \tau_{R}^{h})|^{2}\big) ds,
\end{align*}
where $C>0$ is independent of $h$, $N$ and $R$. Taking supremum over $i=1,\dots,N$ and using Lemma~\ref{cbo_lem4.7}, we obtain
\begin{align*}
     \sup_{i=1,\dots,N}\mathbb{E}\big(|Y^{i}_{N}(t\wedge \tau^{h}_{R}) &- X^{i}_{N}(t\wedge \tau^{h}_{R})|^{2}\big)  \leq CRe^{4\alpha K_{u}\sqrt{R}}h \\ &  + CRe^{4\alpha K_{u}\sqrt{R}}\int_{0}^{t}\sup_{i=1,\dots,N}\mathbb{E}\big(|Y^{i}_{N}(s\wedge \tau^{h}_{R}) - X^{i}_{N}(s\wedge \tau^{h}_{R})|^{2}\big)ds \bigg),
\end{align*}
where $C$ is independent of $h$, $N$ and $R$. 
Using Gr\"{o}nwall's lemma, we get
\begin{align*}
 \sup_{i=1,\dots,N}\mathbb{E}\big(|Y^{i}_{N}(t\wedge \tau^{h}_{R}) &- X^{i}_{N}(t\wedge \tau^{h}_{R})|^{2}\big) \leq CRe^{4\alpha K_{u}\sqrt{R}}e^{CRe^{4\alpha K_{u}\sqrt{R}}}h \leq Ce^{e^{C_{u}\sqrt{R}}}h,
 \end{align*}
 where $C>0$ and $C_{u}>0$ are constants independent of $h$, $N$ and $R$. 
 
 We choose $R= \frac{1}{C_{u}^{2}}(\ln{(\ln{(h^{-1/2})})})^{2}$. Consequently, we have
 \begin{align*}
     \sup_{i=1,\dots,N}\mathbb{E}\big(|Y^{i}_{N}(t) - X^{i}_{N}(t)|^{2}I_{\Omega_{4}(t)}\big)\leq  \sup_{i=1,\dots,N}\mathbb{E}\big(|Y^{i}_{N}(t\wedge \tau^{h}_{R}) &- X^{i}_{N}(t\wedge \tau^{h}_{R})|^{2}\big) \leq Ch^{1/2}, 
 \end{align*}
 where $C>0$ is independent of $h$ and $N$. This implies 
 \begin{align}
    \lim\limits_{h \rightarrow 0}\lim\limits_{N \rightarrow \infty} \sup_{i=1,\dots,N}\mathbb{E}\big(|Y^{i}_{N}(t) - X^{i}_{N}(t)|^{2}I_{\Omega_{4}(t)}\big) =  \lim\limits_{N \rightarrow \infty}\lim\limits_{h \rightarrow 0} \sup_{i=1,\dots,N}\mathbb{E}\big(|Y^{i}_{N}(t) - X^{i}_{N}(t)|^{2}I_{\Omega_{4}(t)}\big) = 0. \label{cbo_neweq_4.54}
 \end{align}
 
The term (\ref{cbo_neweq_4.49}) and the choice of $R$ provide the following estimate:
\begin{align*}
    \sup_{i=1,\dots,N}\mathbb{E}\big(|Y^{i}_{N}(t) - X^{i}_{N}(t)|^{2}I_{\Omega_{3}(t)}\big) \leq \frac{C}{(\ln{(\ln{(h^{-1/2})})})^{2}},
\end{align*}
where $C$ is independent of $h$ and $N$. This gives
\begin{align}
     \lim\limits_{h \rightarrow 0}\lim\limits_{N \rightarrow \infty} \sup_{i=1,\dots,N}\mathbb{E}\big(|Y^{i}_{N}(t) - X^{i}_{N}(t)|^{2}I_{\Omega_{3}(t)}\big) =  \lim\limits_{N \rightarrow \infty}\lim\limits_{h \rightarrow 0}\sup_{i=1,\dots,N}\mathbb{E}\big(|Y^{i}_{N}(t) - X^{i}_{N}(t)|^{2}I_{\Omega_{3}(t)}\big) = 0. \label{cbo_neweq_4.55}
\end{align}
As a consequence of (\ref{cbo_neweq_4.54}) and (\ref{cbo_neweq_4.55}), we get
 \begin{align*}
     \lim\limits_{h\rightarrow 0}\lim\limits_{N \rightarrow \infty}\sup_{i=1,\dots,N}\mathbb{E}\big(|Y^{i}_{N}(t) - X^{i}_{N}(t)|^{2}\big)= \lim\limits_{N \rightarrow \infty}\lim\limits_{h \rightarrow 0}\sup_{i=1,\dots,N}\mathbb{E}\big(|Y^{i}_{N}(t) - X^{i}_{N}(t)|^{2}\big) = 0. 
 \end{align*}
\end{proof}



\section{Numerical Examples}\label{cbo_num_exp}

In this section, we conduct  numerical experiments on the Rastrigin and Rosenbrock functions by implementing the models (\ref{cbos1.5}), (\ref{cbos1.6}), (\ref{cbo_neweq_2.17}) and model with common noise introduced in \cite{cbo4, cbo5}. We use the Euler scheme for implementation with $h= 0.01$. We run $100 $ simulations and quote the success rates. We call a run of $N$ particles a success if $|\bar{Y}_{N}(T) - x_{\min}| \leq 0.25$. Defining success rate in this manner is consistent with earlier CBO papers. 
\begin{experiment}
We\; perform\; the\; experiment\; with\; the\; CBO\; model (\ref{cbos1.5}),\; JumpCBO model (\ref{cbos1.6}), JumpCBOwCPN model (jump-diffuison CBO model with common Poisson noise from (\ref{cbo_neweq_2.17})),  CBOwCWN model (CBO model with common Wiener noise of \cite{cbo4,cbo5}) for the Rastrigin  function
\begin{equation}
f(x)= 10 + \sum_{i=1}^{d}\big((x_i - B)^2 - 10\cos(2\pi (x_i - B))\big)/d,
\end{equation} 
where we take $d= 20$. The minimum is located at $(0,\dots,0)\in \mathbb{R}^{20}$. In this experiment for the Rastrigin function, the initial search space is $[-6,6]^{20}$ and final time, $T= 100$. We take  $\beta = 1$, $\sigma = 5.1 $ for CBO, CBOwCWN, JumpCBO and JumpCBOwCPN models. We take $\gamma(t) = 1 $ when $t \leq 20$ and $\gamma(t) = e^{1-t/20}$ when $t > 20$ for JumpCBO and JumpCBOwCPN models.   Also,  $\Zstroke$ is distributed as standard Gaussian random variable and we choose jump intensity, $\lambda$, of Poisson process  equal to $20$.  
\end{experiment}


\captionof{table}{Success rate for $\alpha =20$}\label{exptab8.1}
\begin{center}
\begingroup
\setlength{\tabcolsep}{2.9pt}
\begin{tabular}{c c c c c } 
\hline\hline 
 $N$ & CBO & \shortstack{CBOwCWN}  & JumpCBO & \shortstack{JumpCBOwCPN}\\  
 [0.5ex] 
\hline 
 20 & 53   & 1 & 61 & 65 \\
  50 & 62  & 0 & 69 & 72\\
 80 & 22 & 2 & 41 & 40 \\
 100 &  1 & 2 & 29 & 25 \\
  [1ex]
\hline 
\end{tabular}
\endgroup
\end{center}
 \hfill%
    \captionof{table}{Success rate for $\alpha = 30 $ }\label{exptab8.2}
\begin{center}
\begingroup
\setlength{\tabcolsep}{2.9pt}
\begin{tabular}{c c c c c } 
\hline\hline 
 $N$ &  CBO & \shortstack{CBOwCWN} & JumpCBO & \shortstack{JumpCBOwCPN} \\  
 [0.5ex] 
\hline 
 20 & 87 & 0 & 90 & 94\\
 50 & 99 & 0  & 100 & 100 \\
 80 & 100 & 0 & 100 & 100 \\
 100 & 100  & 0 &  100 & 100 \\
  [1ex]
\hline 
\end{tabular}
\endgroup
\end{center}




 In the case of Rastrigin function, the performance of JumpCBO model (\ref{cbos1.6}), JumpCBOwCPN model (\ref{cbo_neweq_2.17}) and CBO model (\ref{cbos1.5}) is comparable. However, CBOwCWN of \cite{cbo4,cbo5} does not perform well. As the alpha is increased from $20$ to $30$, the success rates are fairly improved. We have taken constant $\beta$ and $\sigma$, and decaying $\gamma$ for the jump-diffusion CBO models. As one can see, jumps have impacted the performance positively in CBO when $\alpha  = 20$. Another fact to be noticed is that performance of the jump-diffusion models with common or independent Poisson processes is very similar. It is also clear from the experiment that CBOwCWN model of \cite{cbo4,cbo5} does not induce enough noise in the dynamics of the particle system  sufficient for effective space exploration.  

\begin{experiment}
We perform the experiment with the CBO model (\ref{cbos1.5}), JumpCBO model (\ref{cbos1.6}), CBOwCN model (CBO model with common noise of \cite{cbo4,cbo5}) for the Rosenbrock function
\begin{equation}
    \sum\limits_{i=1}^{d-1}[100(x_{i+1} - x_{i}^{2})^{2} + (x_{i} - 1)^{2}]/d,
\end{equation}
where we take $d= 5$. The minimum is located at $(1,\dots,1)\in \mathbb{R}^{5}$. In this experiment for the Rosenbrock function, the initial search space is $[-1,3]^5$ and final time, $T= 120$. We take  $\beta = 1$, $\sigma = 5 $ for CBO as well as CBOwCN models. We take $\beta(t) = 2- e^{-t/100}$, $\sigma(t) = 4 + e^{-t/90}$ and $\gamma(t) = 1 $ for $t \leq 90$ and $\gamma(t) = e^{1-t/90}$ for $t > 90$. Note that $\beta(0) = 1$ and $\sigma(0) = 5$ which are same as parameters $\beta$ and $\sigma$ for the CBO  and CBOwCN models.  Also,  $\Zstroke$ is distributed as standard Gaussian random variable and we choose jump intensity, $\lambda$, of Poisson process  equal to $90$.  
\end{experiment}



    \captionof{table}{Success rate for $\alpha = 20 $ }\label{exptab8.3}
\begin{center}
\begingroup
\setlength{\tabcolsep}{2.9pt}
\begin{tabular}{c c c c c } 
\hline\hline 
 $N$ &  CBO & CBOwCWN & JumpCBO &JumpCBOwCPN  \\  
 [0.5ex] 
\hline 
 20 & 2 & 1 & 35 & 37 \\
 50 & 3 & 1 & 75 & 76 \\
 80 & 3 & 0 & 96 & 89 \\
 100 & 4  & 4 & 85 &  94  \\
  [1ex]
\hline 
\end{tabular}
\endgroup
\end{center}

    \captionof{table}{Success rate for $\alpha = 30 $ }\label{exptab8.4}
\begin{center}
\setlength{\tabcolsep}{2.9pt}
\begin{tabular}{c c c c c } 
\hline\hline 
 $N$ &  CBO & CBOwCWN & JumpCBO & JumpCBOwCPN \\  
 [0.5ex] 
\hline 
 20 & 6 & 2 & 20 & 25 \\
 50 & 3 & 0 & 49  & 45 \\
 80 & 5 & 2 & 69 & 64 \\
 100 & 4  & 1 & 74 &  70  \\
  [1ex]
\hline 
\medskip
\end{tabular}
\end{center}

In the case of Rosenbrock function, there is a significant improvement in finding global minimum when using the jump-diffusion models (\ref{cbos1.6}) and (\ref{cbo_neweq_2.17}) in comparison with (\ref{cbos1.5}) and CBOwCWN of \cite{cbo4,cbo5}. As is the case with the Rastrigin funciton, for the Rosenbrock funciton, both jump-diffusion models have similar performance. We note that the Rosenbrock function has quartic growth. We take time-dependent $\beta(t)$, $\sigma(t) $ and $\gamma(t)$ for the jump diffusion models so that $\beta(t)$ is increasing function, $\sigma(t)$ is a decreasing function, and $\gamma(t)$ is constant for some period of time and then starts decreasing exponentially. This experiment illustrates a good balance of \emph{exploration} and \emph{exploitation}
 delivered by the proposed jump-diffusion models. The particles explore the space until $t= 90$ and after that particles start exploiting the searched space.    

\section{Concluding remarks}
We have developed a new CBO algorithm with jump-diffusion SDEs, for which we have studied its well-posedness both at the particle level and its mean-field approximation. The key feature of the jump-diffusion CBO is a more effective energy landscape exploration driven by the randomness introduced by both Wiener and Poisson processes. In practice, this translates into better success rates in finding the global minimizer, and a more robust initialization, which can be located far away from the global minimizer.
A natural extension of the current work is a systematic study of CBO with constraints in the search space as recently discussed in \cite{cbo50,cbo44,cbo45,cbo46}. This is particularly challenging because of the need to accurately treat boundary conditions for the SDEs (see e.g. \cite{cbo37}).  Another interesting research direction is the exploration of jump-diffusion processes in the framework of kinetic-type CBO models \cite{cbo47,cbo48}.

\section*{Acknowledgements}
AS was supported by EPSRC grant no. EP/W52251X/1. DK
was supported by EPSRC grants EP/T024429/1 and
EP/V04771X/1. For the purpose of open access, the authors have applied a Creative Commons Attribution (CC-BY) licence to any Author Accepted Manuscript version arising.
 

\bibliographystyle{alpha}
\bibliography{references}

\end{document}